\documentclass[11pt,preprint]{article}

\usepackage{geometry}
\usepackage{graphicx}
\usepackage{mathtools}
\usepackage{psfrag}
\usepackage{amssymb}
\usepackage{amsmath}
\usepackage{authblk}
\usepackage{amsthm}
\usepackage{hyperref}

\usepackage{xspace}

\usepackage{enumitem}
\setitemize{leftmargin=4em,label=\scriptsize{$\blacksquare$}}
\setenumerate{leftmargin=4em,label=(\alph*), topsep=0.25em}

\newtheorem{theorem}{Theorem}[section]
\newtheorem{lemma}[theorem]{Lemma}
\newtheorem{proposition}[theorem]{Proposition}

\newtheorem{assumption}[theorem]{Assumption}

\numberwithin{theorem}{section}
\numberwithin{figure}{section}
\numberwithin{equation}{section}

\newcommand{\qq}{q}
\newcommand{\ph}{\varphi}

\newcommand{\C}{\mathbb C}
\newcommand{\R}{\mathbb R}

\newcommand{\N}{\mathbb N}
\newcommand{\sph}{\mathbb S}

\newcommand{\fnum}{{\tt f}}
\newcommand{\gnum}{{\tt g}}

\newcommand{\Bnum}{\boldsymbol{\tt{B}}}
\newcommand{\Mnum}{\boldsymbol{\tt{M}}}

\newcommand{\rmd}{\mathrm d}
\newcommand{\eps}{\epsilon}

\newcommand{\la}{\lambda}

\newcommand\abs[1]{\left\vert#1\right\vert}
\newcommand\sabs[1]{\lvert#1\rvert}
\newcommand\norm[1]{\left\Vert#1\right\Vert}
\newcommand\snorm[1]{\Vert#1\Vert}

\newcommand\set[1]{\left\{#1\right\}}
\newcommand{\kl}[1]{\left(#1\right)}

\newcommand{\skl}[1]{(#1)}
\newcommand{\sset}[1]{\{#1\}}

\newcommand\sinner[2]{\langle#1,#2\rangle}
\newcommand\rest[2]{{#1}\vert_{#2}}
\newcommand\sfrac[2]{{#1}/{#2}}
\newcommand\ii[1]{[#1]}
\newcommand\rp{\alpha}

\makeatletter
\newcommand*\bigcdot{\mathpalette\bigcdot@{.6}}
\newcommand*\bigcdot@[2]{\mathbin{\vcenter{\hbox{\scalebox{#2}{$\m@th#1\bullet$}}}}}
\makeatother

\newcommand{\signal}{\boldsymbol  x}
\newcommand{\dsignal}{\boldsymbol  d}
\newcommand{\zsignal}{\boldsymbol  z}
\newcommand{\res}{\boldsymbol  r}
\newcommand{\asignal}{\boldsymbol \xi }
\newcommand{\fsignal}{\boldsymbol f }
\newcommand{\data}{\boldsymbol y}
\newcommand{\gdata}{\boldsymbol g}
\newcommand{\X}{\mathbb{X}}
\newcommand{\Y}{\mathbb{Y}}

\DeclareMathOperator*{\dom}{\mathcal{D}}
\newcommand{\To}{\mathbf{F}}

\newcommand{\Mo}{\mathbf{M}}

\def\Plus{\texttt{+}}

\allowdisplaybreaks

\title{The Averaged Kaczmarz Iteration for Solving Inverse Problems}

\author{Housen Li} 
\affil{National University of Defense Technology\\
137 Yanwachi street, 410073 Changsha, China\\
Email: {\tt housen.li@outlook.com}} 

\author{Markus Haltmeier} 
\affil{Department of Mathematics, University of Innsbruck\\
Technikestra{\ss}e 13, A-6020 Innsbruck, Austria\\
Email: {\tt markus.haltmeier@uibk.ac.at}}

\date{January 3, 2017}

\begin{document}
\maketitle

\begin{abstract}
We introduce a new iterative regularization method for solving inverse problems that  can be
written as systems of linear or non-linear  equations in Hilbert spaces.
The proposed averaged Kaczmarz (AVEK)  method  can be seen as a hybrid method  between the Landweber
and the Kaczmarz method. As the Kaczmarz method, the  proposed method only requires evaluation
of one direct and one adjoint sub-problem per iterative update.
On the other, similar to the Landweber iteration,  it uses an average over previous auxiliary  iterates   which
increases stability. We present a convergence analysis of the AVEK iteration. Further, detailed numerical
studies are presented for a tomographic image reconstruction problem, namely the limited data problem in
photoacoustic tomography.  Thereby, the AVEK  is compared with other iterative regularization methods
including standard Landweber and Kaczmarz iterations, as well as recently proposed  accelerated versions 
based on error minimizing relaxation strategies.


\end{abstract}

\section{Introduction}
\label{sec:intro}

In this paper, we study the stable solution of linear or non-linear systems of operator
equations of the form
\begin{equation}\label{eq:ip}
	\To_i(\signal) =  \data_i  \quad \text{ for } i=1, \dots, n \,.
\end{equation}
Here $\To_i \colon \dom(\To_i)  \subseteq \X \to \Y_i$  are possibly nonlinear operators
between Hilbert spaces $\X$ and $\Y_i$  with domains of definition $\dom(\To_i)$. We are in particular interested in the case that we only have approximate data
$\data_i^\delta \in \Y_i$ available, which satisfy an  estimate  of the form  $\snorm{\data_i^\delta  -  \data_i} \leq \delta_i$
for some  noise levels $\delta_i >0$.   Moreover, we focus on the ill-posed (or ill-conditioned)
case, where standard solution methods  for~\eqref{eq:ip} are sensitive  to perturbations.
Many inverse problems in biomedical imaging, geophysics or engineering sciences
can be written in such a form (see, for example,  \cite{engl1996regularization,natterer01mathematical,scherzer2009variational}.)
For its solution one has to use regularization methods,  which are based on approximating~\eqref{eq:ip} by neighboring but more stable problems.

There are at least two basic classes of solution approaches for inverse problems of the form~\eqref{eq:ip}, namely (generalized) Tikhonov regularization on the one and iterative regularization on the other hand. (Notice that there are methods sharing structures of both classes, for example iterated Tikhonov regularization~\cite{KiCh79} or Lardy's method~\cite{Lar75}.) These approaches are based
on rewriting~\eqref{eq:ip} as a single equation $\To(\signal) = \data$  with forward operator $\To =  (\To_i)_{i=1}^n$
and exact data $\data  = (\data_i)_{i=1}^n$. In Tikhonov regularization, one defines  approximate solutions as minimizers
of the Tikhonov functional $ \frac{1}{n} \sum_{i=1}^n \snorm{\To_i(\signal)- \data^\delta_i}^2 + \la \snorm{\signal-\signal_0}^2
$, which is the weighted combination of the  residual term $\sum_{i=1}^n \snorm{\To_i(\signal)- \data^\delta_i}^2$
that enforces  all equations to be approximately  satisfied, and the regularization term $\snorm{\signal - \signal_0}^2$ that stabilizes the inversion process; $\lambda > 0$ is usually referred to as the regularization parameter.
In iterative regularization methods,  stabilization  is achieved  via early stopping of
iterative  schemes.  For this class of methods, one develops special
iterative optimization techniques designed  for minimizing the un-regularized  residual term
$ \sum_{i=1}^n \snorm{\To_i(\signal)- \data^\delta_i}^2 $. The iteration index in this case plays the
role of the  regularization parameter which has to be carefully chosen depending on available information about the noise and the unknowns to be recovered.

In this paper we  introduce a new member of the class  of  iterative regularization methods, named  averaged Kaczmarz (AVEK) iteration. The  method combines  advantages of  two main iterative regularization techniques, namely  the Landweber  and the Kaczmarz iteration.

\subsection{Iterative regularization methods}

The most basic iterative method for solving inverse problems is the Landweber iteration \cite{engl1996regularization,hanke1995convergence,kaltenbacher2008iterative,landweber1951iteration},
which reads
\begin{equation} \label{eq:landweber}
	\forall k \in \N \colon \quad
     \signal^\delta_{k+1}  \coloneqq   \signal^\delta_k  -  \frac{s_k}{n} \sum_{i=1}^n \To_i'(\signal^\delta_k)^*\kl{ \To_i(\signal_k^\delta) - \data^\delta_i}  \,.
\end{equation}
Here $ \To_i'(\signal)^*$  is the  Hilbert space adjoint of  the derivative of $\To_i$,
$s_k$ is the step size and $\signal_1^\delta$ the initial guess.  The  Landweber iteration renders
a regularization method when stopped according to Morozov's discrepancy principle, which stops the iteration at the smallest index $k_\star$  such that
$\sum_{i=1}^n \snorm{\To_i (\signal^\delta_{k_\star}) - \data^\delta_i}^2 \leq n (\tau \delta)^2 $ for some constant $\tau>1$.
A convergence analysis of the non-linear Landweber iteration has first been derived  in~\cite{hanke1995convergence}. Among others, similar results have subsequently been established for the steepest-descent method \cite{neubauer1995convergence}, the preconditioned Landweber iteration \cite{egger2005preconditioning},  or Newton-type
methods~\cite{blaschke1997convergence,rieder1999regularization}.

Each iterative update in~\eqref{eq:landweber} can be  numerically  quite expensive, since it requires
solving forward and adjoint  problems for all of the $n$ equations
in \eqref{eq:ip}.   In  situations where $n$ is large and evaluating  the forward and adjoint  problems is
costly, methods like the Landweber-Kaczmarz iteration (see \cite{EHN17,haltmeier2007kaczmarz2,haltmeier2007kaczmarz1,KiLe14,kowar2002})
\begin{equation} \label{eq:kac}
\forall k \in \N \colon \quad
\signal^\delta_{k+1}  \coloneqq   \signal^\delta_k  -   s_k \rp_k   \To_{\ii{k}}'(\signal^\delta_k)^*\kl{ \To_{\ii{k}}(\signal_k^\delta) - \data^\delta_{\ii{k}}}  \,,
\end{equation}
where $\ii{k} \coloneqq  (k-1 \mod  n)+1$,
are often much faster. The acceleration comes from the fact that the update in~\eqref{eq:kac} only requires
the solution of one forward and one adjoint problem instead of solving several  of them, but  nevertheless often yields a comparable decrease per iteration of the reconstruction error.
 The additional parameters $\rp_k \in \set{0,1}$ effect that  in the noisy data case  some of the iterative updates are skipped  which  renders~\eqref{eq:kac} a regularization method.
 Such a skipping strategy has been introduced in~\cite{haltmeier2007kaczmarz1} for the Landweber-Kaczmarz
 iteration and later, among others, combined with  steepest  descent and Levenberg-Marquardt  type iterations~\cite{baumeister2010levenberg,decesaro2008steepest}.

Kaczmarz type methods often perform well in practice. However, unless  allowing asymptotically vanishing step sizes, even for well-posed problems, they do not  converge to a  single point.  This  can easily be seen in the case of two linear equations  in $\R$ without a  common solution where the Kaczmarz method with constant step size has different accumulation points~\cite[Section 2]{Luo91} (compare also~\cite{censor1983strong, Sol98}).  Opposed to that, the AVEK method  that we introduce in this paper can be shown to converge in such a situation. Still, one step in AVEK has computational costs similar to the Kaczmarz method (if evaluating  the forward operators and their adjoints are the computationally most expensive parts). Note that \eqref{eq:landweber} and~\eqref{eq:kac} might be called simultaneous and sequential, respectively~\cite{EHN17,jiang2003convergence,NAE17,nikazad2016convergence}. Further, instead of using the average in \eqref{eq:landweber} one might  also consider convex combinations of $\To_i'(\signal^\delta_k)^*\kl{ \To_i(\signal_k^\delta) - \data^\delta_i}$ to define  the iterative updates in simultaneous schemes (as is Cimmino's method~\cite{Cim38}).

\subsection{The averaged Kaczmarz (AVEK) iteration}

The general AVEK iteration is defined  by
\begin{align} \label{eq:avek1}
\signal_{k+1}^{\delta}& \coloneqq \sum_{\ell = k-n+1}^k \omega_{k-\ell+1} \asignal_\ell^{\delta} \qquad{\text{for } k \ge n}
\\  \label{eq:avek2}
 \asignal_\ell^\delta & \coloneqq \signal_\ell^{\delta} - s_\ell \rp_\ell \To_{\ii{\ell}}'(\signal^\delta_\ell)^*
\kl{ \To_{\ii{\ell}}(\signal^\delta_\ell) - \data^\delta_{\ii{\ell}}} \\
\label{eq:avek3}
\rp_\ell   &\coloneqq
\begin{cases}
      1  &\text{if } \snorm{   \To_{\ii{\ell}}(\signal^\delta_\ell) - \data^\delta_{\ii{\ell}} }
      \geq \tau_{\ii{\ell}} \delta_{\ii{\ell}} \\
      0  & \text{otherwise}
\end{cases} \,,
\end{align}
{where $\signal^\delta_1,\ldots,\signal^\delta_n$ are user-specified initial values, and $\omega_i \ge 0$ are fixed weights satisfying $\sum_{i = 1}^n \omega_i = 1$.}
Instead of discarding the previous computations, {the AVEK iteration remembers} the last  Kaczmarz type auxiliary iterates $\asignal_\ell^{\delta}$
and the update  $\signal_{k+1}^{\delta}$ is defined as the {weighted} average over them.
The  parameters $\rp_\ell $   effect that no update for $\asignal_\ell^{\delta}$ is performed if
$\snorm{   \To_{\ii{\ell}}(\signal^\delta_\ell ) - \data^\delta_{\ii{\ell }} }$  is sufficiently small;
$\tau_i \geq 0$ are control parameters. As the Kaczmarz  iteration, the AVEK iteration only requires evaluating a single gradient
$\To_{i}'(\signal )^*\skl{ \To_{i}(\signal) - \data^\delta_i}$ per iterative update which usually is the numerically most expensive part for evaluating \eqref{eq:avek1}-\eqref{eq:avek3}.  As the Landweber iteration~\eqref{eq:landweber}, {if every $\omega_i$ is positive}, each update in AVEK uses
information of all equations which enhances stability.
{Notice that the Landweber-Kaczmarz iteration in~\eqref{eq:kac} is a special case of the general AVEK iteration with $\omega_1 = 1$ and $\omega_i = 0$ for $i \ge 2$. }

Throughout this paper, we focus on the AVEK with equal weights
$\omega_1 = \cdots = \omega_n = 1/n$. In what follows, by AVEK we always refer to this special case unless explicitly stated. Note,  that the AVEK update \eqref{eq:avek1} in this case can alternatively be written as $\signal^\delta_{k+1} = \signal^\delta_{k} +  \skl{\asignal^\delta_k - \asignal^\delta_{k-n}}/n$. The modified update formula only requires two additions in the space $\X$ and therefore can numerically be more efficient than evaluating~\eqref{eq:avek1}. We further note that in the original form  \eqref{eq:avek1}-\eqref{eq:avek3}, AVEK requires storing $n$ auxiliary updates $ \asignal_\ell^\delta \in \X$. When storage is a limited aspect and $n$ is large, this might be problematic.
However,  the required storage can be reduced by saving the residuals  $\To_{\ii{\ell}}(\signal^\delta_\ell) - \data^\delta_{\ii{\ell}} \in \Y_i$ instead of the auxiliary updates. In the numerical  implementation $\prod_{i=1}^n \Y_i$ (after discretization) will typically have a similar  dimension as  $\X$ (after discretization).
As a consequence, storing all $n$ residuals only requires a storage similar to  saving a single iterate.

In this paper we establish a convergence analysis of \eqref{eq:avek1}-\eqref{eq:avek3} for exact  and
noisy data (see Section~\ref{sec:convergence}).
These results are most closely related  to the convergence analysis of other iterative regularization methods such as
the  Landweber and steepest {descent} methods \cite{hanke1995convergence,neubauer1995convergence}
and extensions to Kaczmarz type
iterations~\cite{decesaro2008steepest,haltmeier2007kaczmarz1,leitao16projective}. However, the AVEK iteration
is new and we are not aware of a convergence analysis for any similar iterative regularization method.
We point out,  that the AVEK shares some similarities with the  incremental gradient method of~\cite{blatt2007convergent} and the averaged stochastic gradient method of  \cite{schmidt2017minimizing} (both studied in finite dimensions). However, the iterations of \cite{blatt2007convergent,schmidt2017minimizing} are notably different from the AVEK method as they use an average  of gradients instead of an average of auxiliary iterates (cf.~Section~\ref{sec:conclusion}). {Given the large amount of publications on averaged incremental gradient and stochastic gradient methods over the last couple of years it seems surprising that these  methods have not been extended in the spirit of  AVEK so far. The present work might initiate future research in such  directions.}

\subsection{Outline}
\label{ssec:outline}

The rest of this paper is organized as follows. In Section~\ref{sec:convergence}
we present the convergence analysis of the AVEK method under typical assumptions for
iterative regularization methods. As main results we show weak convergence of AVEK
in the case of exact data (see Theorem~\ref{thm:exact}) and {(weak and strong) convergence}
as the noise level tends to zero (see Theorem~\ref{thm:noisy}).
The proof of an important auxiliary result (Lemma~\ref{lem:diff}) required
for the convergence analysis is presented in Appendix~\ref{app:deconv}.
In Section~\ref{sec:num}, we apply AVEK method to the limited view problem for the
circular Radon transform and present a numerical comparison with the
Landweber  and the Kaczmarz method. The paper concludes with a summary
presented in Section~\ref{sec:conclusion} and a discussion of open issues and
possible extensions of AVEK.

\section{Convergence analysis}
\label{sec:convergence}

In this section we establish  the convergence analysis of the AVEK method.
For that purpose we first fix the main assumptions
in Subsection~\ref{ssec:preliminaries} and derive the basic quasi-monotonicity
property of AVEK in Subsection~\ref{ssec:quasi}. The actual convergence analysis is presented in Subsections~\ref{ssec:exact}  and~\ref{ssec:noisy}.

\subsection{Preliminaries}
\label{ssec:preliminaries}

Throughout this paper $\To_i \colon \dom(\To_i)  \subseteq \X \to \Y_i$ are continuously  Fr\'echet
differentiable maps for $i \in \sset{1, \dots, n }$.
We consider the system~\eqref{eq:ip}, which can be written as a single equation $\To(\signal) = \data$ with forward
operator $\To =  (\To_i)_{i=1}^n$ and exact data $\data  = (\data_i)_{i=1}^n$ in $\Y \coloneqq \prod_{i=1}^n \Y_i$. Here $\data \in \Y$ are the exact data and $\data^\delta   = (\data_i^\delta)_{i=1}^n \in \Y$ denote noisy data satisfying $\snorm{\data_i - \data_i^\delta}\leq \delta_i$ with $\delta_i \geq 0$.

For the convergence analysis of the AVEK method established below we assume that
the following additional assumptions  are satisfied.

\begin{assumption}[Main\label{ass} conditions for the convergence analysis] \mbox{}
\begin{enumerate}[leftmargin=4em,label = (A\arabic*)]
\item \label{ass1}
There are $\signal_0 \in \X$, $\rho>0$ such that $B_\rho (\signal_0) \coloneqq \sset{\signal\mid\norm{\signal - \signal_0}\le \rho} \subseteq \bigcap_{i \in \sset{1, \dots , n}}
\dom(\To_i)$.

\item \label{ass1a}
For every $ i \in\sset{ 1,\ldots,n}$, it holds  $\sup\sset{ \snorm{\To_i'(\signal)} \mid \signal \in B_{\rho}(\signal_0)}< \infty$.

\item  \label{ass2}
For every $ i \in\sset{ 1,\ldots,n}$, there exists a constant $\eta_i \in [0, 1/2)$ such that
\begin{multline} \label{eq:tcc}
\forall \signal_1, \signal_2 \in B_{\rho}(\signal_0) \colon \quad  \snorm{ \To_i(\signal_1) - \To_i(\signal_2) -
\To_i'(\signal_1)( \signal_1 - \signal_2) }
\\ \leq
     \eta_i \snorm{ \To_i(\signal_1)-\To_i(\signal_2) } \,.
 \end{multline}
Equation \eqref{eq:tcc} is often referred to as local tangential cone condition.

\item \label{ass3}
For the exact data $\data \in \Y$, there exists a solution of \eqref{eq:ip} in $B_{\rho/3}(\signal_0) $.
\end{enumerate}
\end{assumption}

From Assumption~\ref{ass} it follows that~\eqref{eq:ip} has at least one $\signal_0$-minimum norm solution
denoted $\signal^\Plus \in \X$. Such a minimal norm solution  satisfies
\begin{equation*}
\label{eq:mns}
	\snorm{\signal^\Plus - \signal_0}
	=
	\inf\sset{ \snorm{\signal-\signal_0} \mid \signal \in B_{\rho}(\signal_0) \text{ and } \To(\signal) = y} \,.
\end{equation*}
The AVEK  iteration is defined by \eqref{eq:avek1}-\eqref{eq:avek3}.
There we always choose  the   initialisation  such that
$\signal^\delta_1, \ldots, \signal^\delta_n \in B_{\rho/3}(\signal_0)$ and assume
that $\tau_i > {2(1+\eta_i)}/\skl{1-2\eta_i}$.

\subsection{Quasi-monotonicity}
\label{ssec:quasi}

Opposed to the Landweber and the Kaczmarz method, for the
AVEK method the  reconstruction error  $\snorm{\signal^\delta_k - \signal^*}$, where $\signal^*$ is a solution of  \eqref{eq:ip}, is not strictly decreasing. However,  we can show the  following quasi-monotonicity
property which plays a central role in our convergence analysis.

\begin{proposition}[Quasi-monotonicity] \label{prop:mon}
Let $\signal^* \in B_{\rho}(\signal_0)$ be  any solution of~\eqref{eq:ip}. Suppose that $\signal^\delta_k$ is defined by~\eqref{eq:avek1}-\eqref{eq:avek3}, and that Assumption~\ref{ass} holds true.
Additionally, suppose that the step sizes $s_k$ are chosen in such a way that
\begin{equation}\label{eq:stepsz1}
s_k  \snorm{\To_i'(\signal)}^2 \le 1\quad\text{ for every }
i, k  \text{ and  } \signal \in B_{\rho}(\signal_0) \,.
\end{equation}
Then  for every  $k \geq n$ it holds that $\signal^\delta_k \in B_{\rho}(\signal_0)$ and
\begin{multline} \label{eq:mon}
\snorm{\signal_{k+1}^{\delta} - \signal^*}^2 \le \frac{1}{n} \sum_{\ell = k-n+1}^k \snorm{\signal_\ell^\delta - \signal^*}^2 \\
-\frac{1}{n}\sum_{\ell = k-n+1}^k s_\ell\rp_\ell \snorm{\To_{\ii{\ell}}(\signal_\ell^\delta) - \data_{\ii{\ell}}^\delta}\kl{(1-2\eta_{\ii{\ell}})\snorm{\To_{\ii{\ell}}(\signal_\ell^\delta) -\data_{\ii{\ell}}^{\delta}}- 2(1+\eta_{\ii{\ell}}) \delta_{\ii{\ell}}}.
\end{multline}
\end{proposition}

\begin{proof}
Assume for the moment that~\eqref{eq:tcc} and~\eqref{eq:stepsz1} are satisfied on the whole space $\X$ instead
only on  $B_{\rho}(\signal_0)$.
Then, for each $\ell \in\N$, we have
\begin{align*}
\| \asignal_\ell^\delta & - \signal^*\|^2  - \snorm{\signal^\delta_\ell - \signal^*}^2
\\  & =
\snorm{\asignal_\ell^\delta - \signal^\delta_\ell}^2 + 2\sinner{\asignal_\ell^\delta - \signal^\delta_\ell}{\signal^\delta_\ell - \signal^*}
\\  & \leq
s_\ell^2 \rp_\ell^2 \snorm{\To_{\ii{\ell}}'(\signal_\ell^\delta)}^2 \snorm{\To_{\ii{\ell}}(\signal_\ell^\delta) - \data_{\ii{\ell}}^\delta}^2
-2 s_\ell\rp_\ell \sinner{\To_{\ii{\ell}}'(\signal_\ell^\delta)^*\skl{\To_{\ii{\ell}}(\signal_\ell^\delta) - \data_{\ii{\ell}}^\delta}}{\signal_\ell^{\delta} - \signal^*} \\
& \leq
s_\ell\rp_\ell \snorm{\To_{\ii{\ell}}(\signal_\ell^\delta) - \data_{\ii{\ell}}^\delta}^2 - 2 s_\ell\rp_\ell \sinner{\To_{\ii{\ell}}(\signal_\ell^\delta) - \data_{\ii{\ell}}^\delta}{\To_{\ii{\ell}}'(\signal_\ell^\delta)\skl{\signal_\ell^{\delta} - \signal^*}}
\\ & =
s_\ell\rp_\ell \snorm{\To_{\ii{\ell}}(\signal_\ell^\delta) - \data_{\ii{\ell}}^\delta}^2
- 2 s_\ell\rp_\ell \sinner{\To_{\ii{\ell}}(\signal_\ell^\delta) - \data_{\ii{\ell}}^\delta}{\To_{\ii{\ell}}(\signal_\ell^\delta) - \data_{\ii{\ell}}^\delta}
\\ &
\qquad- 2 s_\ell\rp_\ell \sinner{\To_{\ii{\ell}}(\signal_\ell^\delta) - \data_{\ii{\ell}}^\delta}{\To_{\ii{\ell}}(\signal^*) - \To_{\ii{\ell}}(\signal_\ell^\delta) + \To_{\ii{\ell}}'(\signal_\ell^\delta)\skl{\signal_\ell^{\delta} - \signal^*}}
\\ &\qquad
-  2 s_\ell\rp_\ell \sinner{\To_{\ii{\ell}}(\signal_\ell^\delta) - \data_{\ii{\ell}}^\delta}{ \data_{\ii{\ell}}^\delta - \To_{\ii{\ell}}(\signal^*)}\\
& \leq -s_\ell\rp_\ell \snorm{\To_{\ii{\ell}}(\signal_\ell^\delta) - \data_{\ii{\ell}}^\delta}^2
+ 2 \eta_{\ii{\ell}} s_\ell\rp_\ell \snorm{\To_{\ii{\ell}}(\signal_\ell^\delta) - \data_{\ii{\ell}}^\delta} \snorm{\To_{\ii{\ell}}(\signal_\ell^\delta) - \To_{\ii{\ell}}(\signal^*)} \\ &\qquad
+ 2 s_\ell\rp_\ell \delta_{\ii{\ell}} \snorm{\To_{\ii{\ell}}(\signal_\ell^\delta) - \data_{\ii{\ell}}^\delta}
\\ & \leq
- s_\ell\rp_\ell \snorm{\To_{\ii{\ell}}(\signal_\ell^\delta) - \data_{\ii{\ell}}^\delta}\kl{(1-2\eta_{\ii{\ell}})\snorm{\To_{\ii{\ell}}(\signal_\ell^\delta) -\data_{\ii{\ell}}^{\delta}}- 2(1+\eta_{\ii{\ell}}) \delta_{\ii{\ell}}}.
\end{align*}
From Jensen's inequality {(or the triangle inequality)} it follows that
\begin{multline*}
 \snorm{\signal^\delta_{k+1} - \signal^*}^2    =
 \Bigl\Vert \frac{1}{n}\sum_{\ell = k-n+1}^k (\asignal^\delta_\ell - \signal^*)\Bigr\Vert^2
 \le  \frac{1}{n}\sum_{\ell = k-n+1}^k \snorm{\asignal^\delta_\ell - \signal^*}^2
 \le   \frac{1}{n} \sum_{\ell = k-n+1}^k \snorm{\signal_\ell^\delta - \signal^*}^2
 \\
 -\frac{1}{n}\sum_{\ell = k-n+1}^k s_\ell\rp_\ell \snorm{\To_{\ii{\ell}}(\signal_\ell^\delta) - \data_{\ii{\ell}}^\delta}
\kl{ (1-2\eta_{\ii{\ell}})\snorm{\To_{\ii{\ell}}(\signal_\ell^\delta) -\data_{\ii{\ell}}^{\delta}} - 2(1+\eta_{\ii{\ell}}) \delta_{\ii{\ell}} } \,.
\end{multline*}
Recall that there exists a solution $\asignal^*$ of  \eqref{eq:ip} in  $B_{\rho/3}(\signal_0)$ (which can be different from $\signal^*$).
Applying the above inequality  to $\asignal^*$  we obtain   $\snorm{\signal_{k+1}^{\delta} - \asignal^*}^2 \leq \frac{1}{n} \sum_{\ell = k-n+1}^k \snorm{\signal_\ell^\delta - \asignal^*}^2$.
The assumption $\forall \ell \leq k \colon \snorm{\signal_\ell^\delta - \asignal^*}\le 2\rho/3$ therefore implies
$\snorm{\signal_{k+1}^\delta - \asignal^*} \leq  2\rho/3$.
An inductive argument shows that $\snorm{\signal_k^\delta - \asignal^*} \le 2\rho/3$  indeed holds for all $k \in \N$.
Consequently, $\snorm{\signal_k^\delta - \signal_0} \leq \snorm{\signal_k^\delta - \asignal^*} + \snorm{\asignal^* - \signal_0}
\leq \rho$ and therefore $\signal_k^\delta \in B_{\rho}(\signal_0)$.
Thus, for \eqref{eq:mon} to hold, it is in fact sufficient that \eqref{eq:tcc} and~\eqref{eq:stepsz1} are satisfied on
$B_{\rho}(\signal_0) \subseteq \X$.
\end{proof}

The quasi-monotonicity  property~\eqref{eq:mon}
implies that the squared error $\snorm{\signal_{k+1}^{\delta} - \signal^*}^2$
is smaller than the average over
$n$-previous squared errors.  This is a basic ingredient  for our convergence analysis.
However, the absence of  strict monotonicity makes the analysis more
involved than the one of the Landweber and Kaczmarz iterations.

\subsection{Exact data case}
\label{ssec:exact}

In  this subsection  we consider the case of  exact data where
$\delta_i = 0$ for every $i  \in \set{1, \dots, n}$. In this case, we have $\rp_\ell = 1$ and
we write  the  AVEK iteration in the form
\begin{equation}\label{eq:avek-e1}
    \forall k \geq  n \colon \quad
    \signal_{k+1}  = \frac{1}{n} \sum_{\ell = k-n+1}^k \kl{ \signal_\ell  - s_\ell \To_{\ii{\ell}}'(\signal_\ell)^*
    \skl{ \To_{\ii{\ell}}(\signal_\ell) - \data_{\ii{\ell}}}}.
\end{equation}
We will prove weak convergence of~\eqref{eq:avek-e1} to a solution  of \eqref{eq:ip}.
To that end we start  with the following technical lemma.

\begin{lemma}\label{lem:conv}
Assume that $(p_k)_{k\in \N}$ is a sequence of non-negative numbers
satisfying $p_{k+1} \le \frac{1}{n}\sum_{\ell = k - n+1}^{k} p_\ell $ for all $k \geq n$.
Then  $(p_k)_{k\in \N}$ is convergent.
\end{lemma}

\begin{proof}
Define $\qq_k \coloneqq \max \set{ p_\ell  \mid \ell \in \sset{k-n+1 , \dots,   k}}$.
Then $\qq_k$ is a non-increasing sequence and $\lim_{k \to \infty} \qq_k = c$  for some  $c \geq 0$.
Further, $\limsup_{k \to \infty} p_k = c$. Anticipating a contradiction, we assume  that there exists some $\eps >0$
such that $ \liminf_{k \to \infty} p_k = c - 3 \epsilon$.
Then there are a subsequence $\kl{k(i)\in \N}_{i \in \N}$ and a positive integer $i_0$ such that $p_{k(i)} \le c - 2\epsilon$ for all $k(i) \ge k(i_0)$.
Noting that $\limsup_{k \to \infty} p_k = c$, we can assume $i_0$ being  sufficiently large such that
$p_k \leq c + \epsilon/n$ for all $k \ge k(i_0)$. For $\ell = 1,\ldots,n-1$ and $k(i) \ge k(i_0)$, we have
\begin{equation*}
p_{k(i) + \ell} \leq \frac{1}{n}\sum_{j = k(i) + \ell - n + 1}^{k(i) + \ell} p_j \le \frac{n-1}{n}  \kl{c+\frac{\epsilon}{n}}
+ \frac{c-2\epsilon}{n}  \leq c - \frac{\epsilon}{n} \,.
\end{equation*}
Because $p_k \le \max \set{ p_j  \mid  j \in \sset{k(i_0), \dots, k(i_0) + n - 1} } \leq c - \epsilon / n$ for $k \ge k(i_0)$, this contradicts $\limsup_{k \to \infty} p_k = c$. We therefore conclude $\lim_{k \to \infty} p_k = c$.
\end{proof}

Some implications of the quasi-monotonicity  of the AVEK iteration
(see Proposition~\ref{prop:mon}) are collected next.

\begin{lemma}\label{lem:init}
Let Assumption~\ref{ass} be satisfied and let $\signal^* \in B_{\rho}(\signal_0)$ be
a solution of~\eqref{eq:ip}. Define $(\signal_k)_{k\in \N}$  by~\eqref{eq:avek-e1}, where the step sizes  $s_k$
satisfy~\eqref{eq:stepsz1}.  Then the  following hold true:
\begin{enumerate}
\item  \label{prop:init1}
$\snorm{\signal_{k} - \signal^*}$ is convergent as $k \to \infty$.

\item  \label{prop:init2}
If $\inf \sset{ s_k \mid  k \in \N } >0$, then
$\sum_{k \in \N}\snorm{\To_{\ii{k}}(\signal_k) - \data_{\ii{k}}}^2 < \infty$.
\end{enumerate}
\end{lemma}

\begin{proof}
Proposition~\ref{prop:mon} for the case $\delta_i =0$ yields
\begin{equation}\label{eq:mon0}
\snorm{\signal_{k+1} - \signal^*}^2 \le \frac{1}{n} \sum_{\ell = k-n+1}^k \norm{\signal_{\ell} - \signal^*}^2
-\frac{1}{n}\sum_{\ell = k-n+1}^k (1-2\eta_{\ii{\ell}}) s_\ell\snorm{\To_{\ii{\ell}}(\signal_\ell) - \data_{\ii{\ell}}}^2 \,.
\end{equation}
This, together with Lemma~\ref{lem:conv},  implies that $\norm{\signal_{k} - \signal^*}$ is convergent as $k \to \infty$. Summing~\eqref{eq:mon0} from $k = n$ to $ k = m+n$ gives
\begin{multline}\label{eq:res}
\sum_{i = 1}^n i\norm{\signal_{i + m +1}-\signal^*}^2 -\sum_{i = 1}^n i \norm{\signal_{i} - \signal^*}^2
\\ \leq
 -\sum_{k = n}^{m+n} \sum_{\ell = k-n+1}^k (1-2\eta_{\ii{\ell}}) s_\ell\snorm{\To_{\ii{\ell}}(\signal_\ell) - \data_{\ii{\ell}}}^2.
\end{multline}
Therefore, we have $\sum_{k = 1}^{m+n} \snorm{\To_{\ii{k}}(\signal_k) - \data_{\ii{k}}}^2 \leq \frac{1}{M}
\sum_{i = 1}^n i \snorm{\signal_{i} - \signal^*}^2    < \infty$ for all $m \in \N$,
with constant $M \coloneqq   (1 - 2\max_{i=1, \dots, n} \eta_i  ) \inf_{k\in \N}  s_k $.
The assertion follows by letting $m \to \infty$.
\end{proof}

For  the Landweber and  Kaczmarz iterations  strict monotonicity   of
$\snorm{\signal_k -  \signal^* } $ holds. From this one can  show that
$\snorm{\signal_{k+1} -  \signal_k }$ converges to zero. The following
Lemma~\ref{lem:diff} states that the  same result holds true for the AVEK iteration.
However, its proof is much more involved and therefore presented in the appendix.

\begin{lemma}\label{lem:diff}
Under the  assumptions of  Lemma~\ref{lem:init},
we have $\lim_{k \to \infty } \norm{\signal_{k+1} -  \signal_k } = 0 $.
\end{lemma}

\begin{proof}
See Appendix~\ref{app:deconv}.
\end{proof}

For the subsequent analysis we also use the following known result on the sequential closedness of the graph of operators $\To_i$.

\begin{lemma}\label{lem:close}
Suppose that \ref{ass1}-\ref{ass2} in Assumption~\ref{ass} hold
and let  $i \in \set{1, \ldots, n}$. If $(\signal_k)_{k \in \N}$ is a sequence in $B_{\rho}(\signal_0)$
converging weakly to some $\bar\signal$ and
$(\To_i(\signal_k))_{k\in \N}$ converges strongly to $\bar\data$ in $\Y_i$, then $\To_i(\bar\signal) = \bar\data$.
\end{lemma}

\begin{proof}
See~\cite[Proposition 2.2]{leitao16projective}.
\end{proof}

Now we are ready to show the  weak convergence of the AVEK iteration
$(\signal_k)_{k\in\N}$. The  presented  proof uses ideas taken
from~\cite{leitao16projective}.

\begin{theorem}[Convergence\label{thm:exact}  for exact data]
Let Assumption~\ref{ass} hold and define $(\signal_k)_{k\in \N}$ by \eqref{eq:avek-e1},
with step sizes $s_k$ satisfying \eqref{eq:stepsz1} and  $\inf \set{s_k \mid k \in \N} >0$.
Then the following hold:
\begin{enumerate}
\item \label{thm:exact-a}
We have $\signal_k \rightharpoonup \signal^*$ as $k \to \infty$, where
$\signal^* \in B_{\rho}(\signal_0)$ is a solution of~\eqref{eq:ip}.

\item \label{thm:exact-b}
If the initialisation is chosen as $\signal_1 = \cdots = \signal_n = \signal_0$,
and \begin{equation}\label{eq:incl}
\forall  \signal \in B_{\rho}(\signal_0) \colon \quad
\mathcal{N}\kl{\To'(\signal^\Plus)} \subseteq \mathcal{N}\kl{\To'(\signal)}
\end{equation}
where $\signal^\Plus$ is an $\signal_0$-minimal norm solution of~\eqref{eq:ip}, then
$\signal_k \rightharpoonup \signal^\Plus$ as $k \to \infty$.
\end{enumerate}
\end{theorem}

\begin{proof} \ref{thm:exact-a}:
From Proposition~\ref{prop:mon} it follows that $\signal_k \in B_{\rho}(\signal_0)$ and therefore
$\kl{\signal_k}_{k \in \N}$  has at least one weak accumulation point $\signal^*$.
Suppose  $\bar\signal$ is any weak accumulation point of  $\kl{\signal_k}_{k \in \N}$
and assume  $\signal_{k(j)} \rightharpoonup \bar\signal$ as $j \to \infty$.
For every $i = 1, \ldots, n$ define $k_i(j)$ in such a way that  $\ii{k_i(j)} = i $ and $k(j) \le k_i(j) \le k(j) + n-1$.
Then
\[
\forall i \in \set{1, \ldots, n} \colon \quad
\snorm{\signal_{k(j)} - \signal_{k_i(j)}} \le \sum_{\ell = k(j)}^{k(j)+n - 2} \norm{\signal_{\ell+1} - \signal_\ell} \to 0\quad \text{ as } j \to \infty\,.
\]
By Lemma~\ref{lem:init} we have $ \snorm{\To_{i}(\signal_{k_i(j)}) - \data_i}  \to  0$ as $j \to \infty$,
and therefore  $\lim_{j \to \infty} \snorm{\To_{i}(\signal_{k(j)}) - \data_i} = 0$  for all
$i \in \set{1, \ldots, n}$.  Together with {Lemma~\ref{lem:close}}
this implies  that $\bar\signal$ is a solution of~\eqref{eq:ip}.
Now assume that $\hat \signal$ is another weak accumulation point with $\hat\signal \neq \bar\signal$  and that $\signal_{m(j)} \rightharpoonup \hat\signal$
as $j \to \infty$.
Then $\bar\signal$ and $\hat\signal$ are both solutions to \eqref{eq:ip}. By Lemma~\ref{lem:init}
and~\cite[Lemma~1]{Opi67}, we obtain
\begin{equation*}
\lim_{k \to \infty}\snorm{\signal_k - \bar\signal} =
\liminf_{j \to \infty} \snorm{\signal_{k(j)} - \bar\signal}
< \liminf_{j\to\infty}\snorm{\signal_{k(j)} - \hat\signal}
= \lim_{k \to \infty}\snorm{\signal_k - \hat\signal}
\end{equation*}
and likewise
$\lim_{k \to \infty}\snorm{\signal_k - \hat\signal}  < \lim_{k \to \infty} \snorm{\signal_k - \bar\signal}$.
This  leads to a contradiction and therefore the weak accumulation point of $\kl{\signal_k}_{k \in \N}$
is unique which implies $\signal_k \rightharpoonup \signal^*$.

\ref{thm:exact-b}:
An inductive argument, together with the definition of $\signal_k$ shows
\begin{multline*}
\signal_{k} =  \sum_{i = 1}^n w_{i,k} \signal_{i} - \sum_{\ell = 1}^{k-1} c_{\ell,k} s_\ell\To_{\ii{\ell}}'(\signal_\ell)^*\kl{\To_{\ii{\ell}}(\signal_\ell) - \data_{\ii{\ell}}}
\\ = \signal_0 - \sum_{\ell = 1}^{k-1} c_{\ell,k} s_\ell\To_{\ii{\ell}}'(\signal_\ell)^*\kl{\To_{\ii{\ell}}(\signal_\ell) - \data_{\ii{\ell}}}
\end{multline*}
for some  $0 < w_{i,k} < 1$ with $\sum_{i = 1}^n w_{i,k} = 1$ and $0 < c_{\ell,k} < 1$. Note that
\begin{equation*}
\forall \signal \in B_{\rho}(\signal_0) \colon \quad
\mathcal{R}\kl{\To_{i}'(\signal)^*} \subseteq \mathcal{N}\kl{\To'_{i}(\signal)}^\perp\subseteq \mathcal{N}\kl{\To'(\signal)}^\perp\subseteq \mathcal{N}\kl{\To'(\signal^\Plus)}^\perp \,.
\end{equation*}
Thus $\signal_k \in \signal_0 + \mathcal{N}\kl{\To'(\signal^\Plus)}^\perp$ and,
by continuity of $\To'(\signal^\Plus)$, we  have  $\signal^* \in \signal_0 + \mathcal{N}\kl{\To'(\signal^\Plus)}^\perp$. Together with  \cite[Proposition 2.1]{hanke1995convergence} we conclude $\signal^* = \signal^\Plus$.
\end{proof}

\subsection{Noisy data case}
\label{ssec:noisy}

Now we consider the noisy data case, where  $\delta_i>0$ for
$i \in \sset{1, \dots, n}$.  The  AVEK iteration is then defined by \eqref{eq:avek1}-\eqref{eq:avek3}
and stopped at the index
\begin{equation}\label{eq:stop}
k^*(\delta) \coloneqq \min \set{\ell n \in\N \mid \signal^\delta_{\ell n} = \cdots = \signal^\delta_{\ell n+n-1}}.
\end{equation}
The following Lemma shows that the stopping index is well defined.

\begin{lemma}\label{lem:stop}
The stopping index $k^*(\delta)$ defined in~\eqref{eq:stop} is finite, and the corresponding
residuals satisfy  $\snorm{\To_{i}(\signal^\delta_{k^*(\delta)}) - \data_i} < \tau_i \delta_i$ for all
$i  \in \sset{ 1, \ldots, n}$.
\end{lemma}

\begin{proof}
Similar to~\eqref{eq:res}, from Proposition~\ref{prop:mon}  we obtain
\begin{multline*}
 \sum_{k = n}^{m+n} \sum_{\ell = k-n+1}^k \rp_\ell s_\ell\snorm{\To_{\ii{\ell}}(\signal^\delta_\ell) - \data^\delta_{\ii{\ell}}} \kl{(1-2\eta_{\ii{\ell}})\snorm{\To_{\ii{\ell}}(\signal^\delta_\ell) - \data^\delta_{\ii{\ell}}} - 2(1+\eta_{\ii{\ell}})\delta_{\ii{\ell}}}  \\
 \le  \sum_{i = 1}^n i \snorm{\signal_{i}^\delta - \signal^*}^2.
\end{multline*}
Note that either
$\snorm{\To_{\ii{\ell}}(\signal_\ell^\delta) -\data^\delta_{\ii{\ell}}} \ge \tau_{\ii{\ell}} \delta_{\ii{\ell}}$ or
it holds $\rp_\ell =0$. If $k^*(\delta)$ is infinite, there are infinitely many $\ell$ such that $\snorm{\To_{\ii{\ell}}(\signal^\delta_\ell) - \data^\delta_{\ii{\ell}}} \ge \tau_{\ii{\ell}} \delta_{\ii{\ell}}$. This implies that the left hand side of the above displayed equation tends to infinity as $m \to \infty$, which gives a contradiction. Thus $k^*(\delta)$ is finite.
Again by Proposition~\ref{prop:mon}, we obtain $\snorm{\To_{i}(\signal^\delta_{k^*(\delta)}) - \data_i} < \tau_i \delta_i$, for $i = 1, \ldots, n$.
\end{proof}

We next show the continuity of $\signal^\delta_k$ at $\delta = 0$. For that purpose denote
\begin{equation*}
\Delta_k(\delta, \data, \data^\delta) \coloneqq \sum_{\ell = k-n+1}^k\rp_\ell \To_{\ii{\ell}}'(\signal_\ell^\delta)^*\kl{\To_{\ii{\ell}}(\signal^\delta_\ell)-\data^\delta_{\ii{\ell}}}
- \sum_{\ell = k-n+1}^k \To'_{\ii{\ell}}(\signal_\ell)^*\kl{\To_{\ii{\ell}}(\signal_\ell) - \data_{\ii{\ell}}} \,.
\end{equation*}

\begin{lemma}\label{lem:continuity}
For all $k\in\N$, we have
\begin{itemize}
\item $
\lim_{\delta \to 0} \sup\set{\snorm{\Delta_k(\delta, \data, \data^\delta)} \mid
\forall i = 1,\ldots, n  \colon \snorm{\data^\delta_i - \data_i}\le \delta_i }  = 0
$;
\item $\lim_{\delta \to 0} \signal^\delta_{k} = \signal_{k}$.
\end{itemize}
\end{lemma}

\begin{proof}
We prove the assertions  by induction. The case $k \leq  n$ is shown similar to the general case and
therefore omitted. Assume that  $k \ge n+1$ and that  the assertions hold for all $m < k$. It follows immediately that $\signal^\delta_k \to \signal_k$ as $\delta \to 0$. Note that
\begin{equation*}
\norm{\Delta_k(\delta, \data, \data^\delta)} \le
\sum_{\ell = k-n+1}^k \norm{\rp_\ell \To_{\ii{\ell}}'(\signal_\ell^\delta)^*\kl{\To_{\ii{\ell}}(\signal^\delta_\ell)-\data^\delta_{\ii{\ell}}} - \To'_{\ii{\ell}}(\signal_\ell)^*\kl{\To_{\ii{\ell}}(\signal_\ell) - \data_{\ii{\ell}}}}.
\end{equation*}
For each $\ell \in \sset{ k-n+1, \ldots, k}$, we consider two cases. In the  case $\rp_\ell = 1$, the
continuity of $\To$ and $\To'$ implies
$
\snorm{\To_{\ii{\ell}}'(\signal_\ell^\delta)^*\skl{\To_{\ii{\ell}}(\signal^\delta_\ell)-\data^\delta_{\ii{\ell}}} - \To'_{\ii{\ell}}(\signal_\ell)^*\skl{\To_{\ii{\ell}}(\signal_\ell) - \data_{\ii{\ell}}}} \to 0
$ as $\delta \to 0$. In the case $\rp_\ell = 0$, we have $\snorm{\To_{\ii{\ell}}(\signal^\delta_\ell)-\data^\delta_{\ii{\ell}}} < \tau_{\ii{\ell}}\delta_{\ii{\ell}}$ and therefore, as $\delta \to 0$,
\begin{align*}
\| \To'_{\ii{\ell}}(\signal_\ell)^* &\skl{\To_{\ii{\ell}}(\signal_\ell) - \data_{\ii{\ell}}} \|\\
\le &\snorm{\To'_{\ii{\ell}}(\signal_\ell)}\snorm{\To_{\ii{\ell}}(\signal_\ell) - \data_{\ii{\ell}}} \\
\le  &\snorm{\To'_{\ii{\ell}}(\signal_\ell)}\kl{\snorm{\To_{\ii{\ell}}(\signal_\ell) - \To_{\ii{\ell}}(\signal^\delta_\ell)}+\snorm{\To_{\ii{\ell}}(\signal^\delta_\ell) - \data^\delta_{\ii{\ell}}}+\snorm{\data^\delta_{\ii{\ell}}- \data_{\ii{\ell}}}} \\
\le &\snorm{\To'_{\ii{\ell}}(\signal_\ell)}\kl{\snorm{\To_{\ii{\ell}}(\signal_\ell) - \To_{\ii{\ell}}(\signal^\delta_\ell)} + (1+\tau_{\ii{\ell}})\delta_{\ii{\ell}}} \to 0 \,.
\end{align*}
Combining these two cases, we obtain $\snorm{\Delta_k(\delta, \data, \data^\delta)} \to 0$ as $\delta \to 0$.
\end{proof}

\begin{theorem}[Convergence for noisy data]\label{thm:noisy}
Let  $\delta(j)  \coloneqq (\delta_{1}(j), \ldots, \delta_{n}(j))$ be a sequence in $(0, \infty)^n$ with
$\lim_{j \to \infty} \max_{i = 1, \ldots, n} \delta_{i}(j) = 0$, and let  $\data(j) = (\data_{1}(j), \ldots,\data_{n}(j))$ be a sequence of noisy data with $\snorm{\data_{i}(j) - \data_i} \le \delta_{i}(j)$. Define   $\signal_k^{\delta(j)}$ by~\eqref{eq:avek1}-\eqref{eq:avek3} with $\data(j)$ and $\delta(j)$ in place of $\data^\delta$ and $\delta$,
 and define $k^*(\delta(j))$ by~\eqref{eq:stop}.
Then the following assertions hold true:
\begin{enumerate}
\item\label{noise_a}
The sequence $\signal^{\delta(j)}_{k^*(\delta(j))}$ has at least one weak accumulation point and every such  weak
accumulation point  is a solution of~\eqref{eq:ip}.

\item\label{noise_b}
If, in the case of exact data,  $\signal_k$ converges strongly to $\signal^*$, then  $\lim_{j \to \infty} \signal_{k^*(\delta(j))}^{\delta(j)} = \signal^*$.
\item\label{noise_c}
If the initializations are chosen as $\signal^{\delta(j)}_1 = \cdots = \signal^{\delta(j)}_n = \signal_0$, and~\eqref{eq:incl}
is satisfied, then each (strong or weak) limit $\signal^*$ is an $\signal_0$-minimal norm solution of~\eqref{eq:ip}.
\end{enumerate}
\end{theorem}

\begin{proof}
\mbox{} \ref{noise_a}:  By Proposition~\ref{prop:mon} the sequence
$\signal(j) \coloneqq \signal^{\delta(j)}_{k^*(\delta(j))}$ remains in  $B_{\rho}(\signal_0)$ and therefore
has at least   one weak accumulation point.
Let $\signal^*$ be a weak  accumulation point of $(\signal(j))_{j\in \N}$  and
$(\signal(j(\ell)))_{\ell\in \N}$ a subsequence
with $ \signal(j(\ell)) \rightharpoonup \signal^* $ as $\ell \to \infty$.
By Lemma~\ref{lem:stop} and the triangle inequality, for every $i\in \sset{1,\dots, n}$  we have
$\snorm{\To_i(\signal(j(\ell))) - \To_i(\signal^*)} \le {(1+\tau_i)}\delta_i(j(\ell))  \to 0$  as $\ell \to \infty$.
Using {Lemma~\ref{lem:close}} we conclude that $\signal^*$ is a solution of~\eqref{eq:ip}.

\ref{noise_b}: We consider two cases. In the first case we assume that $(k^*(\delta(j)))_{j \in \N}$ is bounded. It is sufficient to show that for each accumulation point $k^*$ of $(k^*(\delta(j)))_{j\in\N}$, which is clearly finite, it holds that $\lim_{j\to\infty}\signal^{\delta(j)}_{k^*} = \signal^*$. Without loss of generality, we can assume that
$k^*(\delta(j)) = k^*$ for all sufficiently large $j$. By Lemma~\ref{lem:stop}, we have $\snorm{\To_i(\signal_{k^*}^{\delta(j)}) - \data^{\delta(j)}_i} \le \tau_i \delta_i(j)$ and, by taking the limit $j \to \infty$, that $\To_i(\signal_{k^*}) = \data_i$.
Thus, it holds that $\signal_{k^*} = \signal^*$ and therefore $\signal^{\delta(j)}_{k^*} \to \signal^*$ as $j \to \infty$.

In the second case, we assume $\limsup_{j\to\infty}k^*(\delta(j)) = \infty$. Without loss of generality, we can assume that $k^*(\delta(j))$ is monotonically increasing. For any $\varepsilon >0$, there exists some  $m \in \N$  with $\snorm{\signal_{m -i +1} - \signal^*} \le \varepsilon/2$ for $i = 1, \ldots, n$. An inductive argument, together with Proposition~\ref{prop:mon}
shows  $\snorm{\signal_{k+m}^{\delta} - \signal^*} \le \sum_{i = 1}^n w_{i,k}
\snorm{\signal^{\delta}_{m-i+1} - \signal^*}$ for certain weighs  $0 < w_{i,k} < 1$ with $\sum_{i = 1}^n w_{i,k} = 1$.
Then for sufficiently large $j$ it holds that
\begin{multline*}
	\snorm{\signal_{k^*(\delta(j))}^{\delta(j)} - \signal^*}
	\leq  \max_{i =1,\ldots,n} \snorm{\signal_{m -i +1}^{\delta(j)} - \signal^*} \\
	\leq  \max_{i =1,\ldots,n} \kl{\snorm{\signal_{m -i +1}^{\delta(j)} - \signal_{m-i+1}}
	+ \snorm{\signal_{m -i +1} - \signal^*}}
	\leq  \max_{i =1,\ldots,n} \snorm{\signal_{m -i +1}^{\delta(j)} - \signal_{m-i+1}} + \varepsilon / 2 \,.
\end{multline*}
From Lemma~\ref{lem:continuity}, we have $\snorm{\signal_{m -i +1}^{\delta(j)} - \signal_{m-i+1}} \le \varepsilon/2$
for sufficiently large $j$. We thus conclude that $\snorm{\signal_{k^*(\delta(j))}^{\delta(j)} - \signal^*} \le \varepsilon$, and therefore, $\lim_{j \to \infty} \signal_{k^*(\delta(j))}^{\delta(j)} = \signal^*$.

\ref{noise_c}: This  follows similarly as in Theorem~\ref{thm:exact}~\ref{thm:exact-b}.
\end{proof}

\begin{psfrags}
\psfrag{O}{$D(R)$}
\psfrag{S}{ $\Gamma$}
\psfrag{C}{ $\partial D(R) \setminus \Gamma$}
\psfrag{f}{$f$}
\psfrag{x}{$z$}
\begin{figure}
\centering
\includegraphics[width=0.6\textwidth]{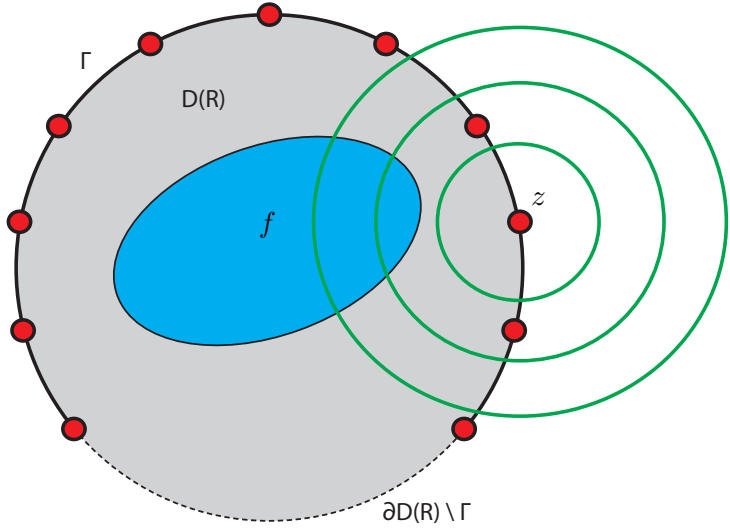}
\caption{\textsc{Recovering a function from the circular Radon transform.}
The function $f$ (representing some physical quantity of interest) is supported inside  the disc  $D(R)$.
Detectors are placed at various locations  on the observable part of the boundary
$\Gamma \subseteq \partial D(R)$ and record averages of $f$  over  circles with varying radii.
No detectors can be  placed at the  un-observable part $\partial D(R) \setminus \Gamma$
of the boundary.
\label{fig:setting}}
\end{figure}
\end{psfrags}

\section{Application to the circular Radon transform}
\label{sec:num}

In this section we apply the AVEK iteration to the limited view problem for the
circular Radon transform. We present numerical results for exact and noisy data, and
compare the AVEK iteration  to other standard iterative schemes, namely the Kaczmarz and
the  Landweber iteration.

\subsection{The circular Radon transform}

Consider the circular Radon transform,  {which} maps a function $\fsignal \colon \R^2 \to \R$
supported in the disc $D(R)  \coloneqq  \sset{ x \in   \R^2 \mid \snorm{x} < R}$  to the function
$\Mo \fsignal  \colon  \Gamma  \times [0, 2 R] \to \R$ defined by
\begin{equation} \label{eq:sm}
\kl{\Mo \fsignal} \kl{z, r}
\coloneqq
\frac{1}{2\pi}
\int_0^{2\pi} \fsignal \kl{ z + (r  \cos \beta , r \sin \beta)  } \rmd \beta
\quad \text{ for }  (z, r ) \in   \Gamma  \times [0, 2 R] \,.
\end{equation}
Here $\Gamma  \subseteq \partial D(R)$ is the observable part of the boundary  $\partial D(R)$ enclosing
the support of $\fsignal$, and the function value $\kl{\Mo \fsignal} \kl{z, r}$ is the average
of $\fsignal$ over a circle with  center  $  z \in \Gamma$ and
radius $r \in [0, 2R]$.
Recovering a function from circular  means  is important for many modern  imaging applications, where the centers of the circles  of integration  correspond to admissible locations of detectors; see Figure~\ref{fig:setting}.
For example, the circular Radon transform is essential for the hybrid imaging modalities photoacoustic  and
thermoacoustic tomography, where the function $\fsignal$ models the initial pressure of
the induced acoustic field~\cite{kuchment2011mathematics,xu2005universal,burgholzer2007temporal,ZanSchHal09b}.
The inversion from circular means is also  important for technologies such as SAR and SONAR imaging~\cite{And88,BelFel09},
ultrasound tomography~\cite{NorLin81}  or seismic imaging~\cite{BleCohSto01}.

The case $\Gamma = \partial D(R)$  corresponds to the complete data situation, where  the
circular Radon transform is known to be smoothing as {half integration}; therefore its inversion
 is mildly ill-posed. This follows, for example, from the explicit inversion formulas derived
 in~\cite{finch2007inversion}.  In this paper we are particularly  interested  in the limited data case
corresponding to $\Gamma \subsetneq \partial D(R)$. In such a situation, no explicit inversion
formulas exist. Additionally, the limited data problem is severely ill-posed and  artefacts are expected  when
reconstructing a general  function with  support in $D(R)$;
see~\cite{barannyk2016artefacts,frikel2016artifacts,nguyen2015artefacts,stefanov2013curved}.

\subsection{Mathematical  problem formulation}

In the following, let $\Gamma_i \subseteq \partial D(R)$ for $i \in \sset{ 1, \dots , n}$ denote relatively
closed subsets of $\partial D(R)$ whose interiors are pairwise disjoint. We call $\Gamma_i $ the $i$-th
detection curve  and define the $i$-th partial circular Radon  transform by
\begin{equation*}
	\Mo_i \colon L^2( D(R) )  \to   L^2 ( \Gamma_i \times [0, 2R] ; 4 r \pi ) \colon \fsignal \mapsto \rest{\Mo \fsignal}{\Gamma_i \times [0,2R] } \,.
\end{equation*}
Here $\Mo \fsignal$ is defined by~\eqref{eq:sm} and $\rest{\Mo \fsignal}{\Gamma_i \times [0,2R] }$  denotes the  restriction of  $\Mo \fsignal$ to circles whose centers are  located on $\Gamma_i$.  Further, $ L^2 ( \Gamma_i \times [0, 2R] ; 4 r \pi)$ is the Hilbert space  of all functions $\gdata_i \colon \Gamma_i \times [0, 2R] \to \R$ with
$\snorm{ \gdata_i }^2 \coloneqq 4 \pi  \int_{\Gamma_i} \int_{0}^{2R} \abs{\gdata_i(z,r)}^2  r \rmd r \rmd s(z) < \infty $, where $\rmd s$ is the arc length measure (i.e.~the standard one-dimensional surface measure).
Inverting the circular Radon transform is then equivalent to solving the system of linear of equations
\begin{equation}\label{eq:ipm}
	\Mo_i( \fsignal ) =  \gdata_i  \quad \text{ for } i=1, \dots, n \,.
\end{equation}
In the case that $\bigcup_{i=1}^n \Gamma_i = \partial D(R)$  we have complete
data; otherwise we face the limited data problem. In any case,  regularization methods have to be
applied for solving \eqref{eq:ipm}. Here we apply iterative regularization
methods for  that purpose.

\begin{lemma} \label{lem:M}
For any $i \in \sset{ 1 , \dots ,  n}$,  the following hold:
\begin{enumerate}
\item
$\Mo_i$ is well defined, bounded and linear.

\item
We have $\norm{ \Mo_i  } \leq  \sqrt{\abs{\Gamma_i} }$,
where   $\abs{\Gamma_i}$ is the arc length measure of $\Gamma_i$.

\item
The adjoint $\Mo_i^\ast  \colon L^2 ( \Gamma_i \times [0, 2R] ; 4 r \pi )  \to L^2( D(R) )$
is given by
\begin{equation*}
(\Mo_i^\ast \gdata)(x) = 2 \int_{\Gamma_i} \gdata(z, \snorm{z-x}) \rmd s(z)
\quad \text{ for } x \in  D(R) \,.
\end{equation*}

\end{enumerate}
\end{lemma}

\begin{proof}
All claims are easily verified using Fubini's theorem.
\end{proof}

From Lemma \ref{lem:M} we conclude that \eqref{eq:ipm}  fits in the general framework
studied in this paper, with  $\To_i = \Mo_i$, $\X= L^2( D(R) )$
and  $\Y_i = L^2 (\Gamma_i \times [0, 2R]; 4 r \pi) $. {Note that the norm of $\Mo_i$ implicitly  depends on the radius $R$  through the arc length of $\Gamma_i$.} {Because the circular Radon transform is linear,  the local tangential cone condition~\eqref{eq:tcc} is satisfied with $\eta_i=0$
 for all $i \in \sset{1, \dots, n}$.} In particular, the  established convergence  analysis  for the AVEK  method can be applied. The same holds true for the Landweber and the Kaczmarz  iteration.

Suppose noisy data  $\gdata_i^\delta \in L^2 ( \Gamma_i \times [0, 2R] ; 4 r \pi )$  with
$\snorm{\gdata_i^\delta - \Mo_i \fsignal} \leq \delta_i$ are given.
The  Landweber, Kaczmarz and AVEK iteration  for
reconstructing  $\fsignal$  from such data are given by
\begin{align*}
	\fsignal^\delta_{k+1}
	& =  \fsignal^\delta_k  -  \frac{s_k}{n} \sum_{i=1}^n \Mo_i^*\skl{ \Mo_i(\fsignal_k^\delta) - \gdata^\delta_i}
\\
	\fsignal^\delta_{k+1}
	& =  \fsignal^\delta_k  -   s_k \rp_k \Mo_{\ii{k}}^*\skl{ \Mo_{\ii{k}}(\fsignal_k^\delta) - \gdata^\delta_{\ii{k}}}  \\
	\fsignal_{k+1}^{\delta}
	& =
	\frac{1}{n} \sum_{\ell = k-n+1}^k \fsignal_\ell^{\delta} - s_\ell\rp_\ell \Mo_{\ii{\ell}}^*
\skl{ \Mo_{\ii{\ell}}(\fsignal^\delta_\ell) - \gdata^\delta_{\ii{\ell}}} \,,
\end{align*}
respectively.  Here $s_k$ are step sizes and $\rp_k \in \sset{0,1}$ the additional
parameters for noisy data.
How we implement these iterations is outlined  in the following subsection.

\subsection{Numerical implementation}

In the numerical implementation, $\fsignal \colon \R^2 \to \R$ is represented by a discrete vector
$\fnum \in \R^{(N_x+1) \times (N_x+1)} $ obtained by uniform  sampling
\begin{equation*}
\fnum[j]   \simeq f((-R,-R) + j 2R / N_x) \quad \text{ for } j  = (j_1, j_2)\in \{0, \dots, N_x\}^2
\end{equation*}
on a cartesian grid.
Further,  any function $\gdata \colon  \partial D(R) \times [0, 2R] \to \R$
is  represented by a discrete vector $\gnum \in \R^{N_\ph \times (N_r+1)}$, with
\begin{equation*}
\gnum[k,\ell]  \simeq
g \kl{   \kl{R \cos ({2\pi k}/{N_\varphi}), R\sin ({2\pi k}/{N_\varphi})}, \ell \, \frac{2 R}{N_r} } \,.
\end{equation*}
Here  $N_\ph$  denotes the number  of  equidistant detector locations on the full boundary $\partial D(R)$.
We further write $K_i$ for the set of all indices in $\set{0, \dots, N_\varphi -1}$ with  detector
location  $R\bigl(\cos (2\pi k/N_\varphi), \sin (2\pi k/N_\varphi)\bigr)$ contained in  $\Gamma_i$;
 the corresponding   discrete data  are denoted by  $\gnum_i \in \R^{\sabs{K_i} \times (N_r+1)}$.

The AVEK, Landweber and Kaczmarz   iterations are implemented by replacing
$\Mo_i$ and $\Mo_i^\ast$ for any $i \in \sset{1, \dots, N}$ with discrete counterparts
\begin{align*}
	 & \Mnum_i
	 \colon \R^{(N_x+1) \times (N_x+1) }  \to \R^{\sabs{K_i} \times (N_r+1)} \,, \\ 
 	 & \Bnum_i
	 \colon  \R^{\sabs{K_i} \times (N_r+1)}  \to \R^{(N_x+1) \times (N_x+1) }  \,.
\end{align*}
For that purpose we compute the discrete  spherical means $\Mnum_i \fnum$  using the trapezoidal rule
for discretizing the integral over $\beta$ in~\eqref{eq:sm}.
The function values of $\fsignal$ required the trapezoidal rule are obtained  by the bilinear
interpolation  of $\fnum$.  The discrete circular backprojection $\Bnum_i$ is a numerical approximation  of
the adjoint of the $i$-th partial circular Radon transform. It is implemented using a backprojection
procedure described in detail in~\cite{burgholzer2007temporal,finch2007inversion}.
Note that $\Bnum_i$ is based on the continuous adjoint $\Mo^*_i $  and is not the exact adjoint of the
discretization  $\Mnum_i \fnum$.   See, for example,  \cite{wang2014discrete} for a discussion on
the use of discrete and continuous adjoints.

Using the above discretization, the  resulting discrete Landweber, Kaczmarz and
AVEK  iterations are given  by
 \begin{align*}
	\fnum_{k+1}^\delta  &=
	\fnum_{k}^\delta
    -
    \frac{s_k}{n} \sum_{i=1}^n
	 \Bnum_i
	\skl{ \Mnum_i\fnum^\delta_k - \gnum^\delta_k} \\
	\fnum_{k+1}^\delta  &=
	\fnum_k^{\delta} - s_k  \rp_k \Bnum_{\ii{k}}
	\skl{ \Mnum_{\ii{k}}\fnum^\delta_k - \gnum^\delta_{\ii{k}}}  \\
	\fnum_{k+1}^\delta  &=
	\frac{1}{n} \sum_{\ell = k-n+1}^k
	\fnum_{\ell}^{\delta} - s_\ell \rp_\ell \Bnum_{\ii{\ell}}
	\skl{ \Mnum_{\ii{\ell}}\fnum_\ell^\delta - \gnum^\delta_{\ii{\ell}}} \,.
\end{align*}
respectively. Here $\gnum^\delta_i \in \R^{\abs{K_i} \times (N_r+1)}$ are discrete
noisy data, $s_k$ are step size parameters and  $\rp_k \in \sset{0,1}$ additional
tuning parameters  for noisy data. We always choose the zero vector $\tt 0\in \R^{(N_x+1) \times (N_x+1)}$ as the initialization; that is, $\fnum_{1}^\delta \coloneqq \tt 0$ for the Landweber and the Kaczmarz iteration, and $\fnum_{1}^\delta  = \cdots = \fnum_{n}^\delta \coloneqq \tt 0$ for the AVEK iteration.

\begin{figure}[htb!]\centering
\includegraphics[width=0.4\textwidth]{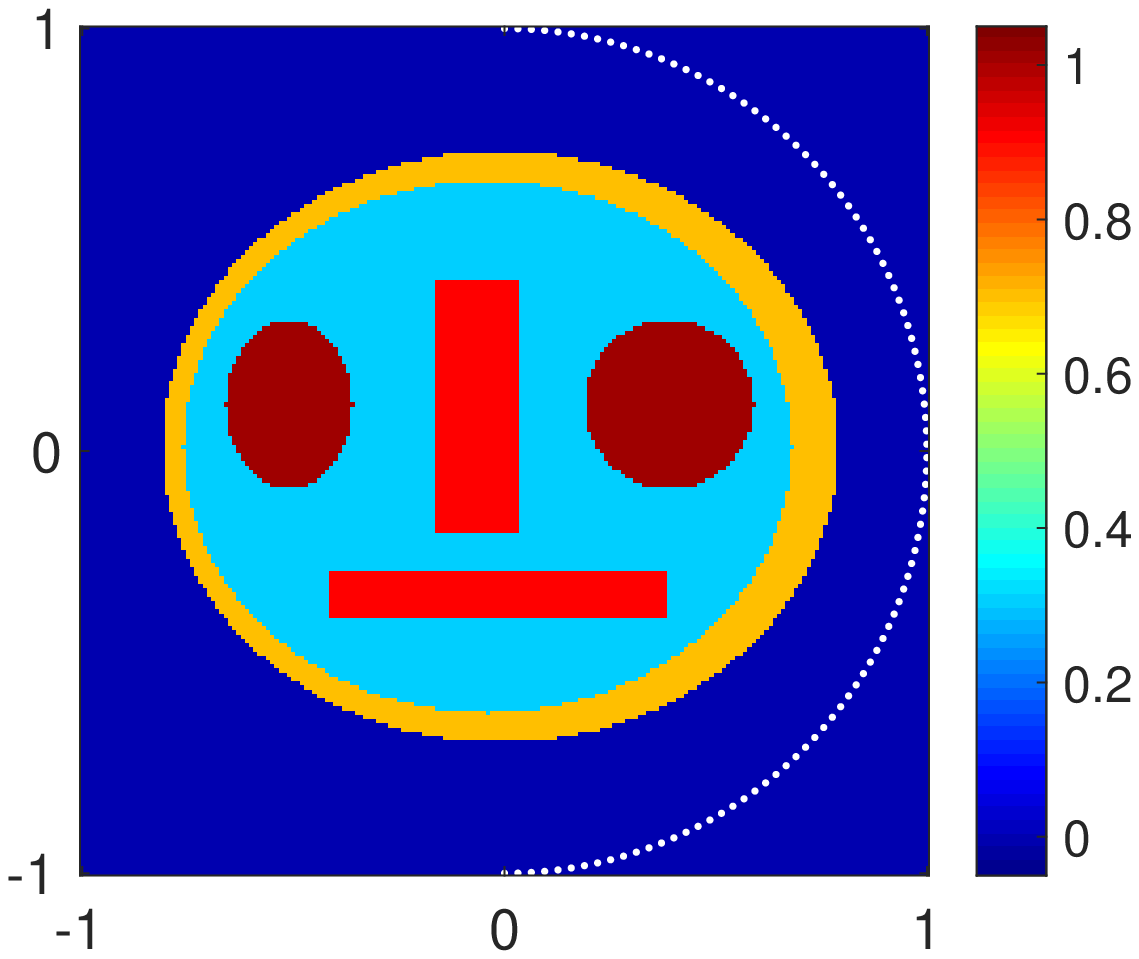}\hspace{0.8cm}
\includegraphics[width=0.4\textwidth]{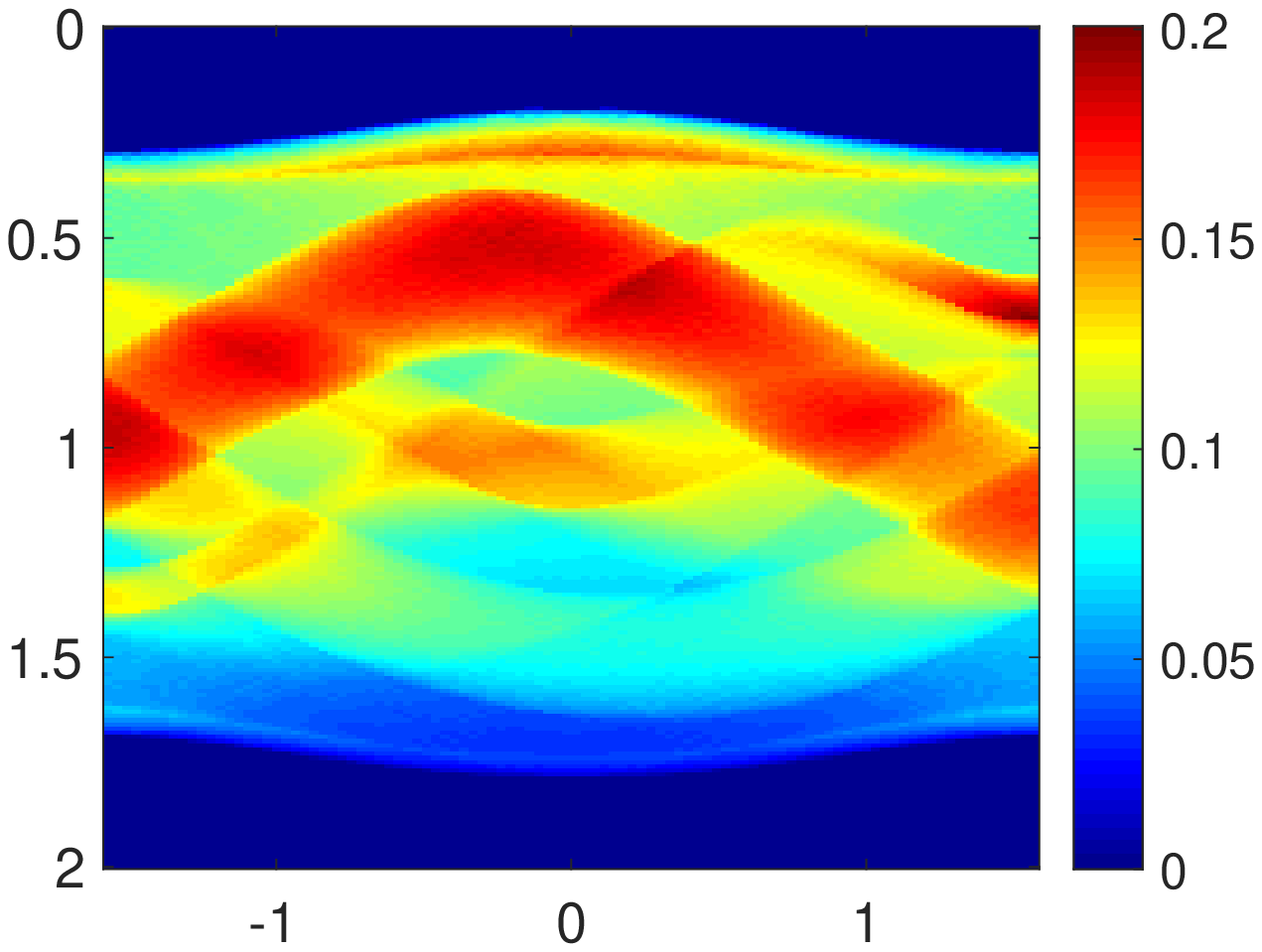}
\caption{Left: The phantom  $\fnum \in \R^{201 \times 201}$ discretizing the head like function supported in a
disc of radius 1. The white dots indicate locations of detectors. Right: The simulated  discrete circular Radon transform $\gnum \in \R^{100 \times 201}$.
The horizontal axis  is the detector location in $[-\pi/2, \pi/2]$; the vertical  axis the radius in $[0, 2]$.
Any partial data $\gnum_i \in \R^{1 \times 201}$ corresponds to a column.}
\label{fig:fg}
\end{figure}

\subsection{Numerical simulations}\label{ss:num}

In the  following numerical results we consider  the case where $R = 1$. We
assume  measurements  on the half circle $\Gamma  = \sset{ (\zsignal_1, \zsignal_2) \in \sph^1  \mid
\zsignal_2 >0}$, choose $N_x  = N_r = 200$ and use $N = 100$ detector locations on $\Gamma$.
Further, we use a partition  of $\Gamma$  in 100 arcs $\Gamma_i$ of equal arc length {(i.e.~$n =100$)}.  The phantom
$\fnum \in \R^{201 \times 201}$  used for the  presented results and the numerically
computed data  $\Mnum_i \fnum \in \R^{1 \times 201}$ for $i =1, \dots, 100$ are shown  Figure~\ref{fig:fg}.
We refer to  one cycle of the iterative methods after we performed an update using any of the equation.
One such  cycle consists of  $n$ consecutive  iterative updates for the AVEK and the Kaczmarz
iteration and one iterative update for the Landweber  iteration. The numerical  effort for one
cycle in any of the considered methods  is given by  $\mathcal{O}(N N_x^2)$,  with similar
leading  constants.  {For a fair comparison of step sizes, we rescale any of the  operators  $\Mo_i$  and $\Mo$ in such a way that $\norm{\Mo_i} \simeq \norm{\Mo} \simeq 1$ for  $i =1, \dots, 100$. Further, in the Kaczmarz and the  AVEK method the equations are randomly rearranged prior to each cycle.} {We empirically observed that this accelerates the  convergence of both methods.}

\subsubsection*{Results for exact data}

We first consider the case of exact data shown in Figure~\ref{fig:fg}.  The step sizes for Landweber, Kaczmarz and AVEK are chosen constant and at different values. The convergence behavior during the first  80 cycles is shown in Figure~\ref{fig:vars}. As can be seen, the Landweber is the slowest and the Kaczmarz  and the AVEK are comparably fast under suitable choice of step sizes.  Note that although our convergence analysis {of AVEK} assumes a  step size below 1, the AVEK method allows for a rather wide range of step sizes (up to 30 for this example), and that larger step sizes turn out to be  stable and yield faster convergence. This is not the case for the Landweber and the Kaczmarz method, where a step size above 3 yields divergence.

\begin{figure}[htb!]\centering
\includegraphics[width=0.32\textwidth]{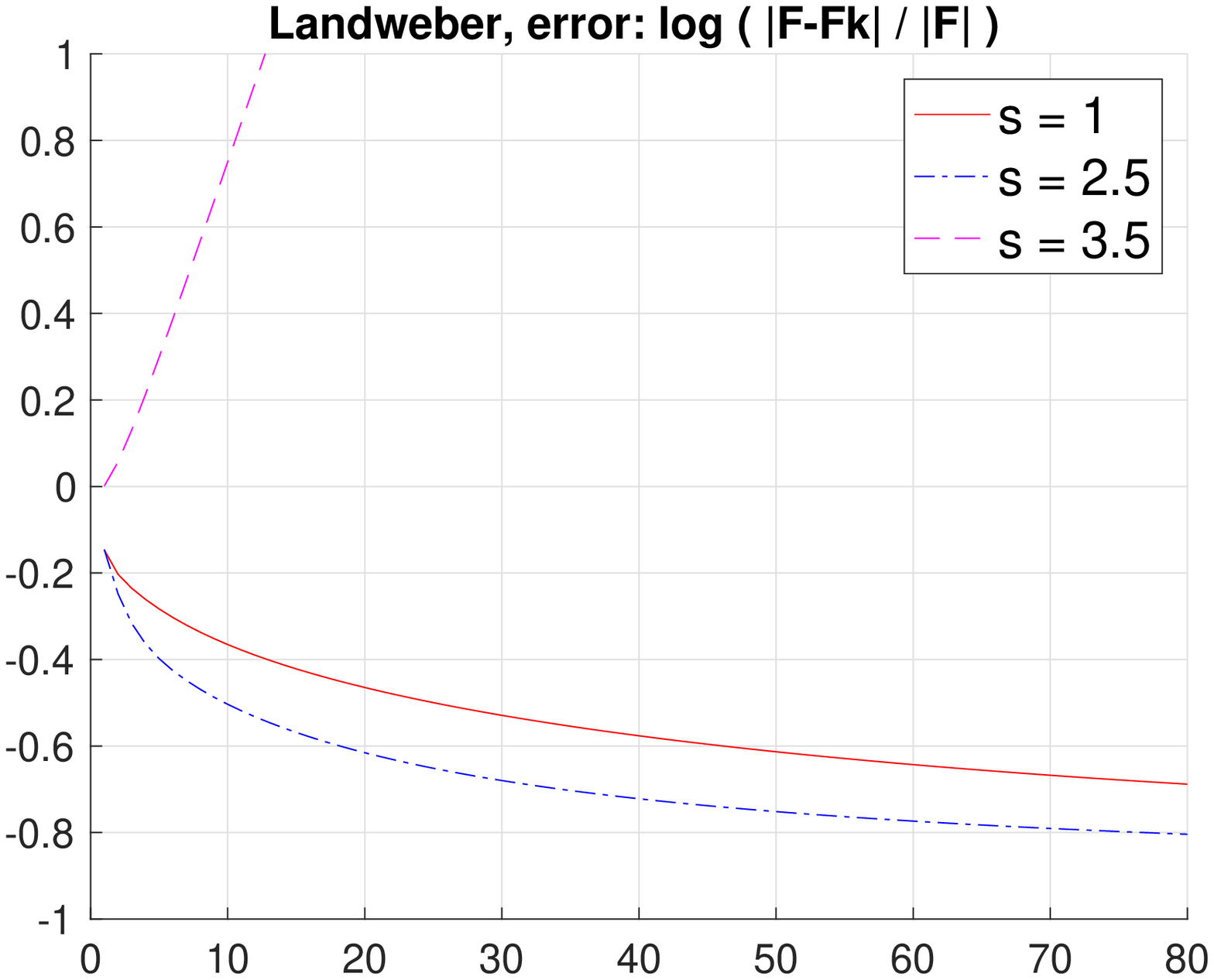}\hspace{0.1cm}
\includegraphics[width=0.32\textwidth]{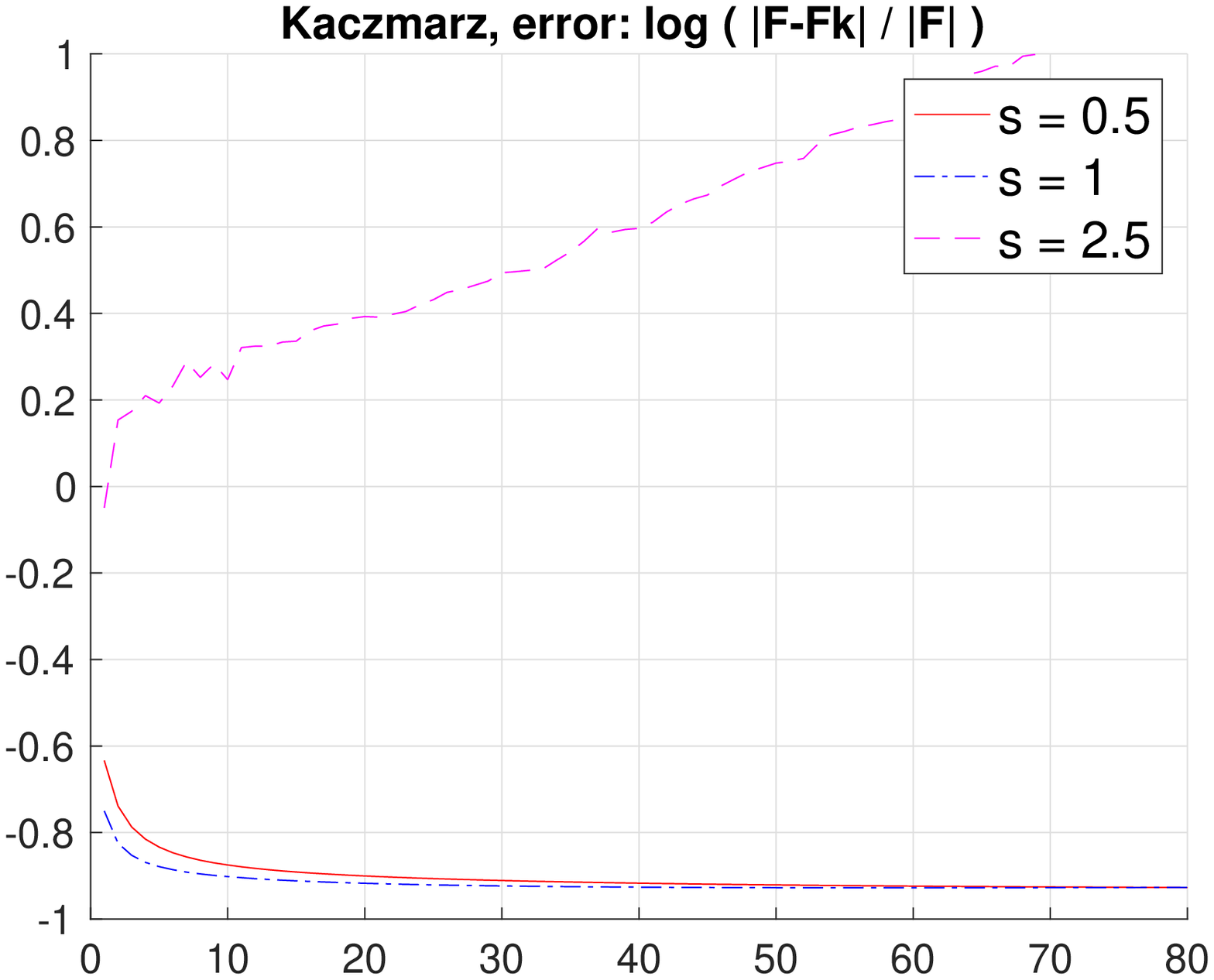}\hspace{0.1cm}
\includegraphics[width=0.32\textwidth]{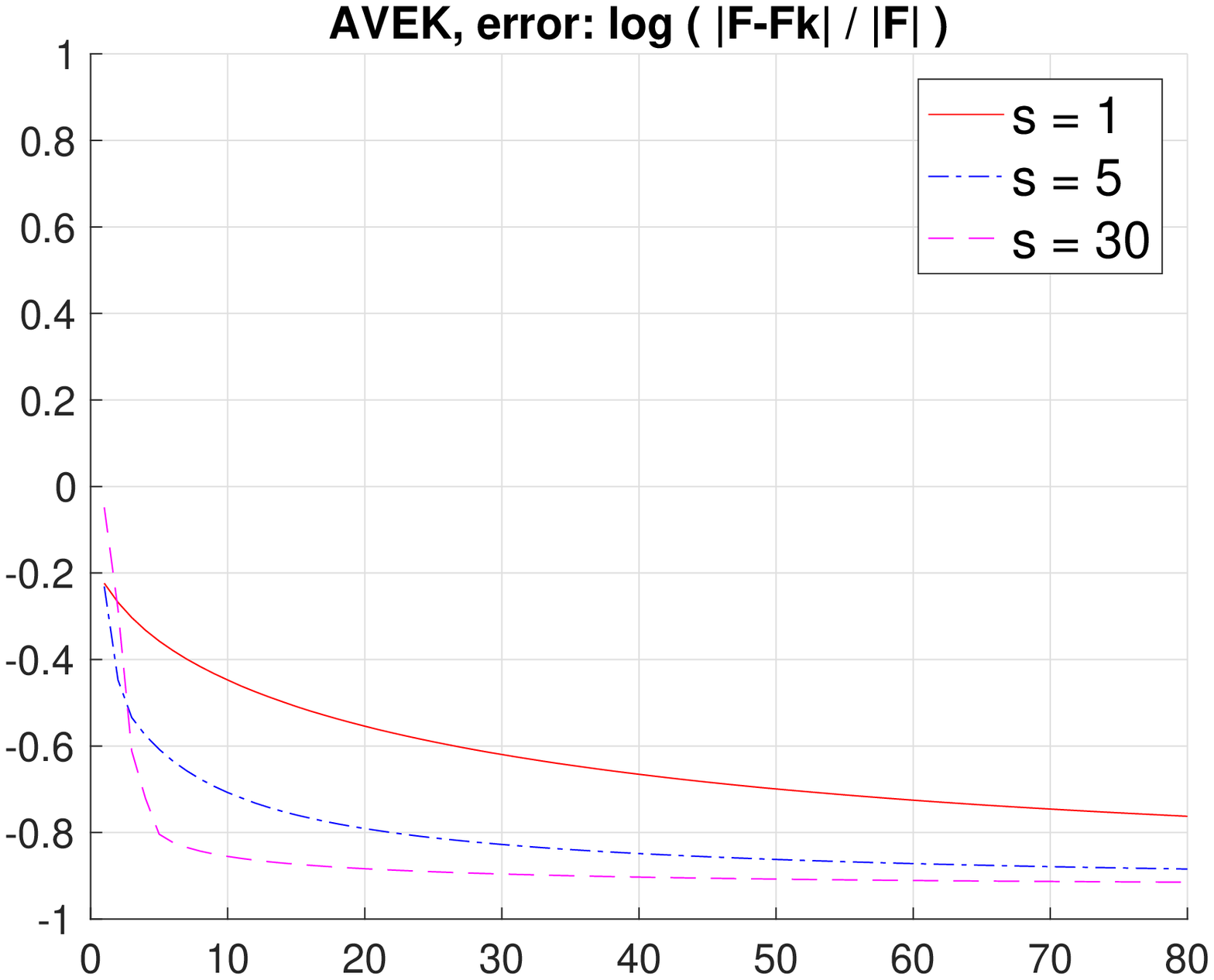}\vspace{0.3cm}\\
\includegraphics[width=0.32\textwidth]{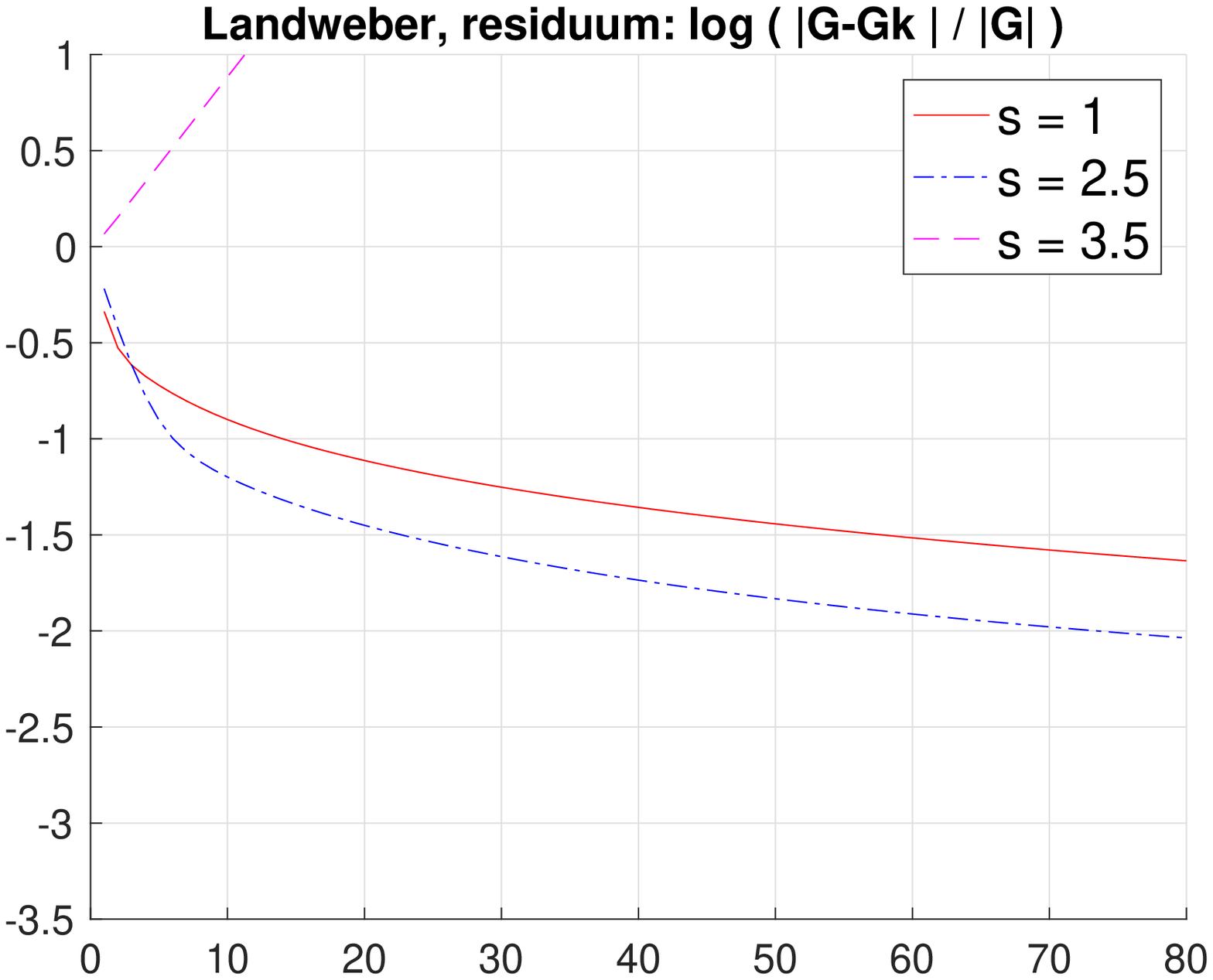}\hspace{0.1cm}
\includegraphics[width=0.32\textwidth]{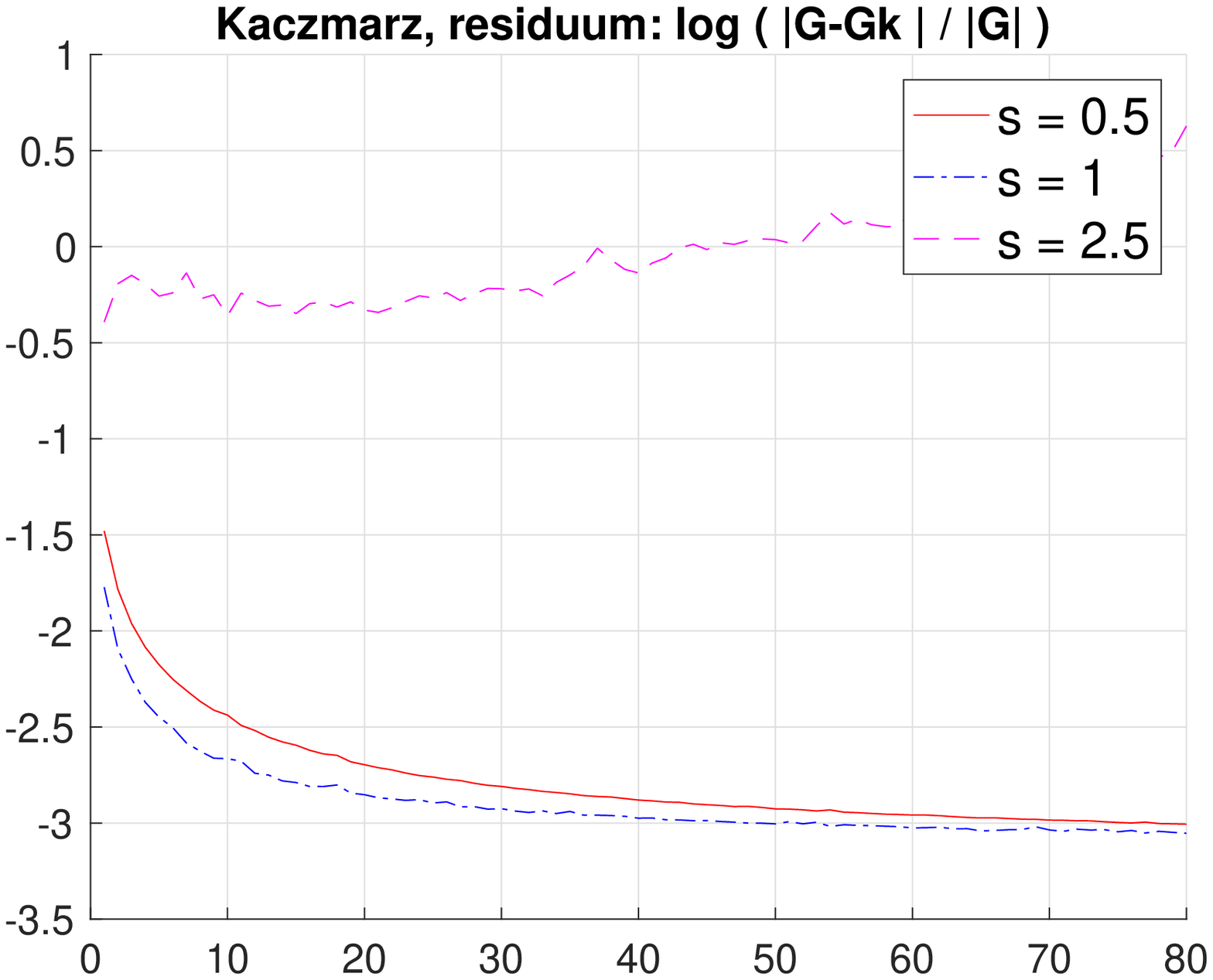}\hspace{0.1cm}
\includegraphics[width=0.32\textwidth]{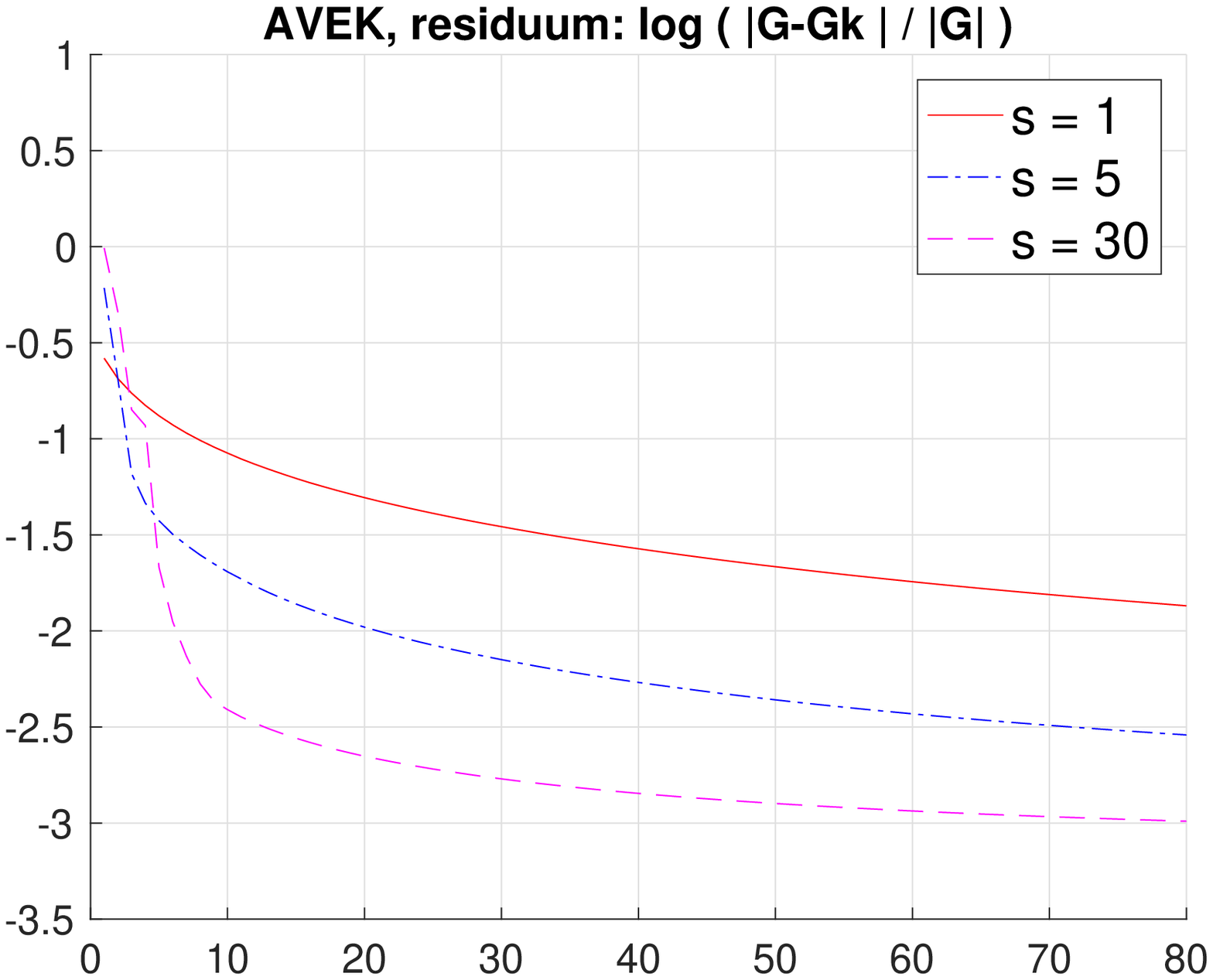}
\caption{{Residuum and relative reconstruction  error (after taking logarithm to basis 10) of Landweber, Kaczmarz and AVEK with different step sizes for exact data during the first 80 cycles.}}
\label{fig:vars}
\end{figure}

In order to visually compare the results, we choose proper step sizes for all methods in the sense that the iterations are fast and on the other hand stable. More precisely,  for the Landweber iteration the step size has been taken as $s_{\rm LW} = 2.5$, for the Kaczmarz iteration as $s_{\rm K} = 1$ and for the AVEK as $s_{\rm AVEK} = 30$. In Figure~\ref{fig:exact} we show  reconstructions using the three considered methods after 10, 20 and 80 iterations.  In any case, one notes reconstruction artifacts outside the convex hull of the detection curve, which is expected using limited view data \cite{barannyk2016artefacts,frikel2016artifacts,nguyen2015artefacts,stefanov2013curved}. Inside the convex hull, the  Kaczmarz and the AVEK give quite accurate results already after a reasonable number of cycles.

\begin{figure}[htb!]\centering
\includegraphics[width=0.28\textwidth]{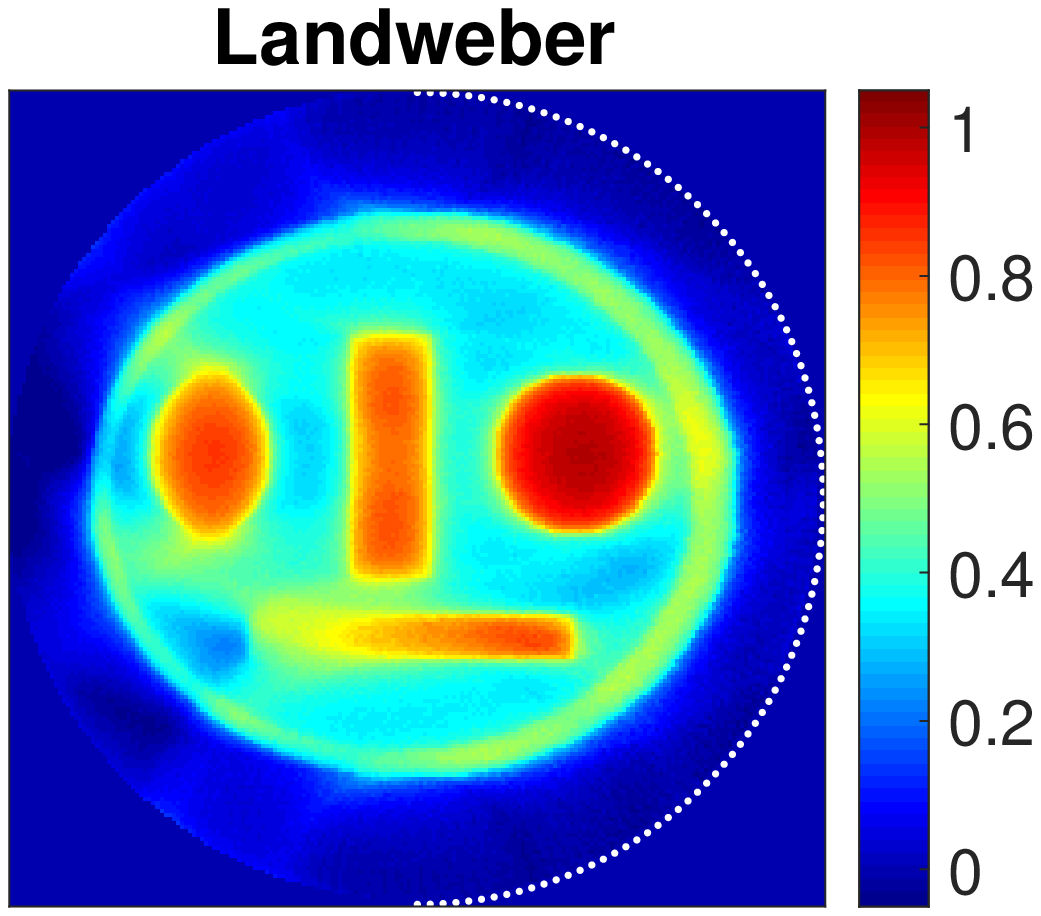}\hspace{0.15cm}
\includegraphics[width=0.28\textwidth]{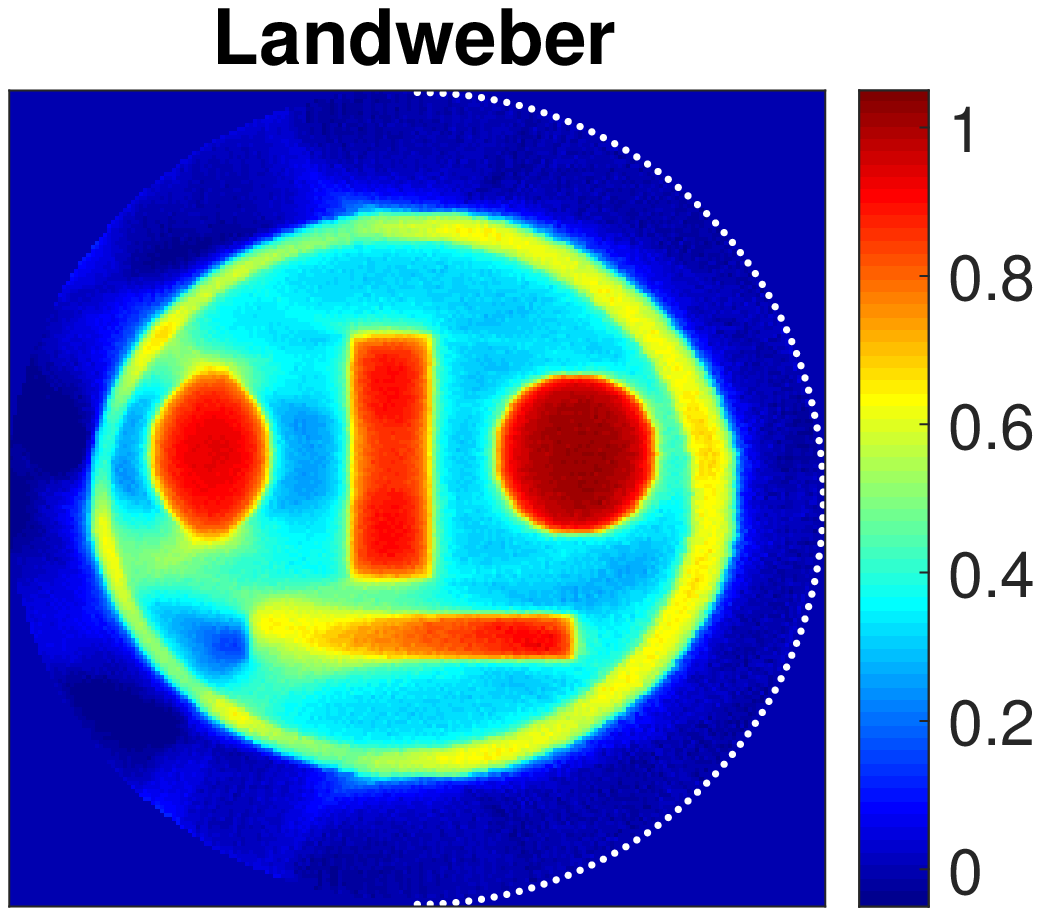}\hspace{0.15cm}
\includegraphics[width=0.28\textwidth]{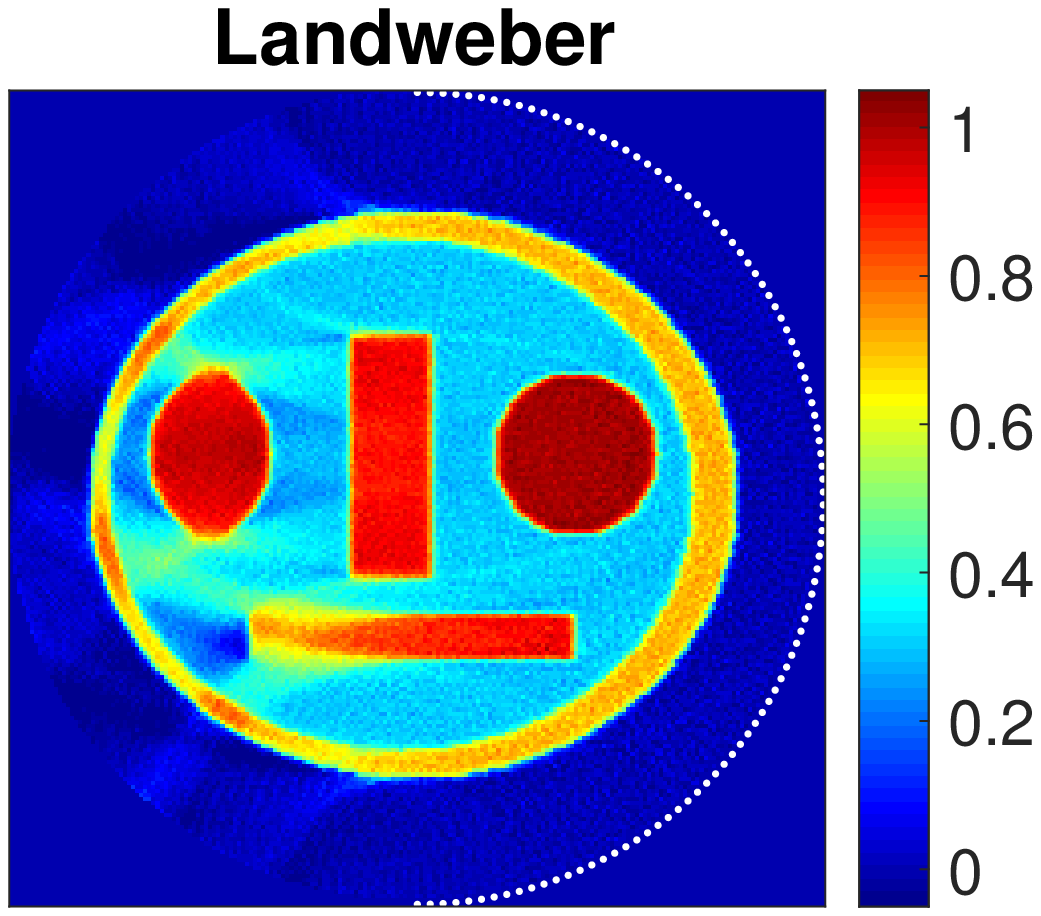} \vspace{0.15cm}\\
\includegraphics[width=0.28\textwidth]{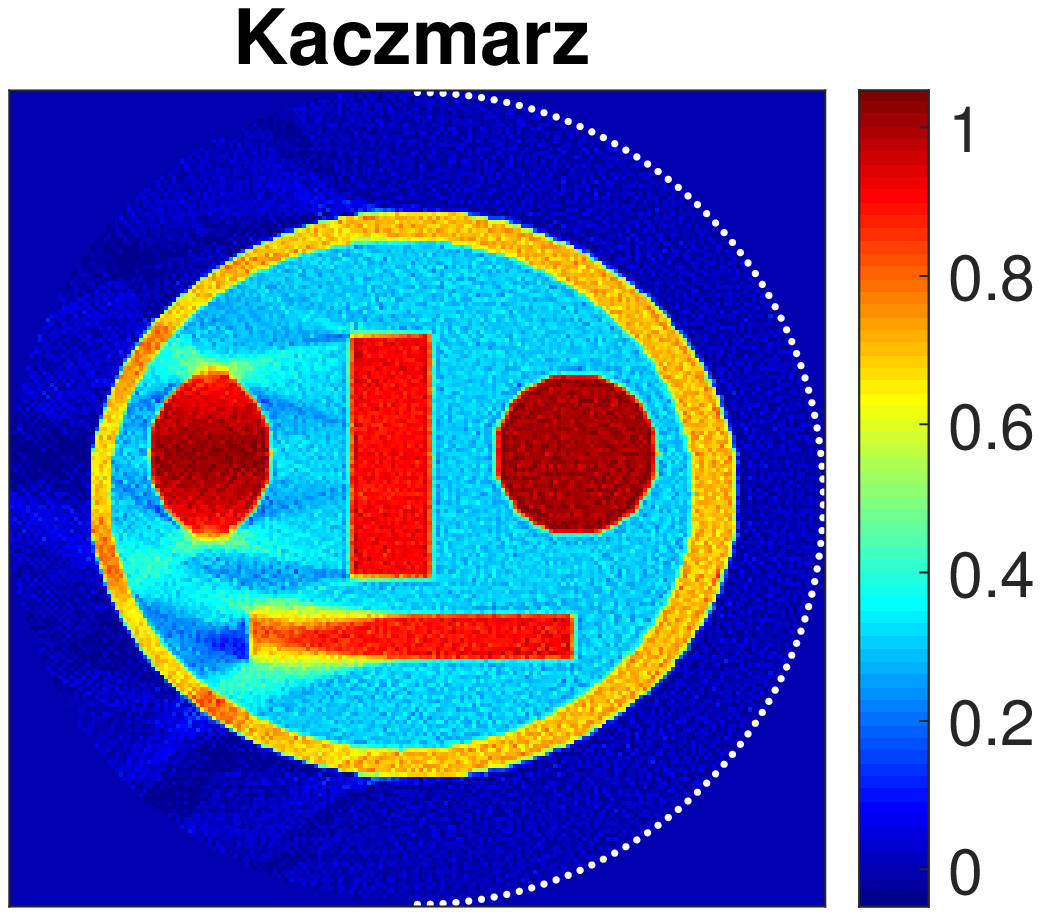}\hspace{0.15cm}
\includegraphics[width=0.28\textwidth]{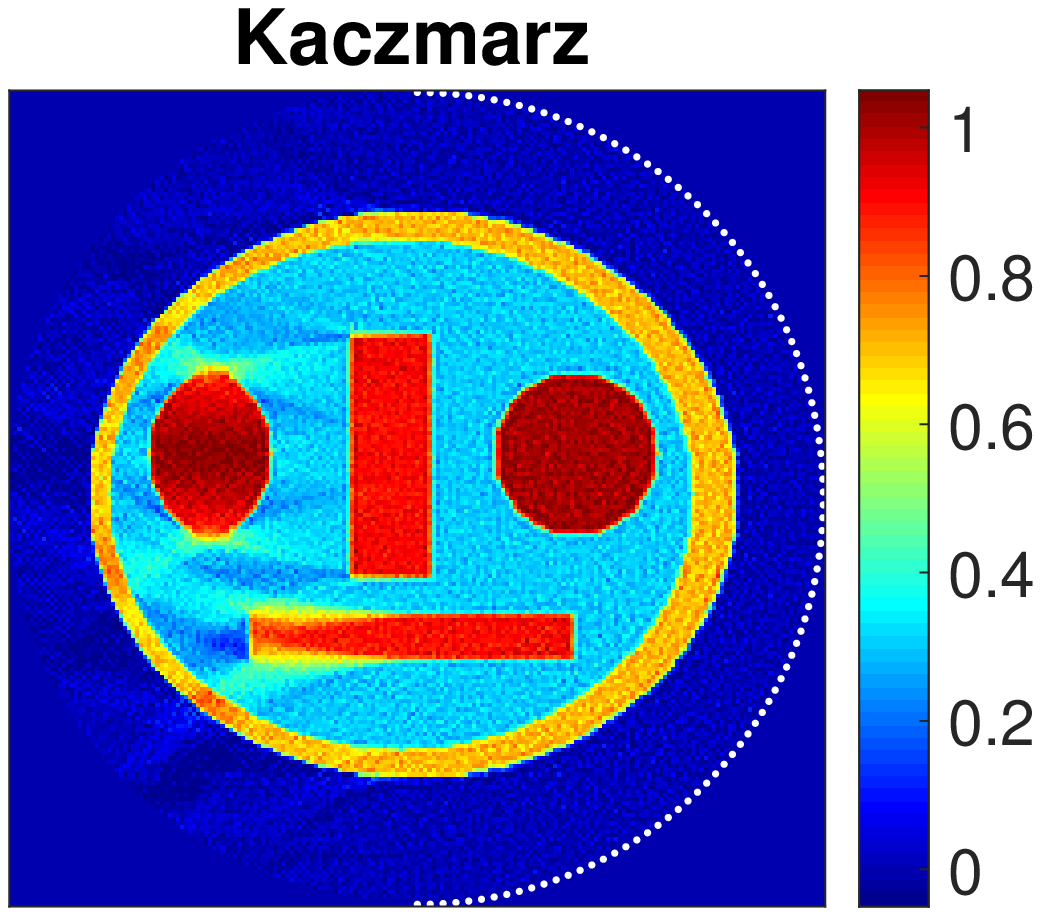}\hspace{0.15cm}
\includegraphics[width=0.28\textwidth]{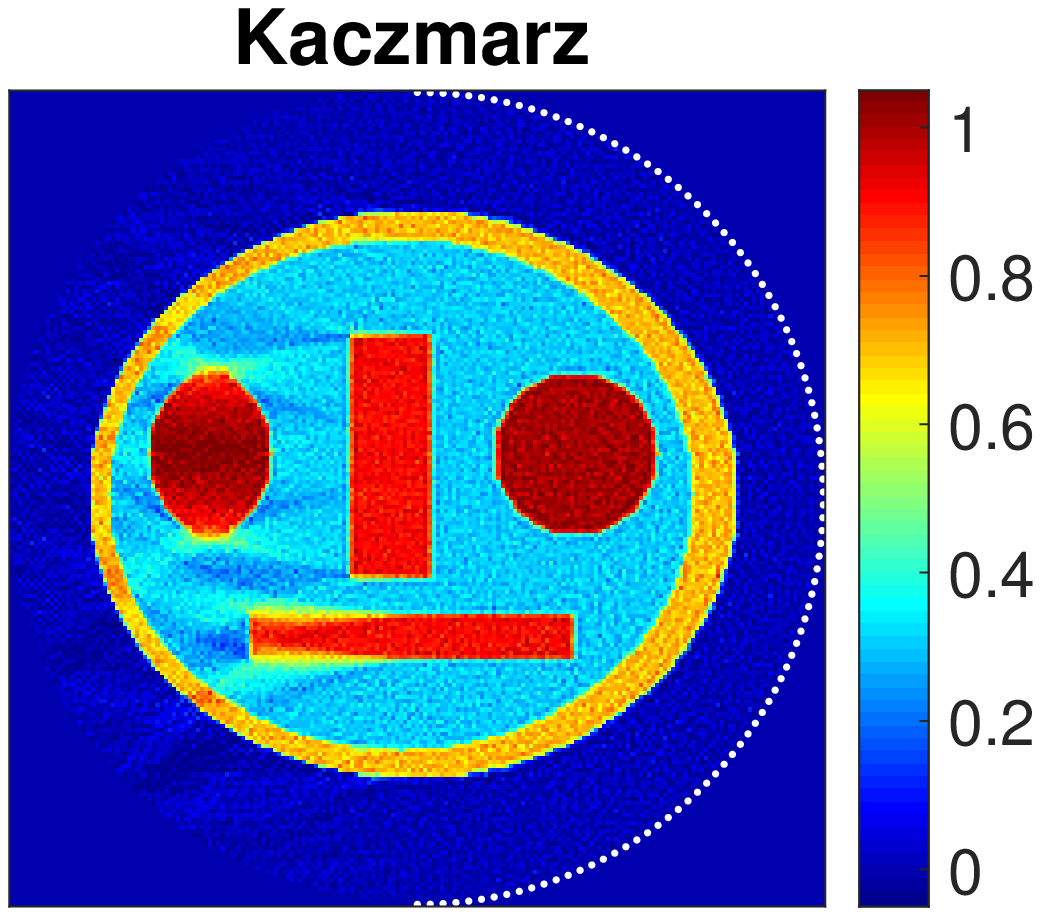} \vspace{0.15cm} \\
\includegraphics[width=0.28\textwidth]{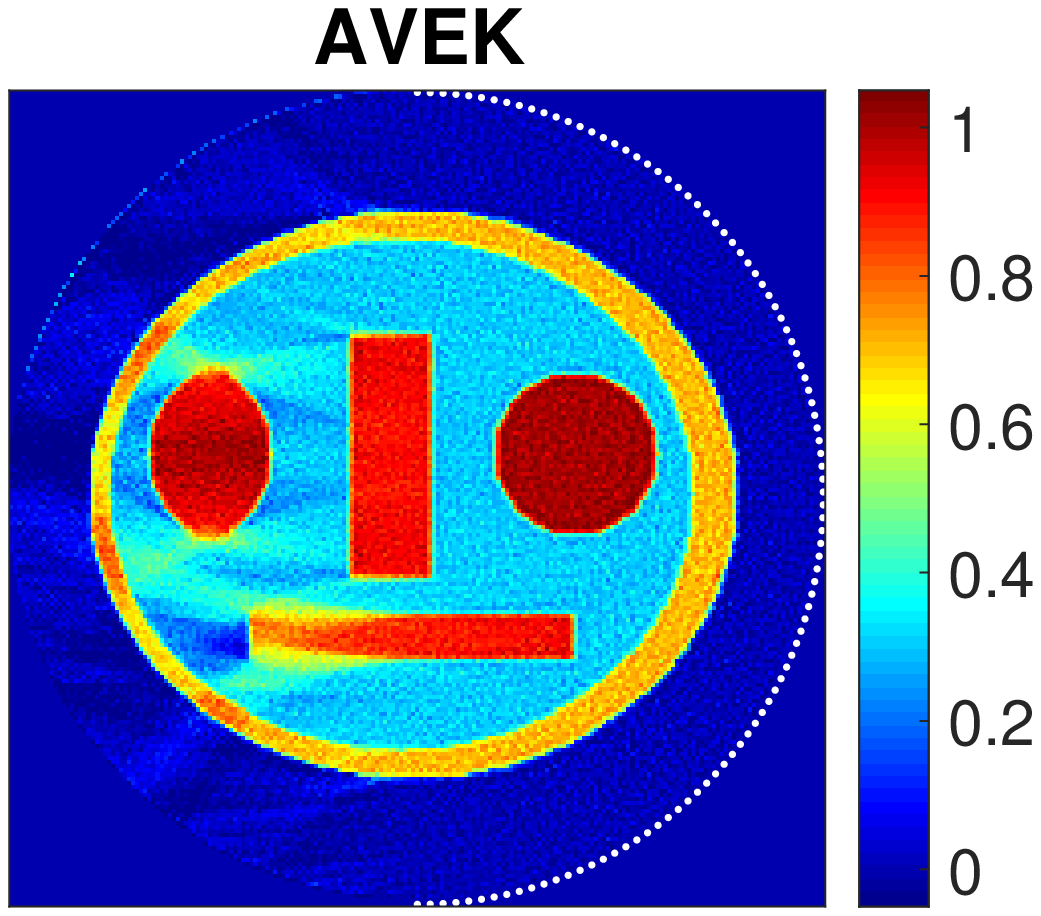}\hspace{0.15cm}
\includegraphics[width=0.28\textwidth]{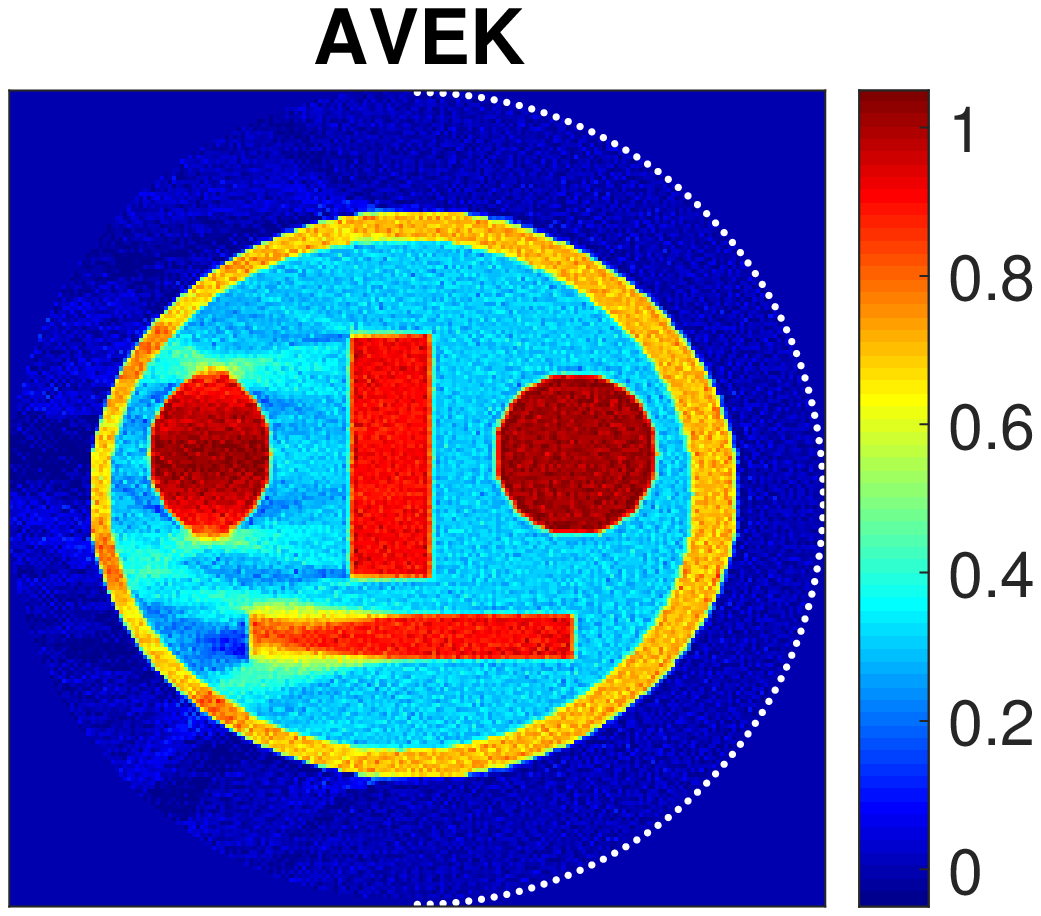}\hspace{0.15cm}
\includegraphics[width=0.28\textwidth]{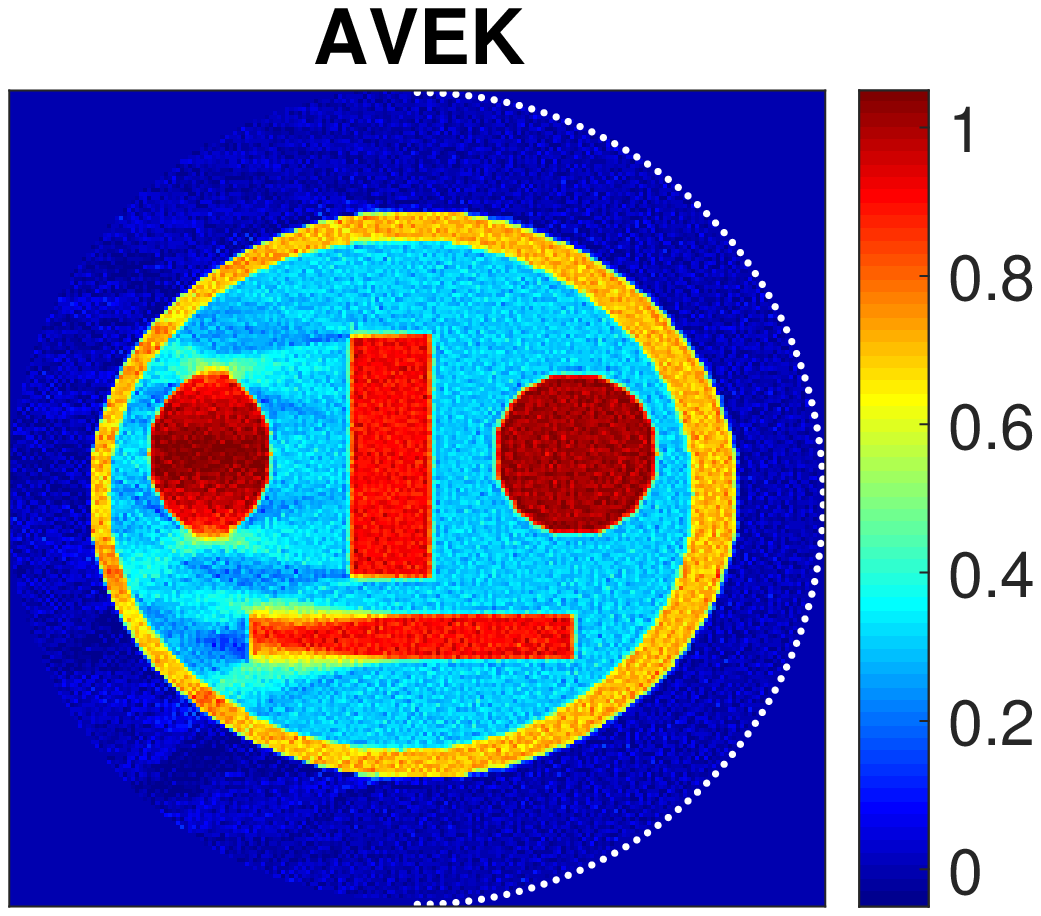}
\caption{{Reconstructions by Landweber, Kaczmarz and AVEK with proper choice of step sizes from exact data after 10 cycles (left column), 20 cycles (center column) and 80 cycles (right column).}}
\label{fig:exact}
\end{figure}

\subsubsection*{Results for noisy data}

We also tested the iterations on data $\gnum^\delta$ after adding $5\%$ noise. For that  purpose  added Gaussian white noise to $\gnum^\delta$ such  that  the resulting data  satisfy $\snorm{\gnum^\delta - \gnum}/ \snorm{\gnum} \simeq 0.05$. {Different step sizes are taken for each method} as in the exact data case and  $\tau_i$ are chosen in such a way that no iterations are skipped. The convergence behavior during the first 80 cycles using noisy data is shown in Figure~\ref{fig:varsn}. {The Kaczmarz method is the fastest, followed by the AVEK method, and the Landweber method is again the slowest. As in the exact data case, the AVEK iteration allows for way larger step sizes than the other two methods. Further, if step sizes are sufficiently small, the residuals $\snorm{\Mo \fnum_k^\delta  - \gnum^\delta }$ are decreasing for all methods, while the reconstruction errors $\snorm{ \fnum_k^\delta - \fnum}$ show the typical semi-convergence behavior for ill-posed problems. Interestingly, we point out that, in sharp contrast to the exact data case, iterations with small step sizes may outperform those with large step sizes. In noisy data case, slower convergence may provide smaller minimal reconstruction errors and further yields higher robustness in the choice of the iteration number as  regularization parameter.}

\begin{figure}[htb!]\centering
\includegraphics[width=0.32\textwidth]{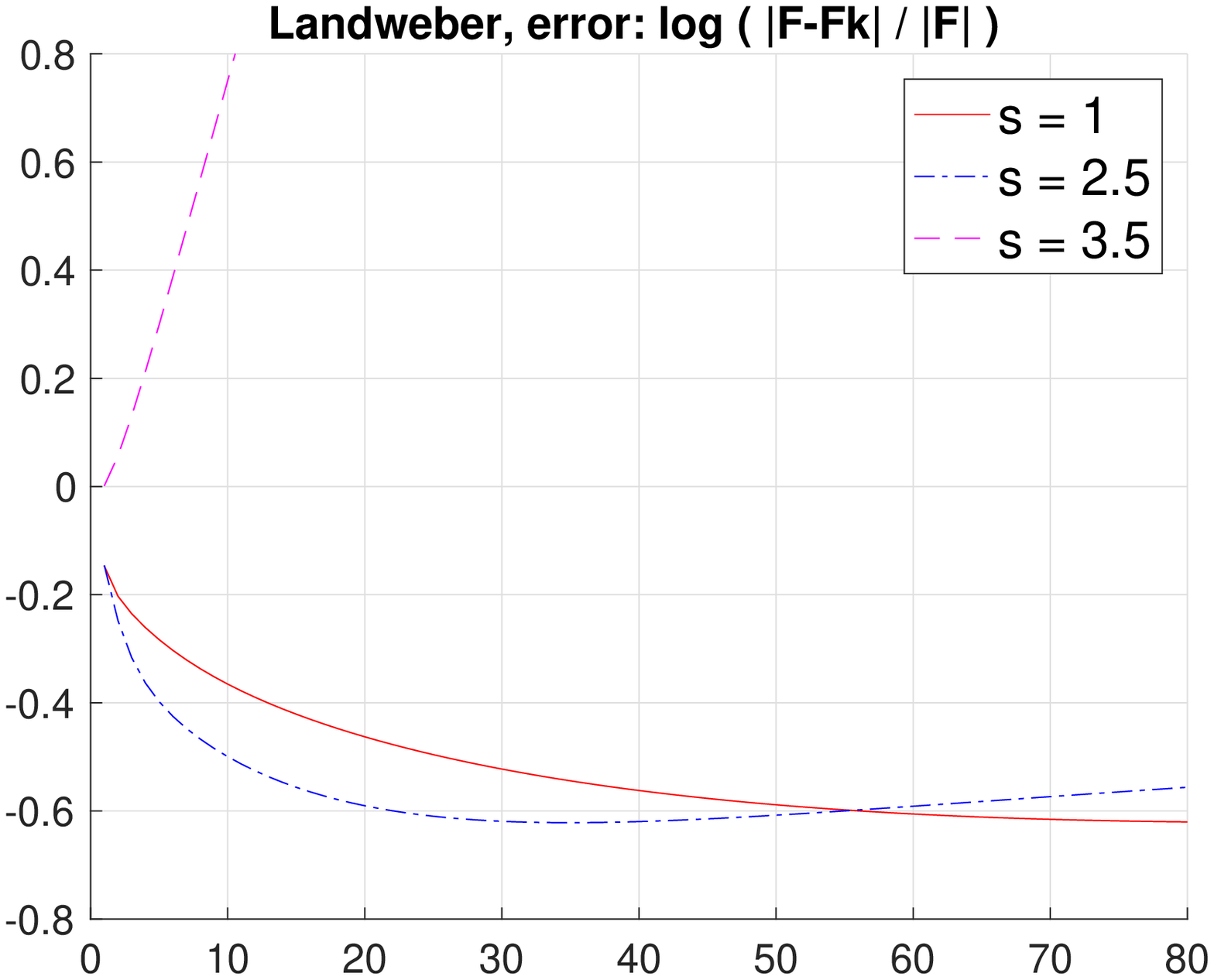}\hspace{0.1cm}
\includegraphics[width=0.32\textwidth]{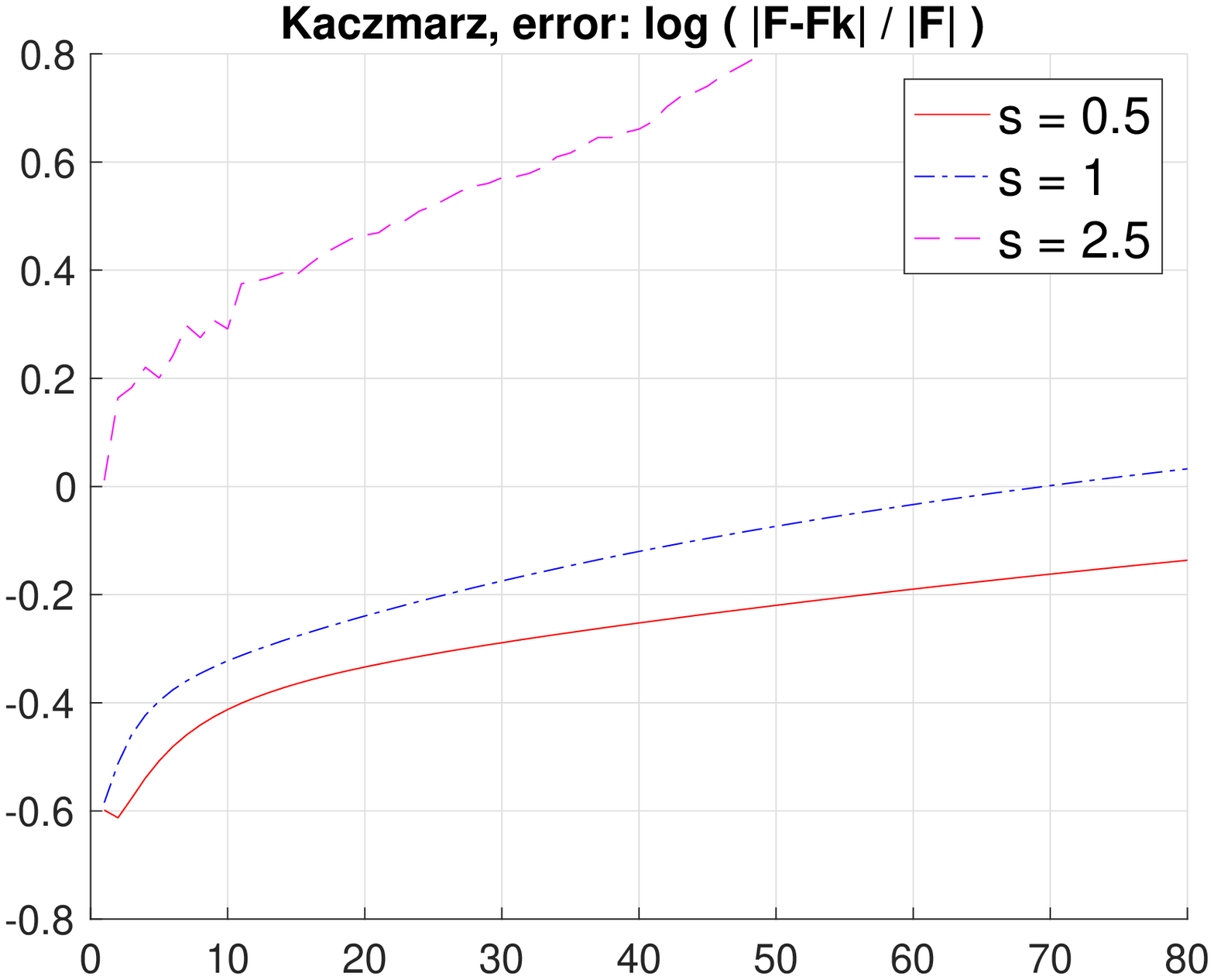}\hspace{0.1cm}
\includegraphics[width=0.32\textwidth]{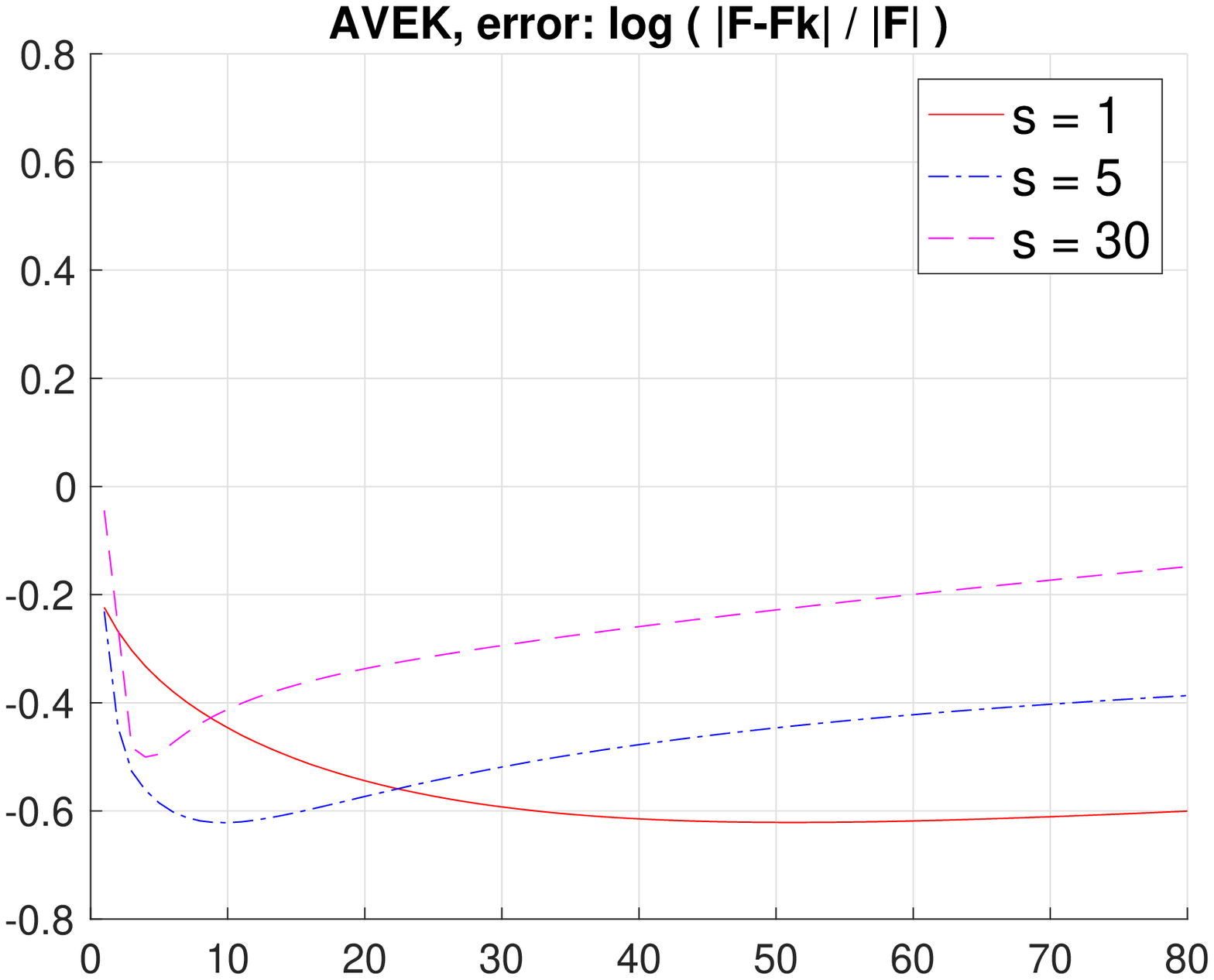}\vspace{0.3cm}\\
\includegraphics[width=0.32\textwidth]{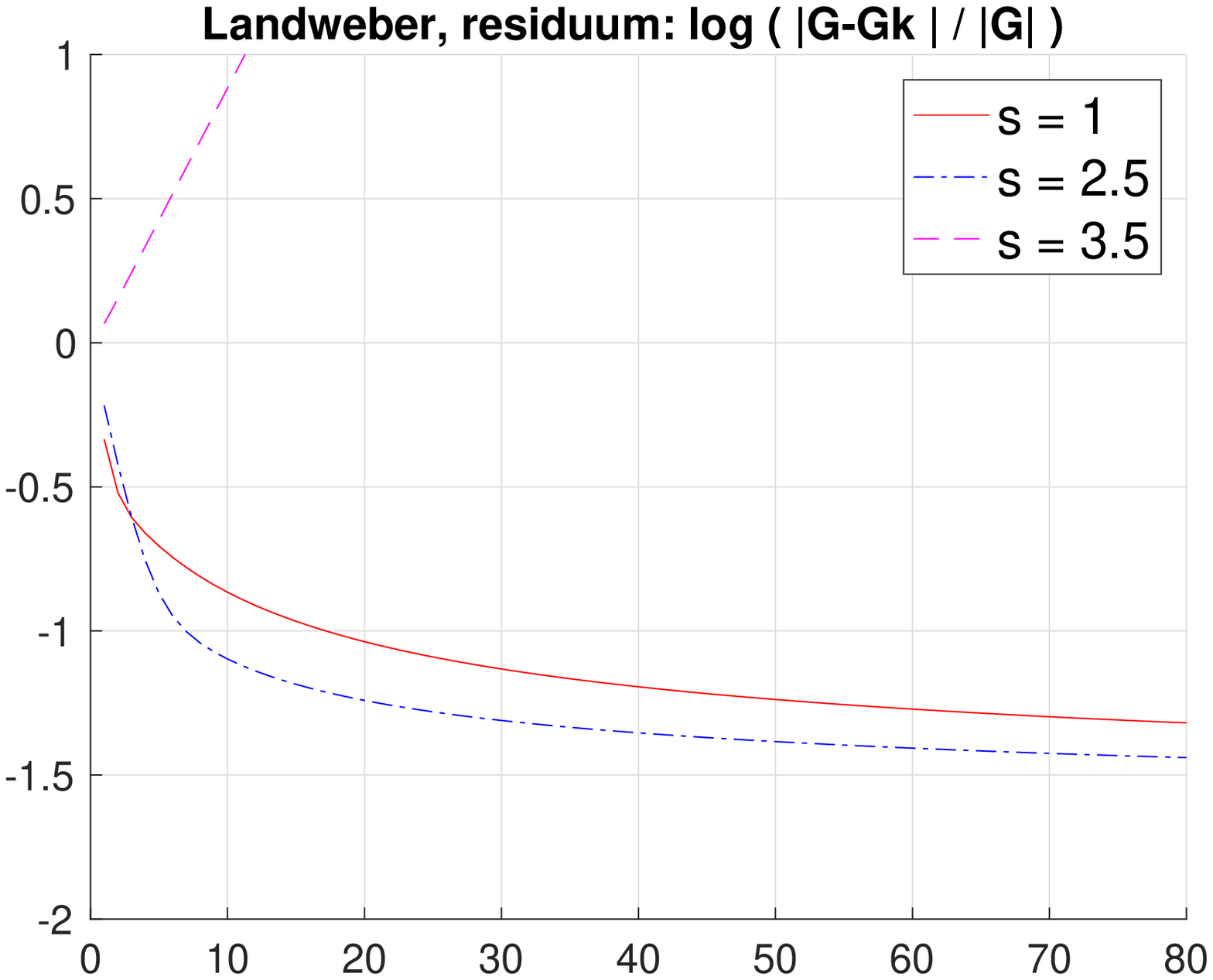}\hspace{0.1cm}
\includegraphics[width=0.32\textwidth]{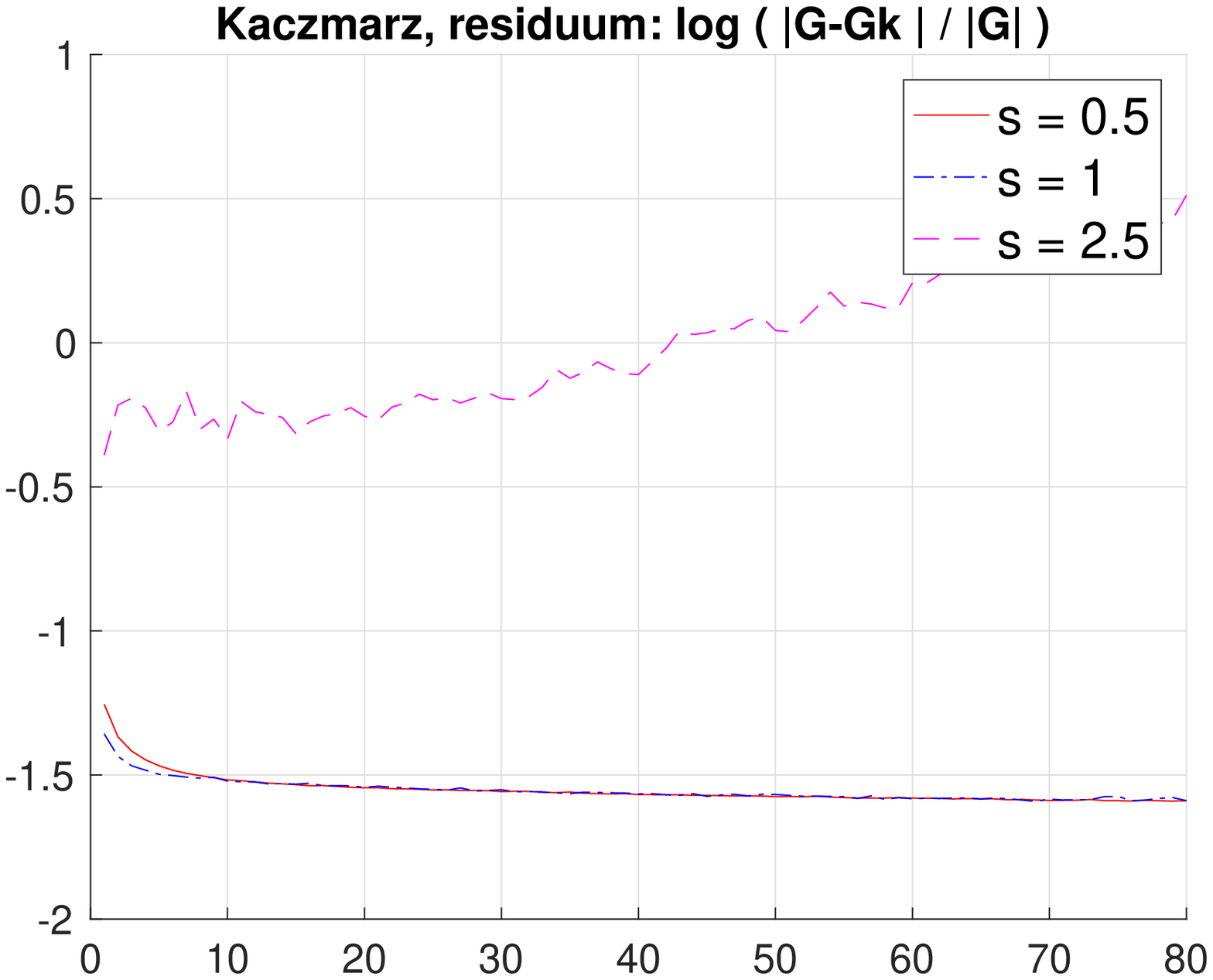}\hspace{0.1cm}
\includegraphics[width=0.32\textwidth]{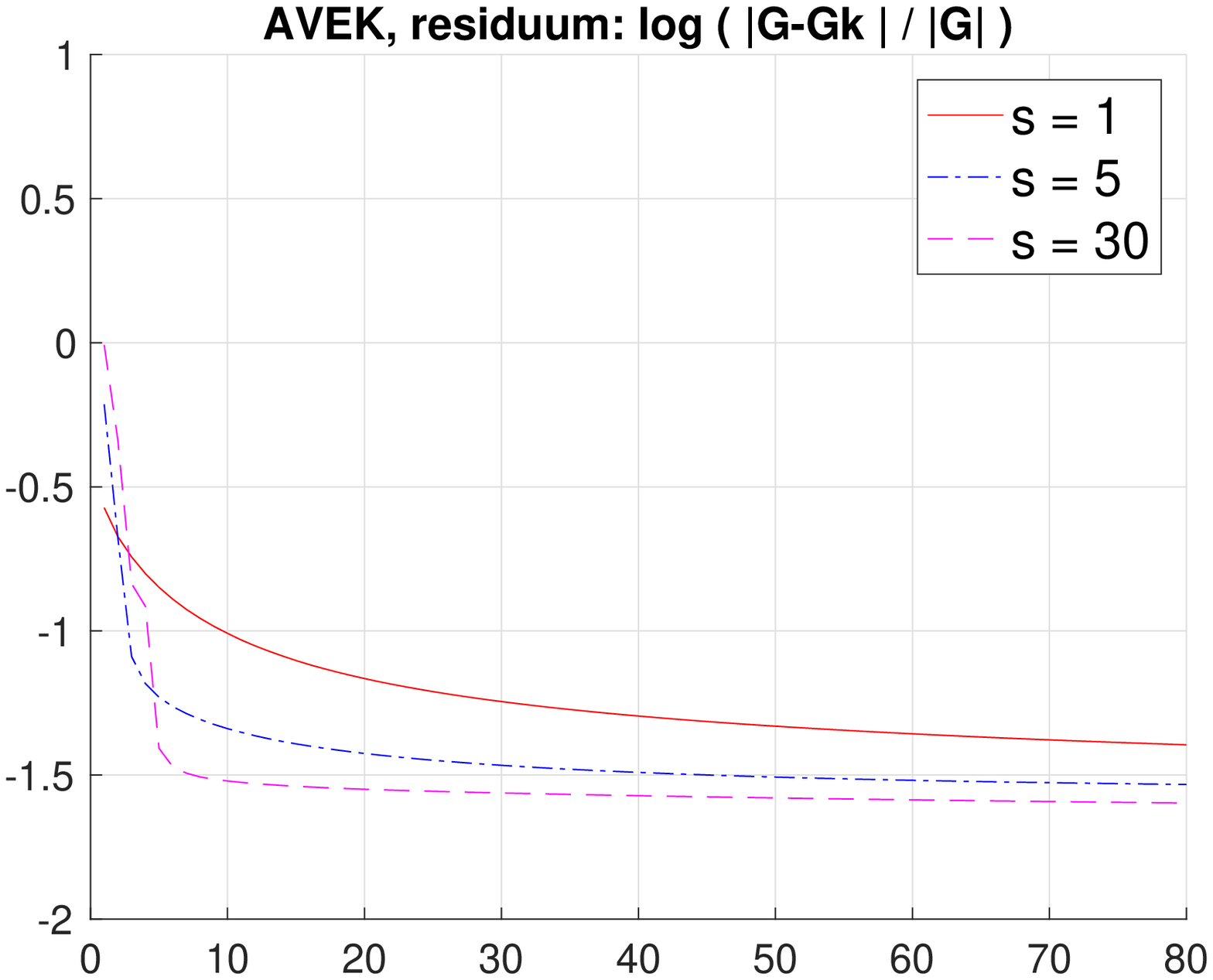}
\caption{{Residuum and relative reconstruction  error (after taking logarithm to basis 10) of Landweber, Kaczmarz and AVEK with different step sizes for noisy data during the first 80 cycles.}}
\label{fig:varsn}
\end{figure}

For comparison of visual quality, we choose  the empirically  best step sizes for all methods; namely, $s_{\rm LW} = 2.5$ for the Landweber iteration, $s_{\rm K} = 1$ for the Kaczmarz iteration and $s_{\rm AVEK} = 5$ for the AVEK iteration. The minimal $L^2$-reconstruction errors have been obtained after {35} iterations for the Landweber iteration, after 2 cycles for the Kaczmarz iteration, and after {10} cycles for the AVEK. The corresponding relative reconstruction errors $\snorm{ \fnum_k^\delta - \fnum} / \snorm{\fnum}$ are {$0.0595$} for the Kaczmarz method and $0.0571$ for the Landweber as well as the AVEK method. The Landweber and the AVEK method therefore slightly outperform the Kaczmarz method in terms of the minimal reconstruction error. Reconstruction results after 2, 10 and {35} iterations are shown in Figure~\ref{fig:noisy}. Further, through extensive simulations (not shown here), we find that the choice $s_{\rm AVEK} = 5$ for the AVEK method is robust to different noise levels, which is thus recommended as the default step size for noisy data in practice. 

In summary, from the simulations with exact and with noisy data, we conclude that the AVEK method is as comparably fast as the Kaczmarz method, and is meanwhile surprisingly stable with respect to the choice of step sizes. Such  favorable properties are also observed for other data sets and are highly valuable in a great many of applications.

\begin{figure}[htb!]\centering
\includegraphics[width=0.28\textwidth]{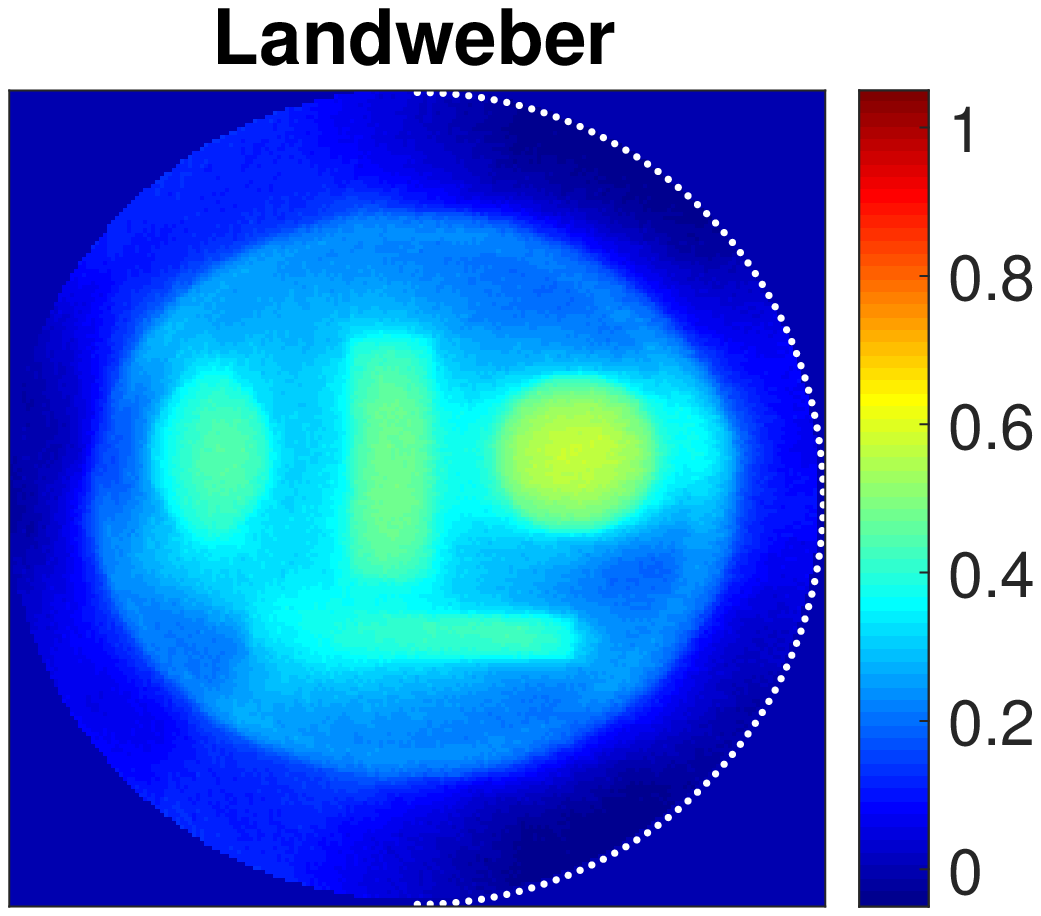}\hspace{0.15cm}
\includegraphics[width=0.28\textwidth]{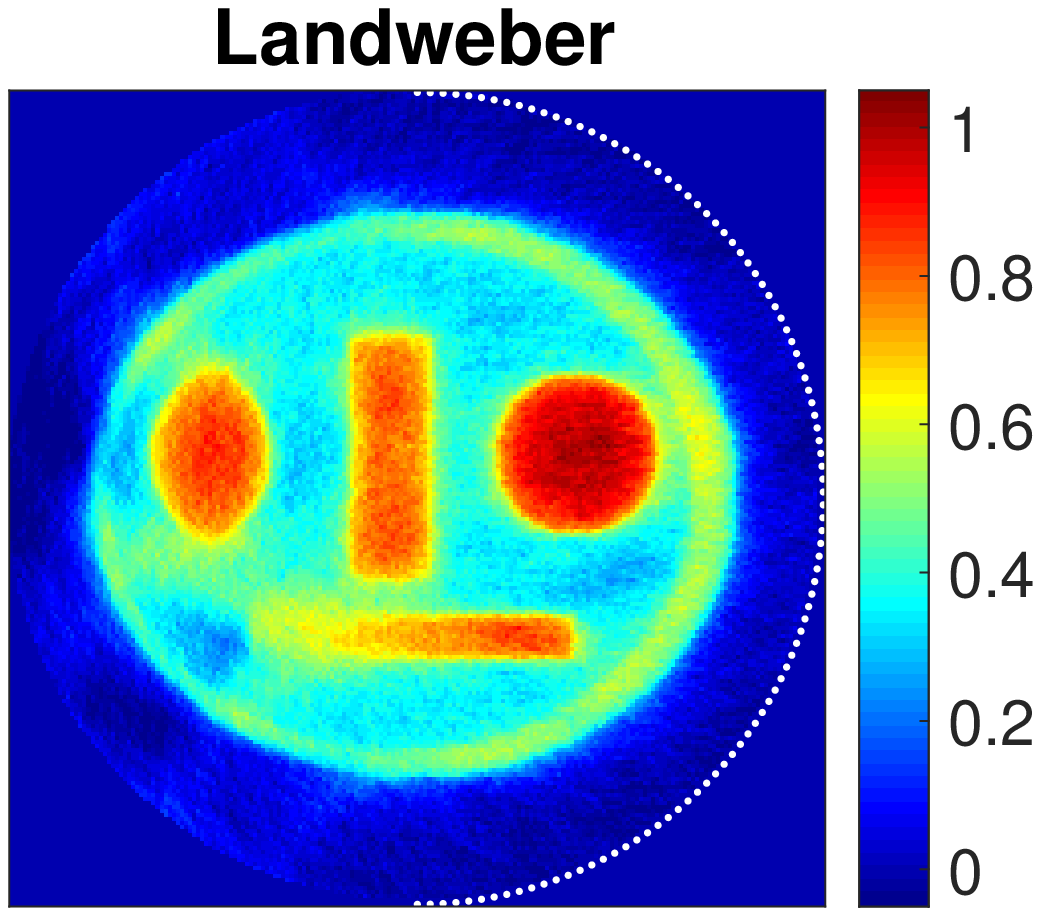}\hspace{0.15cm}
\includegraphics[width=0.28\textwidth]{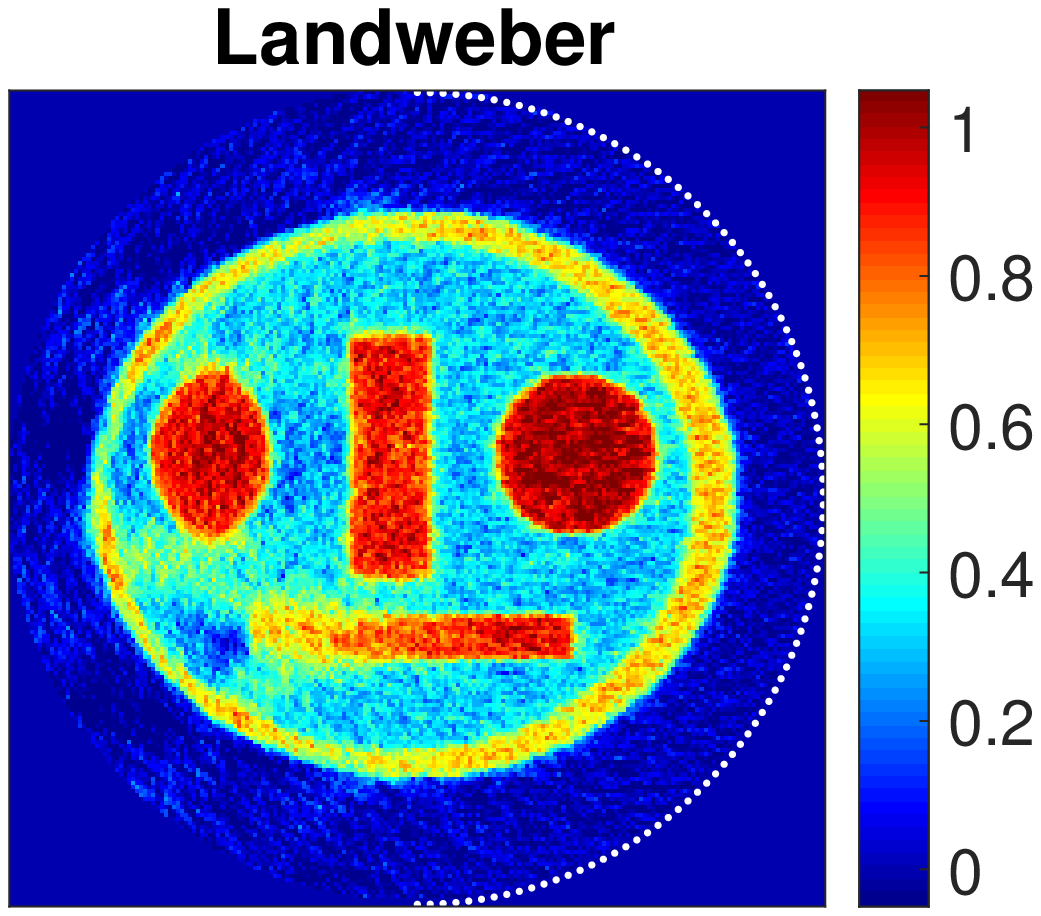} \vspace{0.15cm}\\
\includegraphics[width=0.28\textwidth]{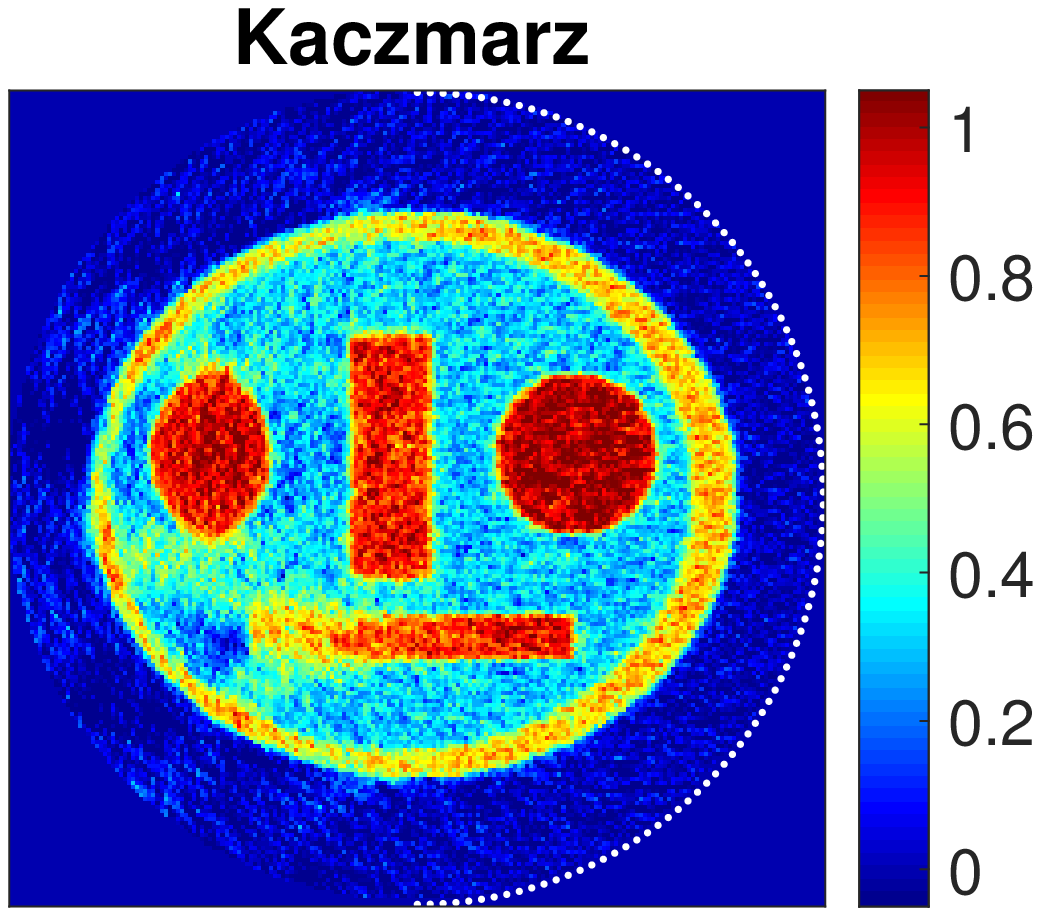}\hspace{0.15cm}
\includegraphics[width=0.28\textwidth]{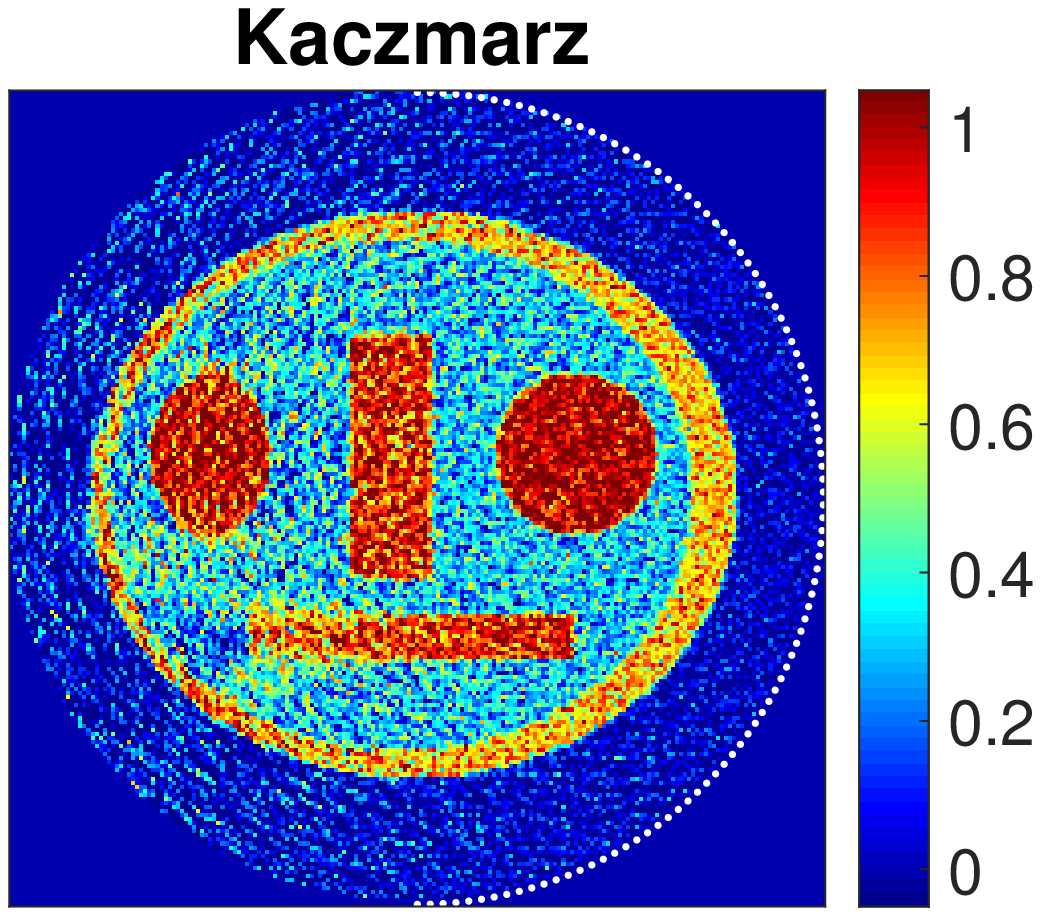}\hspace{0.15cm}
\includegraphics[width=0.28\textwidth]{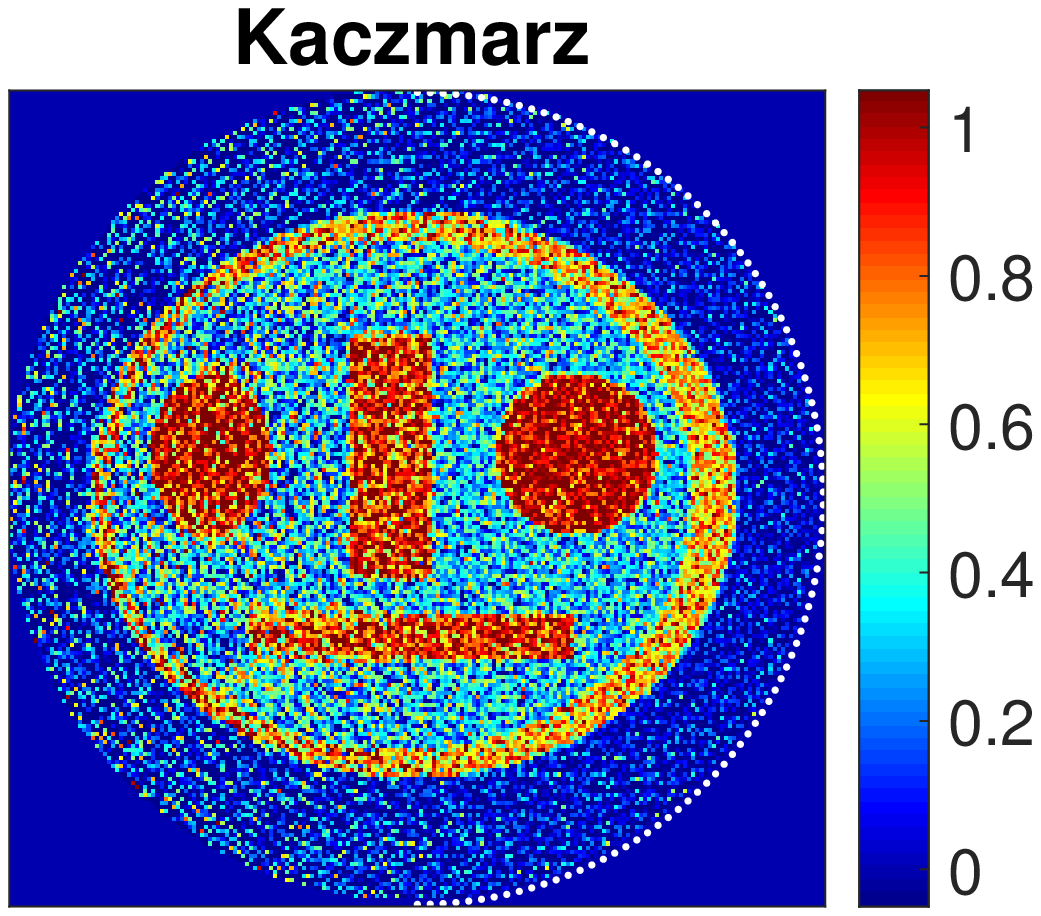} \vspace{0.15cm} \\
\includegraphics[width=0.28\textwidth]{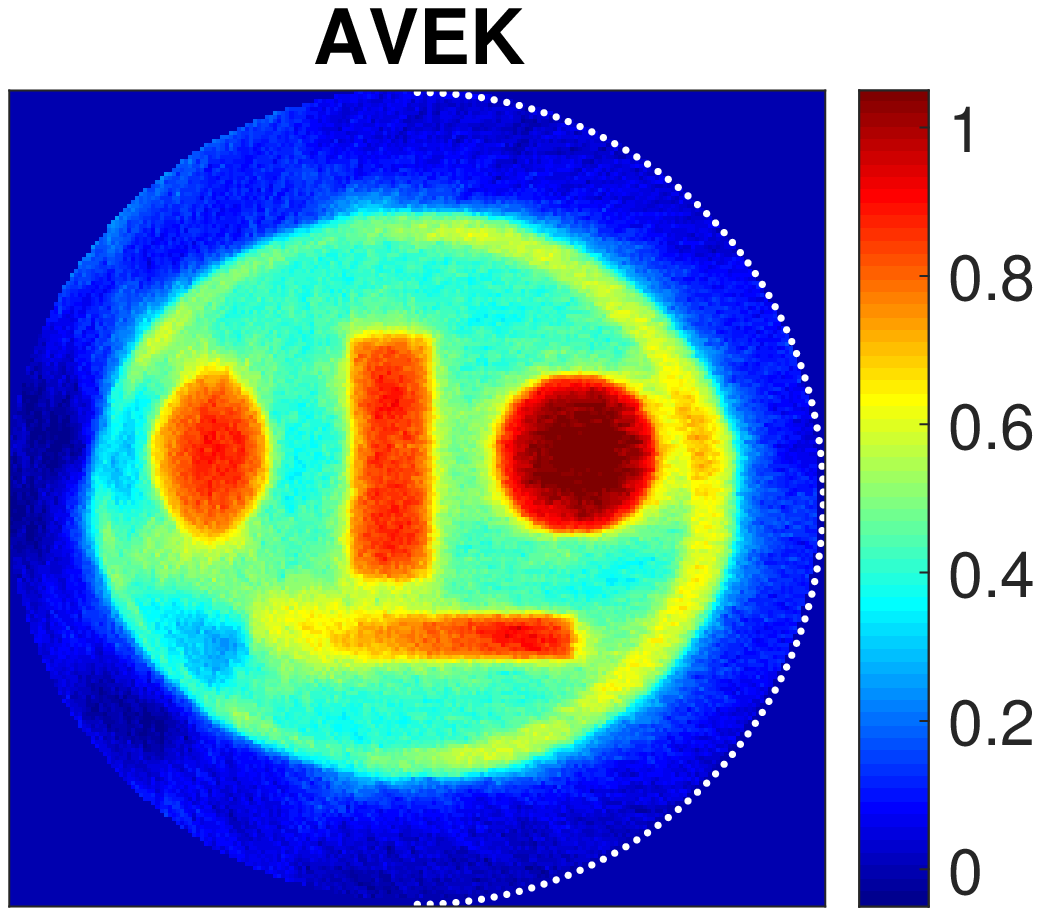}\hspace{0.15cm}
\includegraphics[width=0.28\textwidth]{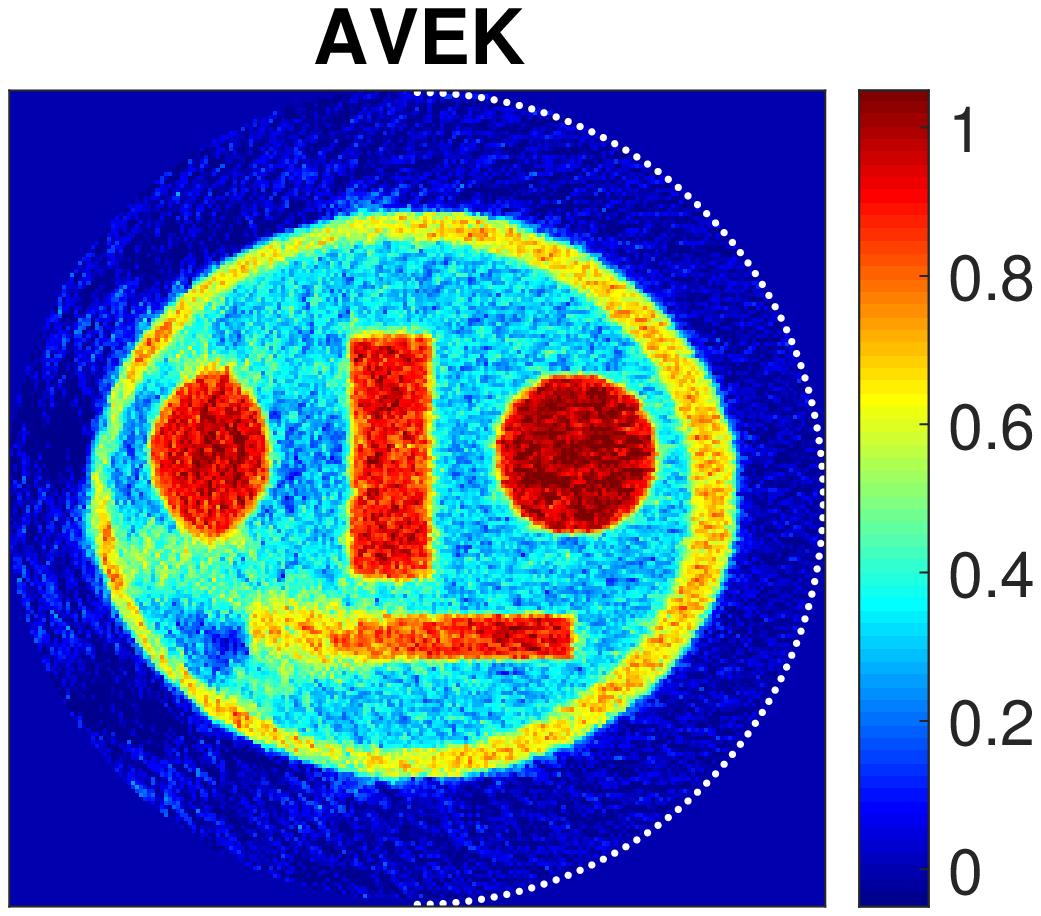}\hspace{0.15cm}
\includegraphics[width=0.28\textwidth]{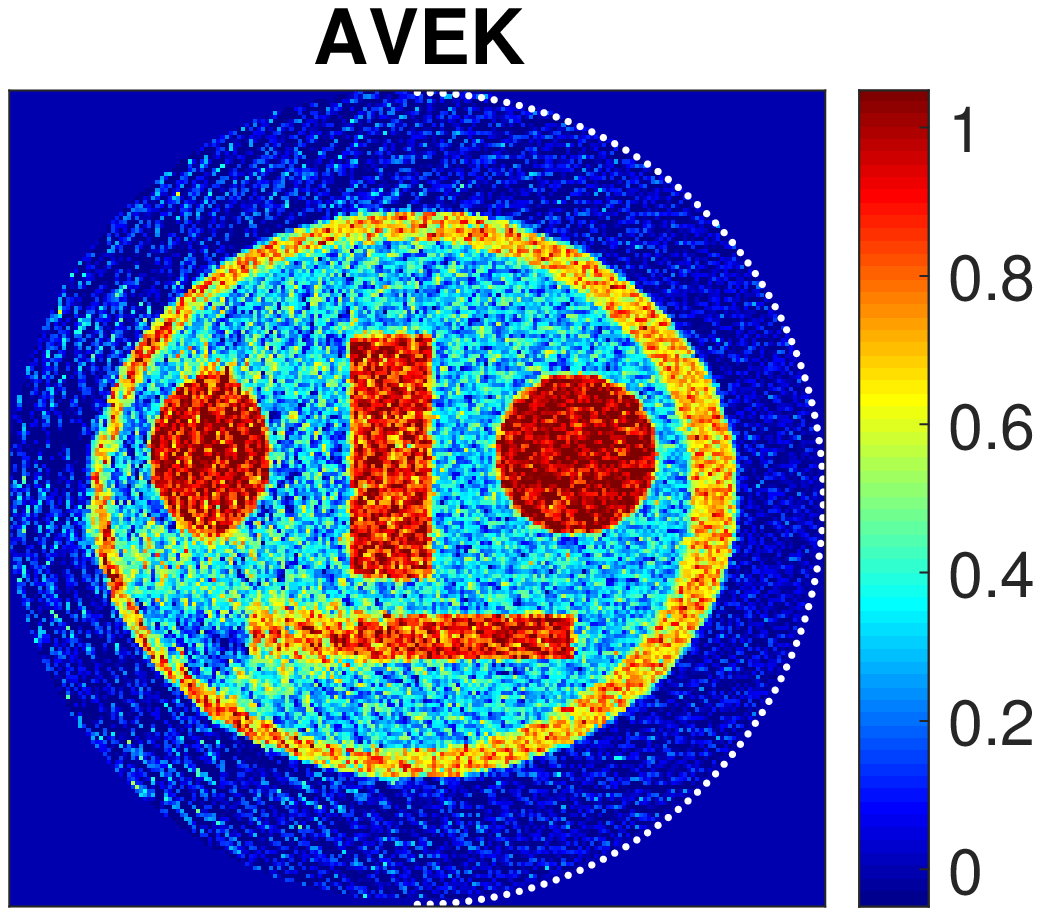}
\caption{{Reconstructions by Landweber, Kaczmarz and AVEK with proper choice of step sizes from noisy data after 2 cycles (left), 10 cycles (center) and 35 cycles (right).}}
\label{fig:noisy}
\end{figure}

\subsection{{Comparison with other methods}}

We further investigate the performance of the proposed AVEK method by comparing it with state-of-the-art accelerated versions of the Landweber and the Kaczmarz method proposed in~\cite{NAE17}; compare also \cite{dos1987parallel,nikazad2013acceleration,nikazad2016convergence}. These accelerated methods take the same forms as the basic Landweber and the Kaczmarz method, with the only difference lying in the choice of step sizes; they select step sizes at each iteration via error minimizing relaxation (EMR) strategies. More precisely, the step size for the $k$-th iterative update is chosen to minimize $\sinner{\fsignal^\delta_k  - \fsignal}{(\Mo^*\Mo)^s(\fsignal^\delta_k  - \fsignal)}$ in case of the Landweber method, and to minimize $\sinner{\fsignal^\delta_k  - \fsignal}{(\Mo_{[k]}^*\Mo_{[k]})^s(\fsignal^\delta_k  - \fsignal)}$  in case of the Kaczmarz method, for fixed $s \in \N_0$ (see~\cite{NAE17} for  details). We denote the resulting accelerated versions by Landweber-EMR and Kaczmarz-EMR, respectively. Additionally, we consider the incremental aggregated  gradient (IAG) method~\cite{blatt2007convergent}, being closely related to the AVEK method, which is defined as
\[
\fsignal_{k+1}^{\delta}  =
	 \fsignal_k^{\delta} - \frac{s_k}{n} \sum_{\ell = k-n+1}^k \Mo_{\ii{\ell}}^*
\skl{ \Mo_{\ii{\ell}}(\fsignal^\delta_\ell) - \gdata^\delta_{\ii{\ell}}} \qquad \text{ for } k \ge n.
\]
See~\eqref{eq:ig} for the definition in case of general (possibly nonlinear) problems. We consider the same setting as in Section~\ref{ss:num}. In numerical simulations, parameter $s$ is set to $0$ or $1$ for the Landweber-EMR and the Kaczmarz-EMR method; the step sizes for the AVEK method are chosen the same as earlier (i.e.~$s_{\rm AVEK} = 30$ for exact data and $s_{\rm AVEK} = 5$ for noisy data); the step size for IAG is chosen as $s_{\rm IAG} = 0.08$ for exact data and $s_{\rm IAG} = 0.06$ for noisy data, which leads to the best empirical performance.  Moreover, for all methods the equations have been randomly rearranged prior to each cycle, which empirically accelerates the convergence.

\begin{figure}[htb!]\centering
\includegraphics[width=0.4\textwidth]{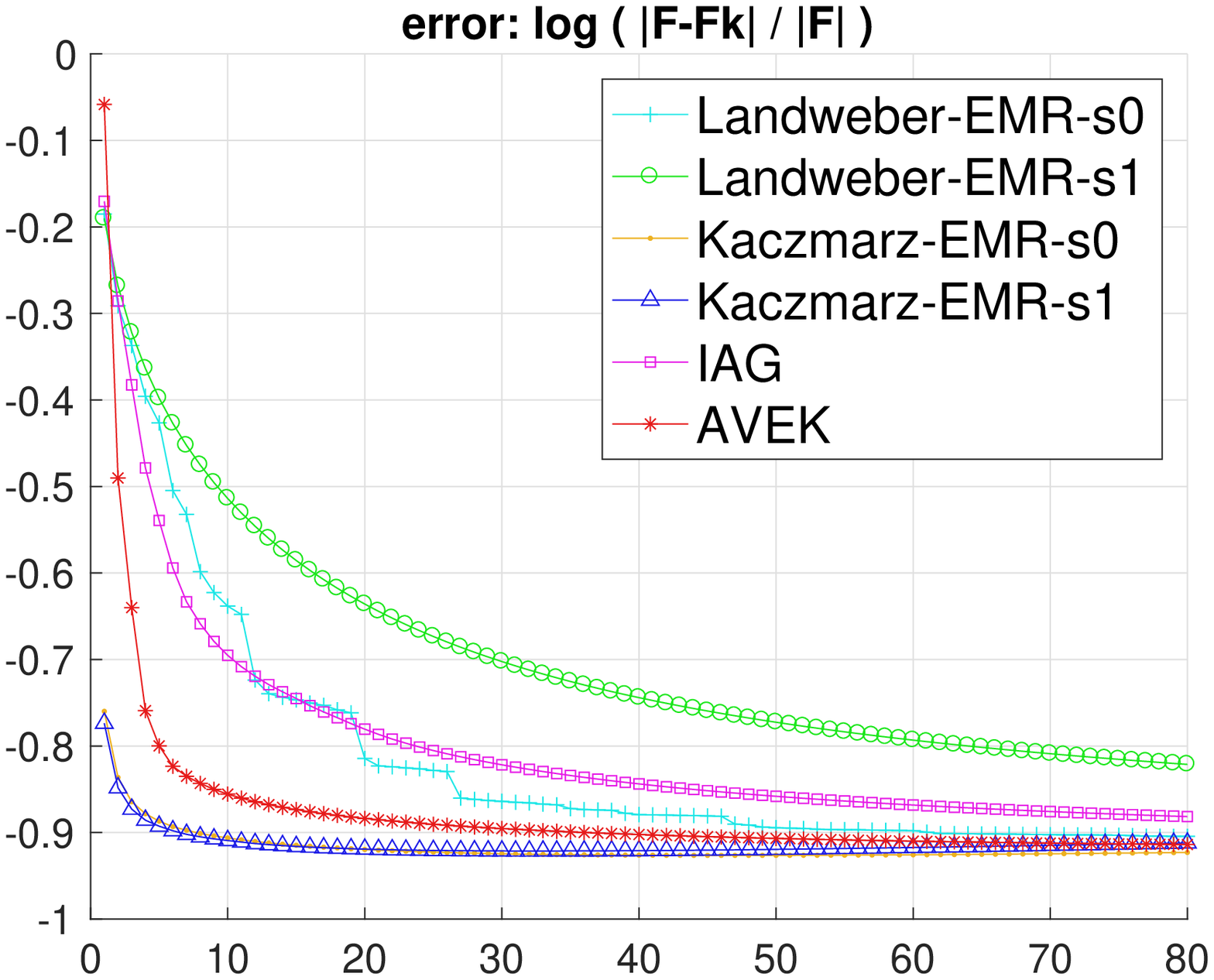}\hspace{0.8cm}
\includegraphics[width=0.4\textwidth]{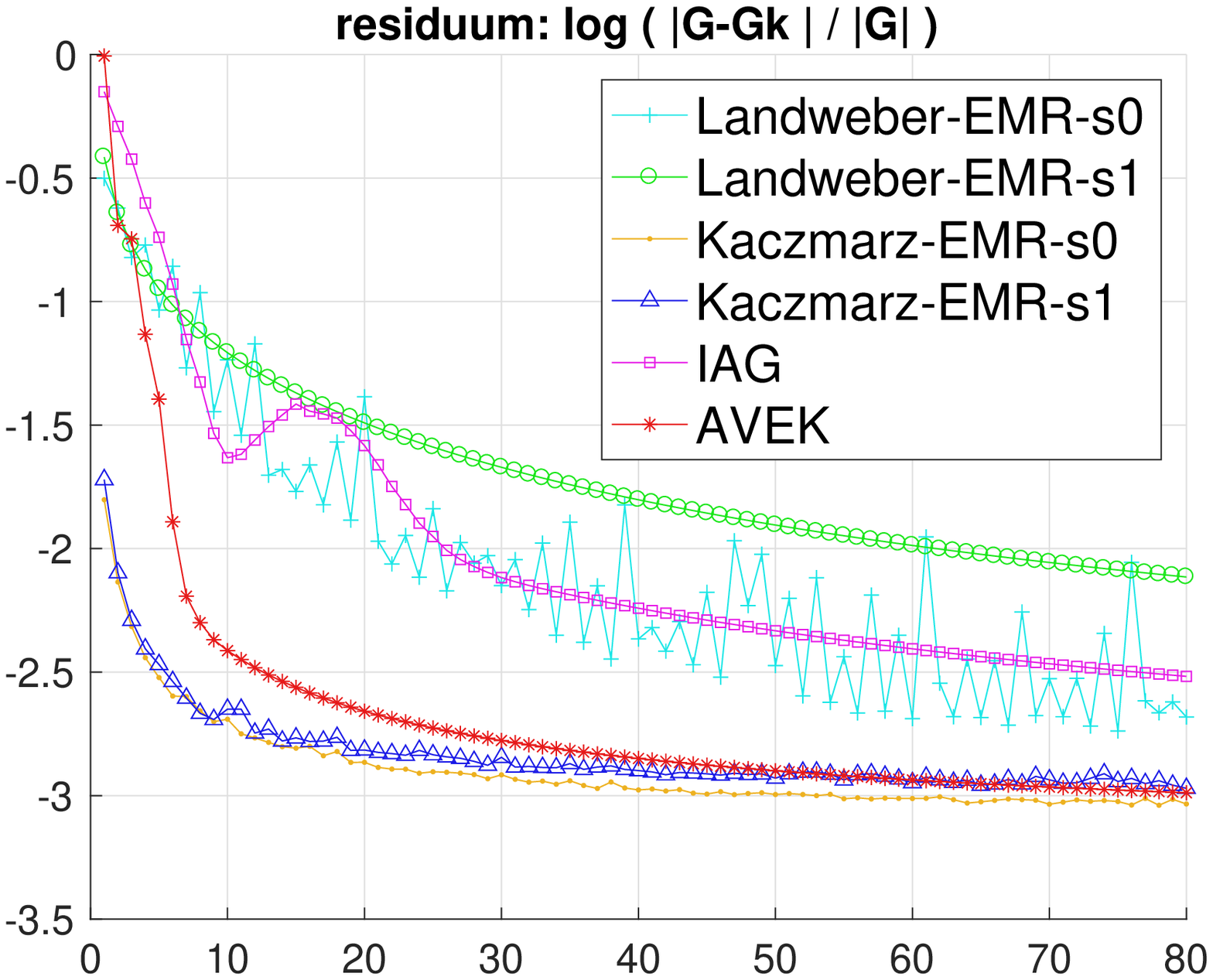}
\caption{{Residuum and relative reconstruction error (after taking logarithm to basis 10) of Landweber-EMR ($s = 0$ or $1$), Kaczmarz-EMR ($s = 0$ or $1$), IAG and AVEK for exact data during the first 80 cycles.}}
\label{fig:ccExact}
\end{figure}

\begin{figure}[htb!]\centering
\includegraphics[width=0.28\textwidth]{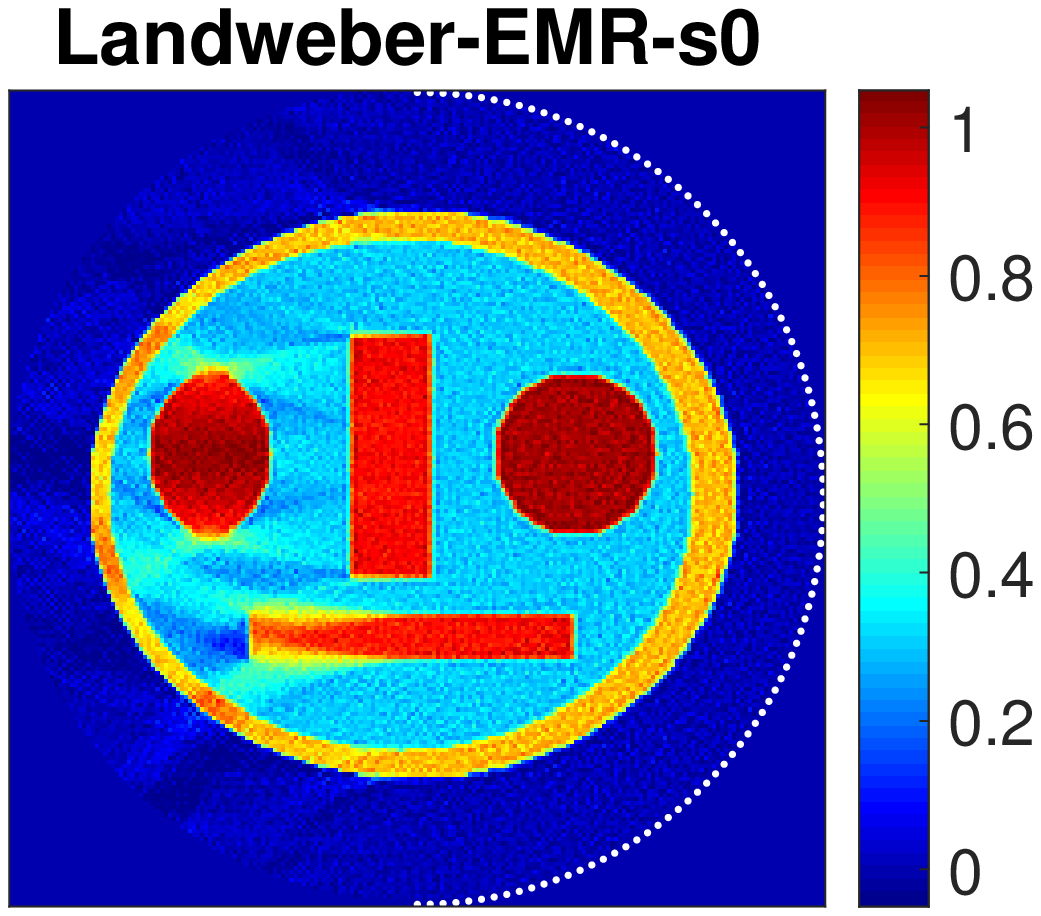}\hspace{0.15cm}
\includegraphics[width=0.28\textwidth]{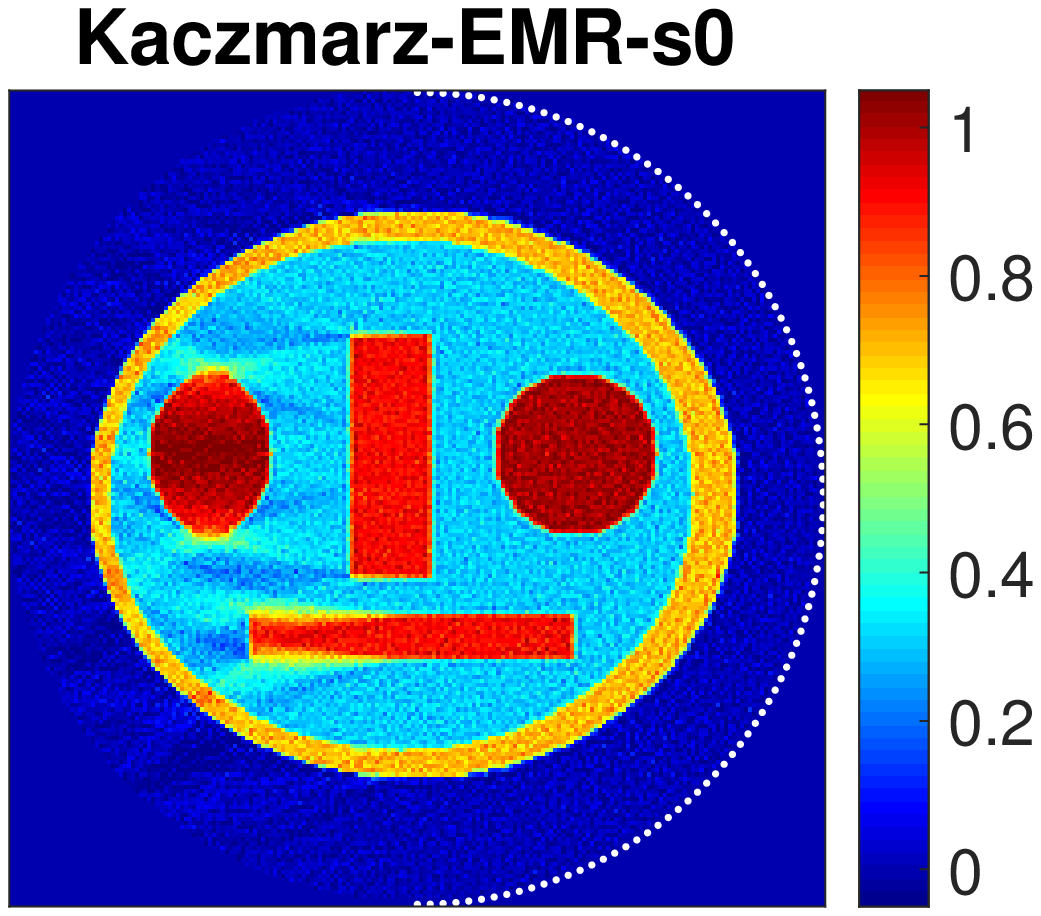}\hspace{0.15cm}
\includegraphics[width=0.28\textwidth]{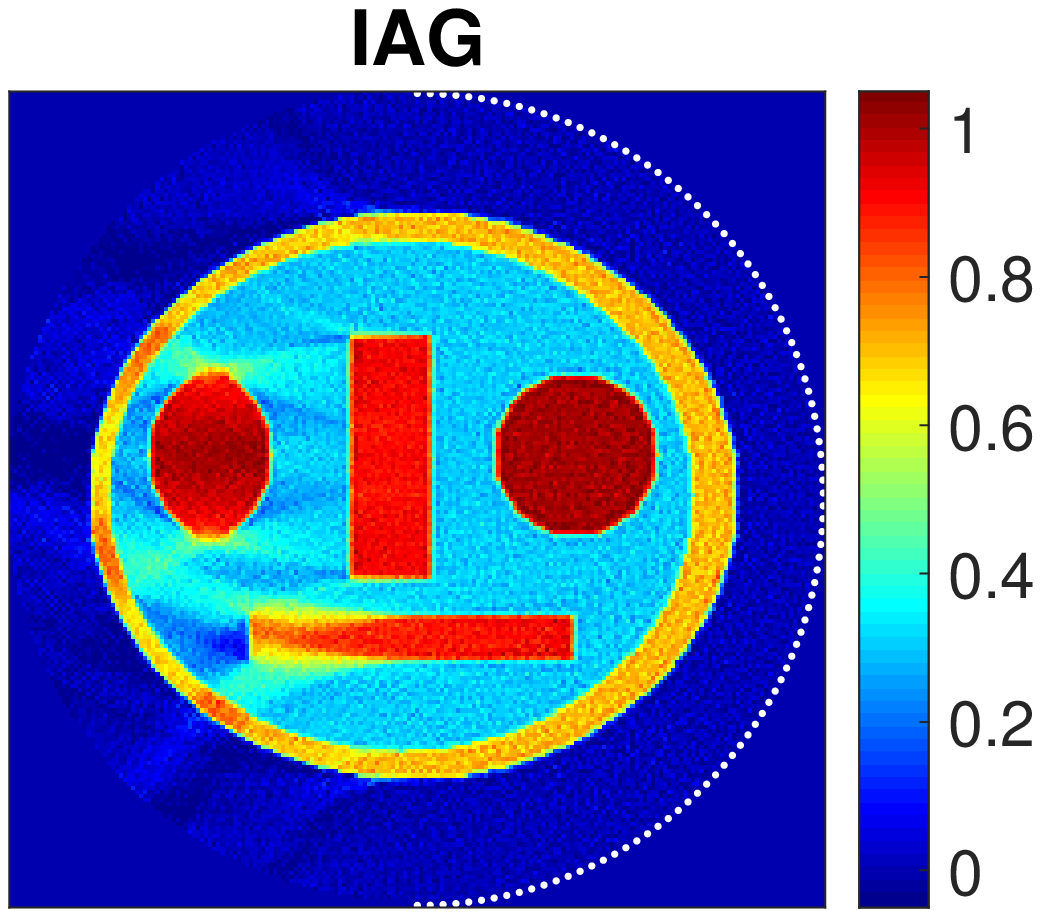} \vspace{0.15cm}\\
\includegraphics[width=0.28\textwidth]{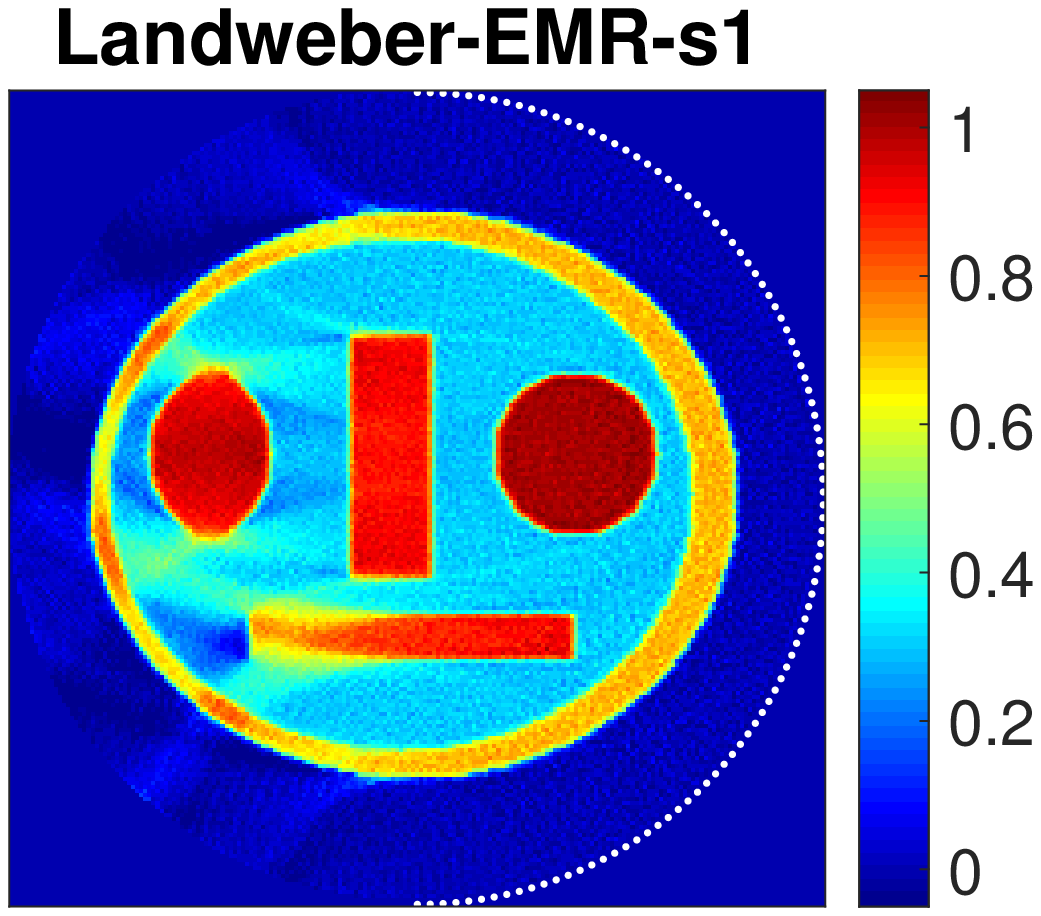}\hspace{0.15cm}
\includegraphics[width=0.28\textwidth]{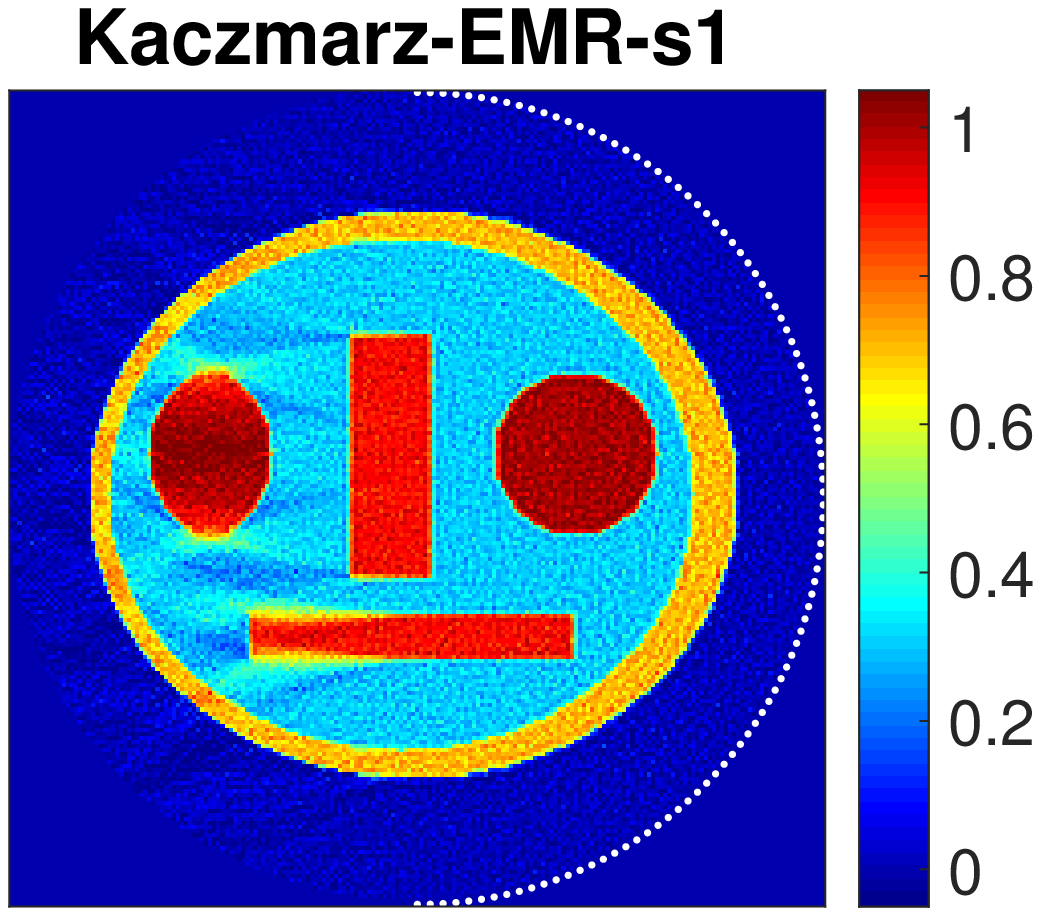}\hspace{0.15cm}
\includegraphics[width=0.28\textwidth]{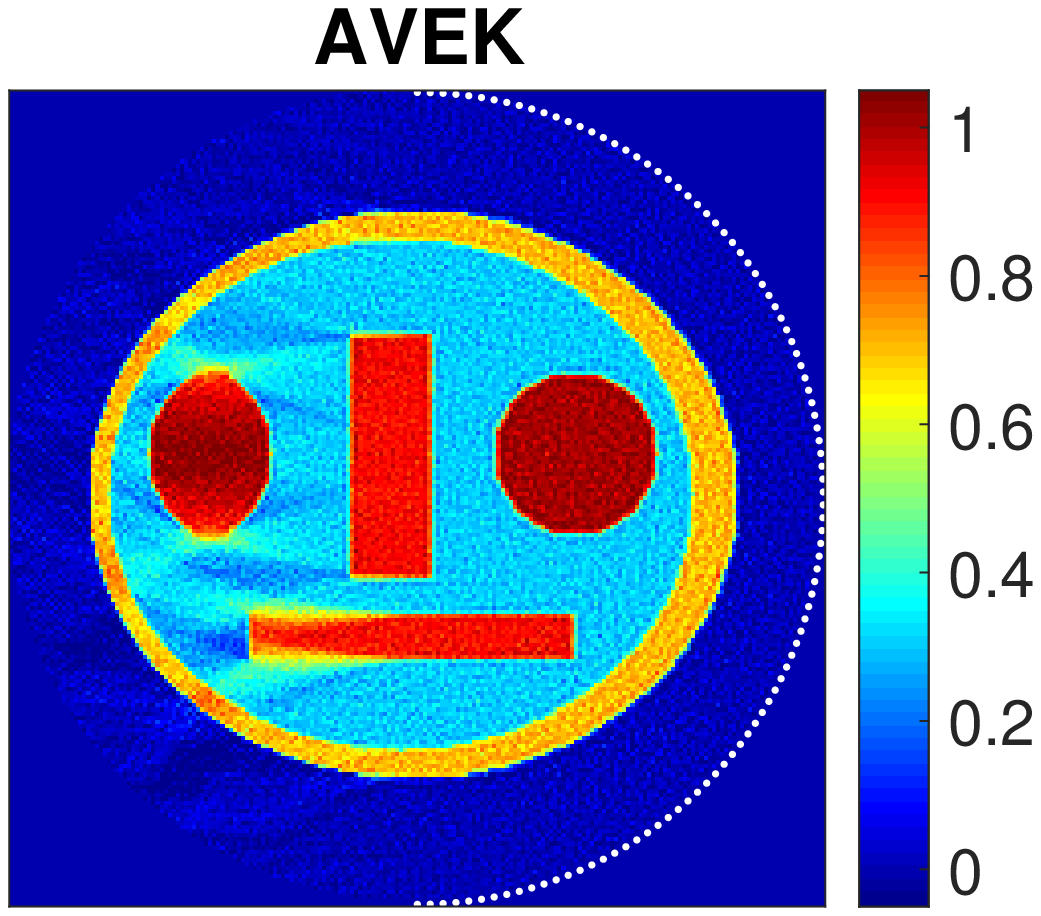} \vspace{0.15cm}
\caption{{Reconstructions by Landweber-EMR ($s = 0$ or $1$), Kaczmarz-EMR ($s = 0$ and $1$), IAG and AVEK from exact data after 80 cycles.}}
\label{fig:rcExact}
\end{figure}

For exact data the comparison of convergence behavior  is illustrated in Figure~\ref{fig:ccExact}. It shows that the two Kaczmarz-EMR methods are the fastest, closely followed by the AVEK, then the IAG and the Landweber-EMR ($s = 1$), while the Landweber-EMR ($s = 0$) is the slowest. Both the AVEK and the Kaczmarz-EMR methods obtain the smallest relative reconstruction errors and the smallest residuals among all methods. By comparing with Figure~\ref{fig:vars}, one notes that the EMR strategies indeed accelerate the original Landweber and Kaczmarz methods for the circular Radon transform in terms of convergence rates. Further, notice that  AVEK  converges faster than IAG. Figure~\ref{fig:rcExact} gives a visual inspection of the convergence behavior for all methods.

\begin{figure}[htb!]\centering
\includegraphics[width=0.4\textwidth]{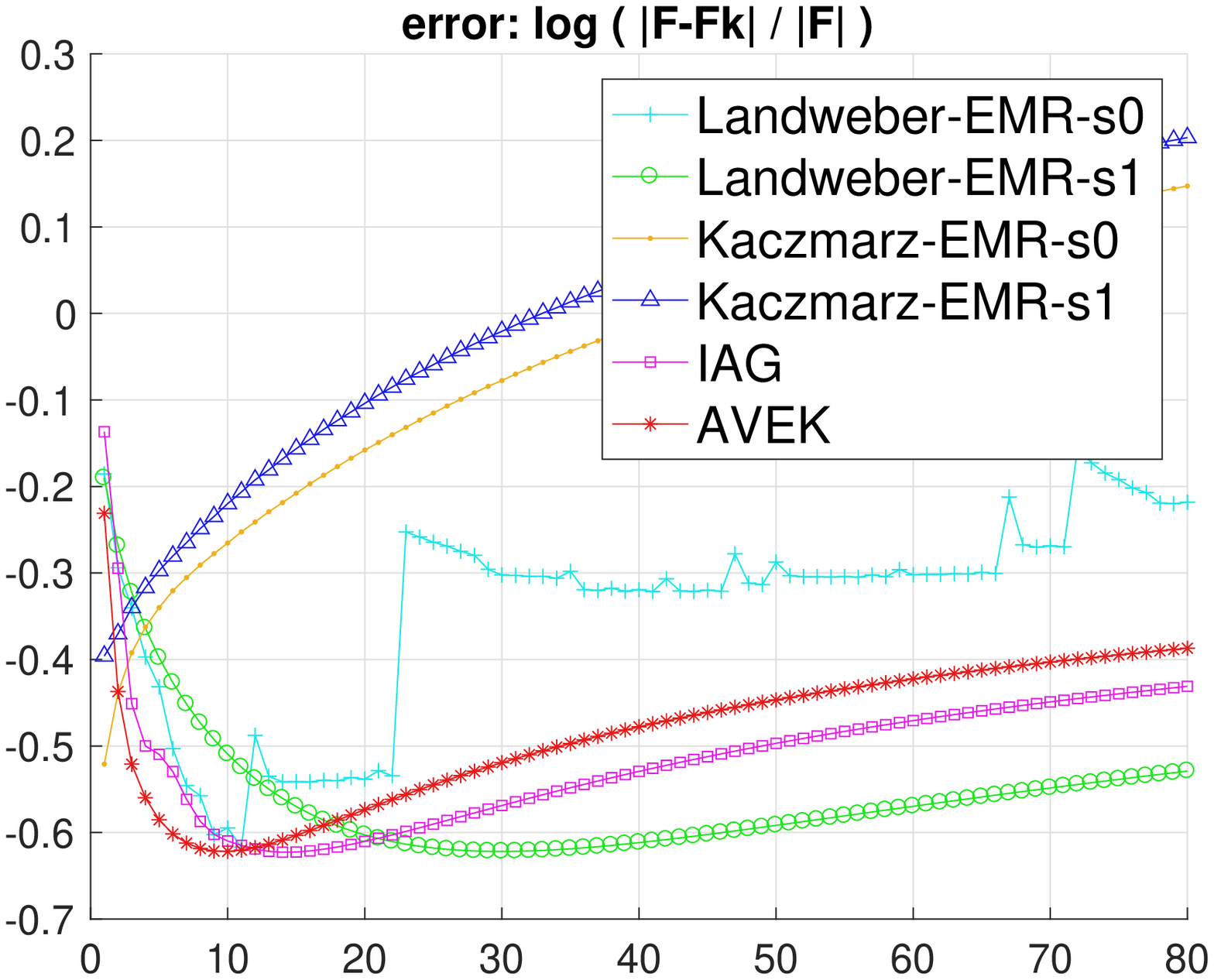}\hspace{0.8cm}
\includegraphics[width=0.4\textwidth]{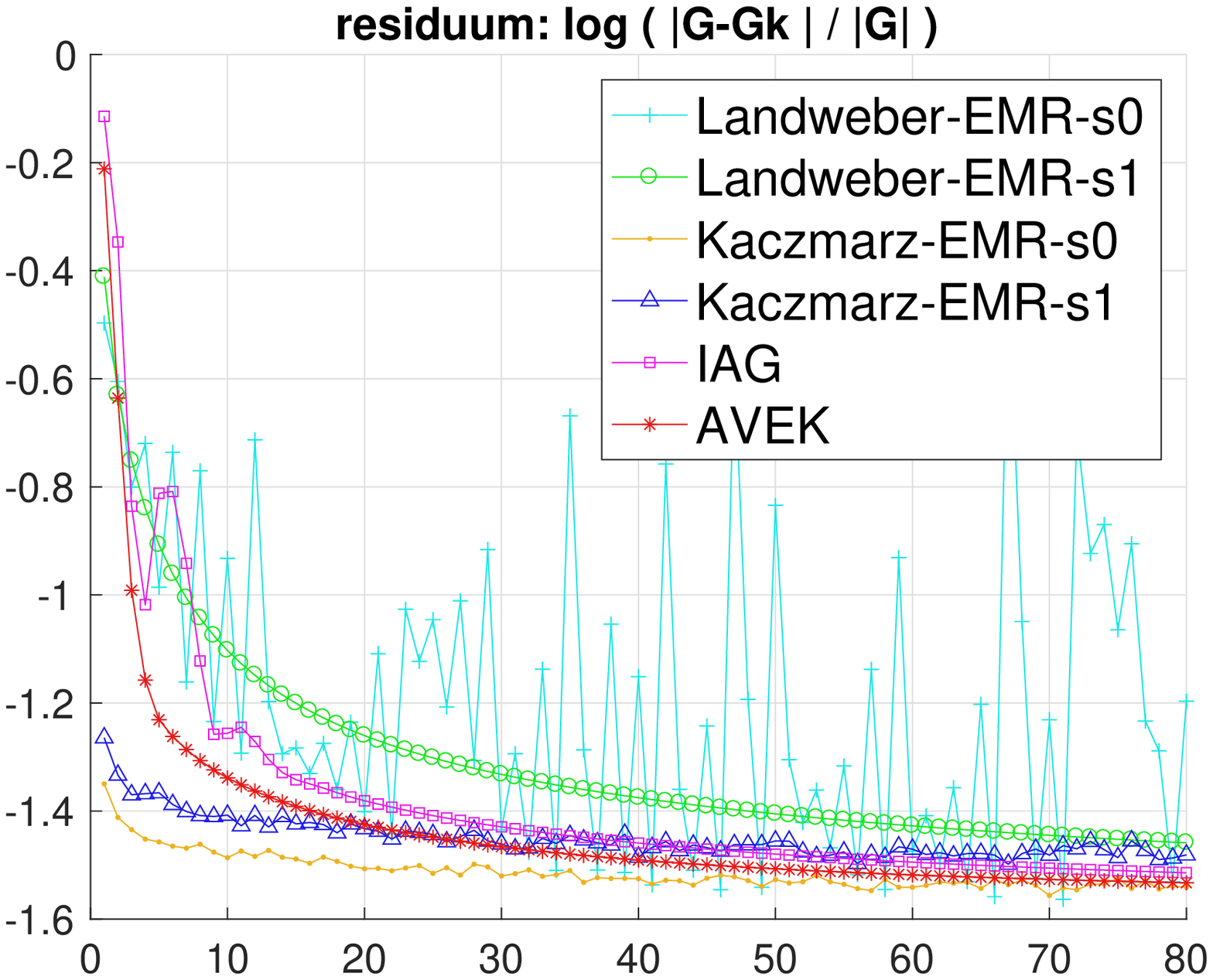}
\caption{{Residuum and relative reconstruction error (after taking logarithm to basis 10) of Landweber-EMR ($s = 0$ or $1$), Kaczmarz-EMR ($s = 0$ or $1$), IAG and AVEK for noisy data during the first 80 cycles.}}
\label{fig:ccNoisy}
\end{figure}

\begin{figure}[htb!]\centering
\includegraphics[width=0.28\textwidth]{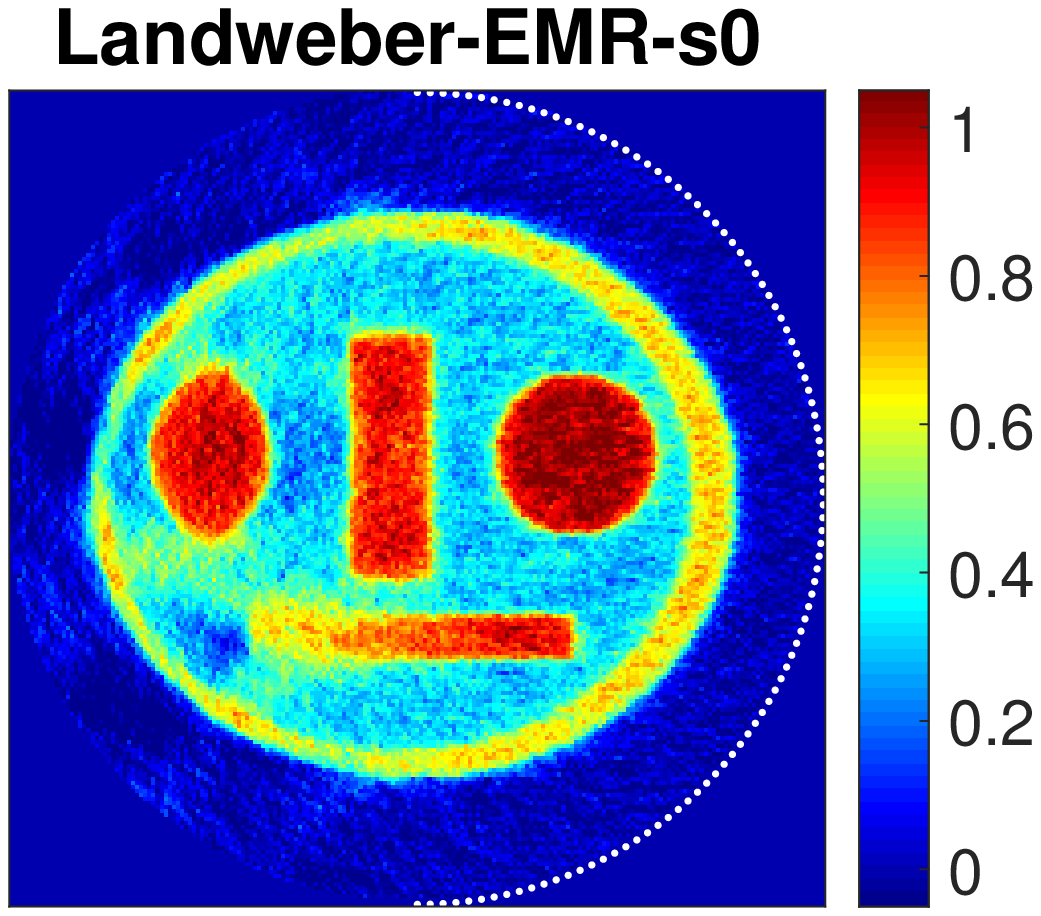}\hspace{0.15cm}
\includegraphics[width=0.28\textwidth]{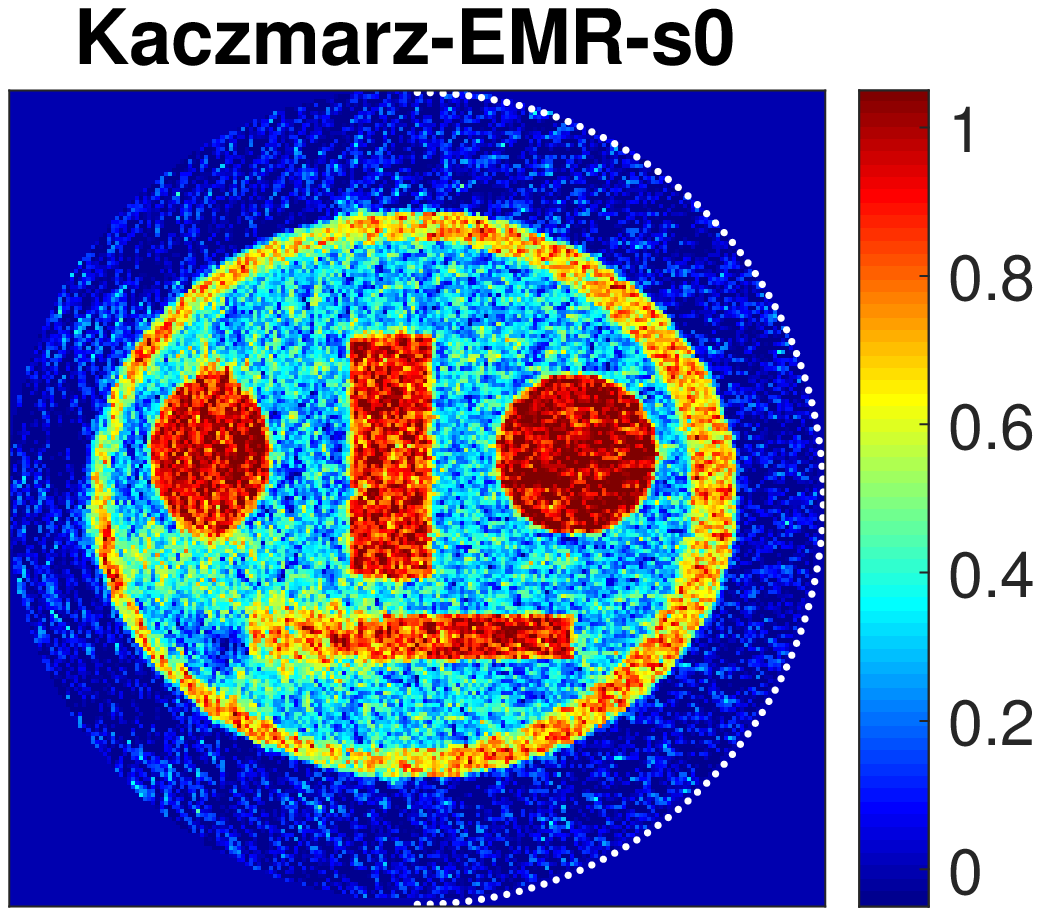}\hspace{0.15cm}
\includegraphics[width=0.28\textwidth]{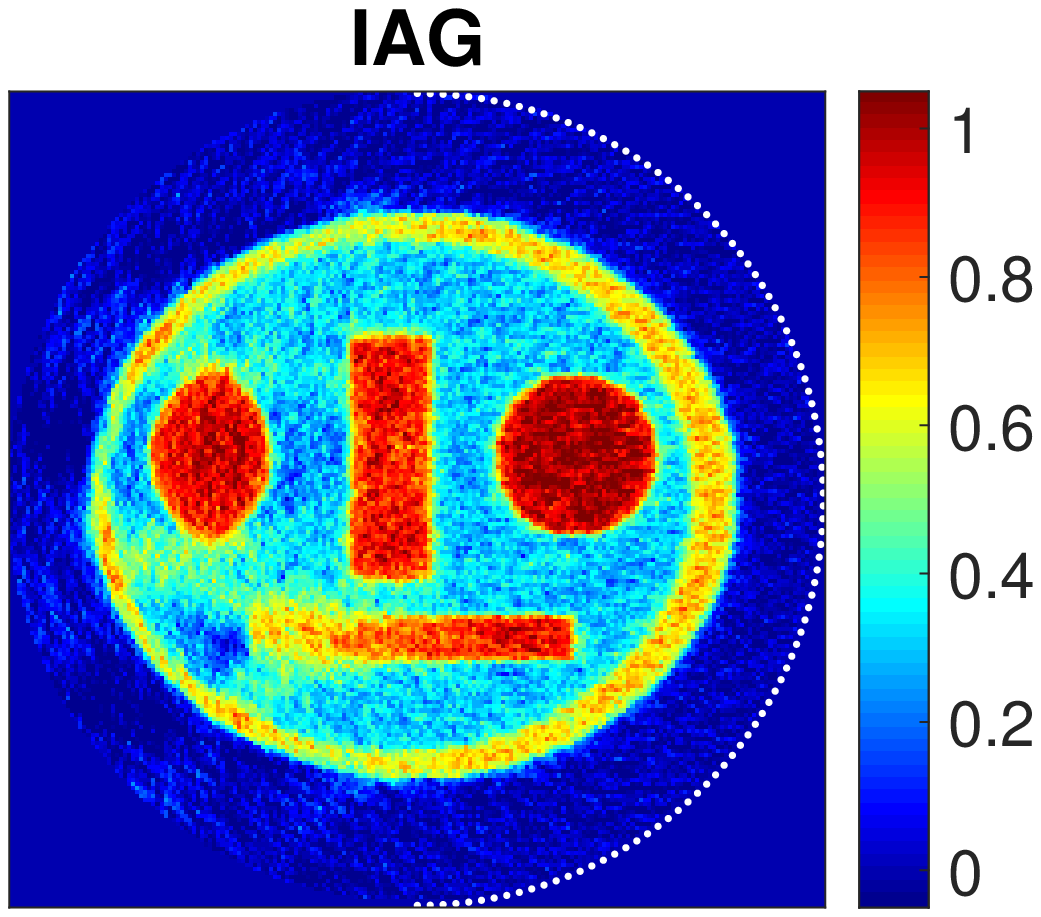} \vspace{0.15cm}\\
\includegraphics[width=0.28\textwidth]{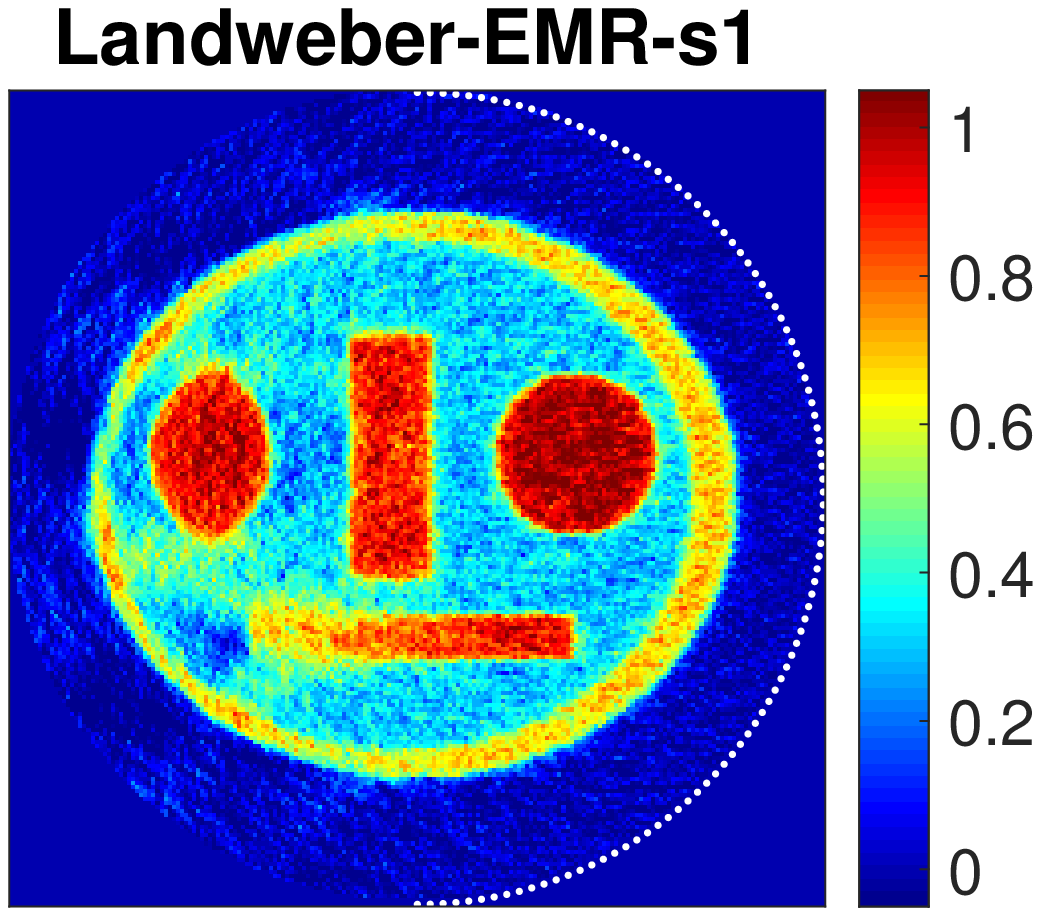}\hspace{0.15cm}
\includegraphics[width=0.28\textwidth]{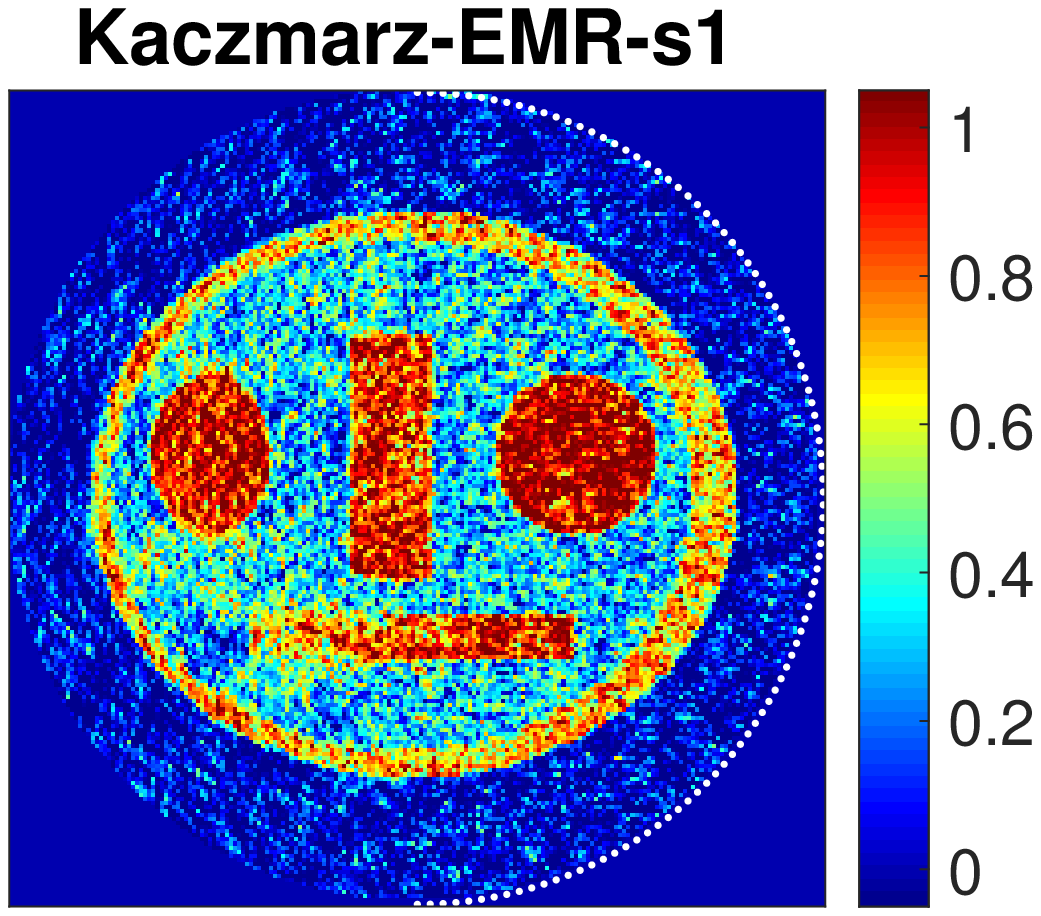}\hspace{0.15cm}
\includegraphics[width=0.28\textwidth]{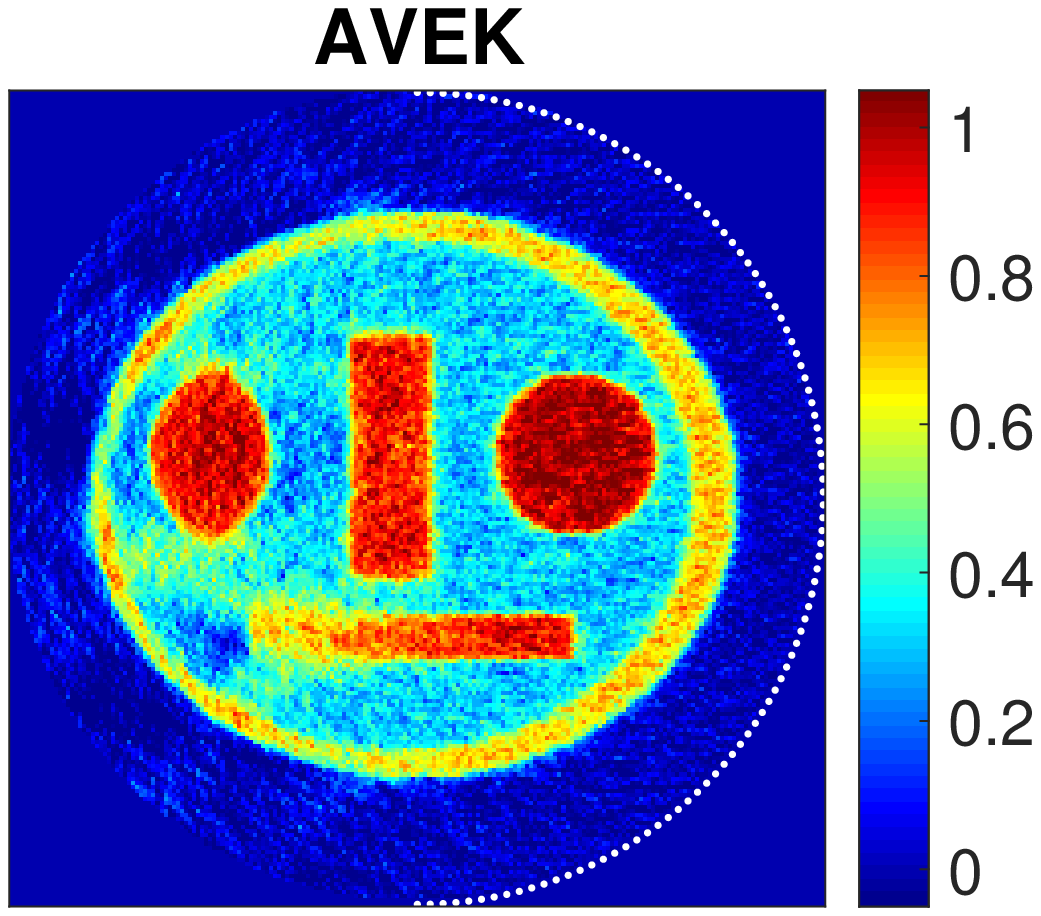} \vspace{0.15cm}
\caption{{Reconstructions by Landweber-EMR ($s = 0$ or $1$), Kaczmarz-EMR ($s = 0$ or $1$), IAG and AVEK from noisy data at the cycles with minimal $L^2$-reconstruction errors.}}
\label{fig:rcNoisy}
\end{figure}

The comparison for noisy data is summarized in Figure~\ref{fig:ccNoisy}. In terms of relative reconstruction errors (which for inverse problems are more important than residuals), the AVEK performs the best, the IAG and the Landweber-EMR ($s = 0$) rank second, followed by the Landweber-EMR ($s = 1$). Unlike in the exact data case, the Kaczmarz-EMR methods are less satisfactory. This indicates that the convergence speed should not be the only concern for iterative methods if they are applied as regularization methods (cf.~also Figure~\ref{fig:varsn}). The minimal relative reconstruction errors are achieved after 10 cycles for the AVEK, after 11 iterations for the Landweber-EMR ($s = 0$), after 14 cycles for the IAG,  after 30 iterations for the  Landweber-EMR ($s = 1$), and after 1 cycle for the Kaczmarz-EMR methods. The reconstructions with minimal reconstruction errors for all methods are shown in Figure~\ref{fig:rcNoisy}.

As we have already noticed, developing appropriate step size strategies can significantly improve the results (see also~\cite{censor1983strong}). Here we have simply used constant and conservative step sizes for the AVEK method. Further, adjusting the skipping  parameters $\rp_k$ can potentially improve and stabilize the AVEK method. A precise comparison of the methods using parameter fine-tuning and implementing adaptive and data-driven choices deserves further investigation; this, however, is beyond the scope of this paper.

\section{Conclusion and outlook}
\label{sec:conclusion}

In this paper we introduced the averaged Kaczmarz (AVEK) method as a paradigm of a new iterative
regularization method. AVEK can be seen as a  hybrid    between   Landweber's and Kaczmarz's method
for solving inverse problems  given  as systems  of equations $\To_i (\signal) =  \data_i$.
As the Kaczmarz method, AVEK requires only solving one forward and one adjoint problem per iteration.
As the Landweber method, it uses information from all equations per update which can have a stabilizing effect.
As main theoretical results, we have shown that the AVEK method converges   weakly in the case of exact data (see Theorem~\ref{thm:exact}),
and presented convergence results for noisy  data (see Theorem~\ref{thm:noisy}).
Note that the convergence as $\delta \to 0$  in Theorem~\ref{thm:noisy}~\ref{noise_b}  assumes
strong convergence in the exact  data case.  It is an open problem if the same conclusion holds under its
weak convergence only.  Another open problem is the strong convergence for exact data  in the general case.
We  conjecture both issues to hold true. {Finally, it is of also interest to investigate the AVEK method \eqref{eq:avek1}-\eqref{eq:avek3} for general convex combinations with weights $\omega_i$ instead of equal weights $\omega_i = 1/n$.}

In Section~\ref{sec:num}, we presented numerical results for the  AVEK method applied to the limited
view problem  for the  circular  Radon, which is relevant for  photoacoustic tomography. For comparison
purpose we also  applied the Landweber and the Kaczmarz method to the same problem.
In the exact data case, the observed convergence speed (number of cycles versus reconstruction error)
of the AVEK turned out to be somewhere  between the Kaczmarz (fastest) and the Landweber method
(slowest).  A similar behavior has been observed in the noisy data case. In this case,
the  minimal reconstruction error for the AVEK is slightly smaller that the one of the Kaczmarz method
and equal to the  Landweber method. The required number of iterations however is less than the one of the
Landweber method. These  initial results are encouraging   and show that the AVEK is a useful iterative method  for tomographic  image reconstruction. Detailed  studies are required in future work on the optimal  selection of parameter
such as the step sizes or the number of partitions.  {The increased stability of AVEK in terms of step sizes is worthy of further theoretical studies.} {Additionally, application of AVEK for non-linear inverse problems is another possible line of future research.}

We see AVEK as the basic member of a new class  of iterative reconstruction method. It shares some
similarities with the incremental gradient method  proposed in the seminal  work~\cite{blatt2007convergent}
(studied for well-posed problems in finite dimensions). Applied to~\eqref{eq:ip}, the  incremental gradient method
reads
\begin{equation} \label{eq:ig}
    \forall k \geq  n \colon \quad
    \signal_{k+1}  = \signal_k  - \frac{s_k}{n} \sum_{\ell = k-n+1}^k    \To_{\ii{\ell}}'(\signal_\ell)^*
    \skl{ \To_{\ii{\ell}}(\signal_\ell) - \data_{\ii{\ell}}} \,.
\end{equation}
Instead of an average  over individual  auxiliary updates, the incremental gradient method
uses an average over the individual  gradients. Studying and analyzing the  incremental gradient method
for inverse problems is an interesting open issue. The incremental gradient method has been generalized
in various directions.  This includes proximal incremental  gradient methods~\cite{bertsekas2011incremental} or
the  averaged stochastic gradient method of~\cite{schmidt2017minimizing}.
Similar extensions for the AVEK (for ill-posed as well as well-posed problems) are interesting lines of
future research.

\section*{Acknowledgment}
H.L. acknowledges support through the National Nature Science Foundation of China 61571008.

\appendix

\section{Deconvolution of sequences and proof of Lemma~\ref{lem:diff}}\label{app:deconv}

The  main aim of this  appendix is to prove Lemma~\ref{lem:diff},
concerning the convergence of the difference of two consecutive iterates
of the AVEK iteration. For that purpose, we will first  derive
auxiliary results concerning deconvolution equations for sequences in Hilbert spaces that
are of interest in its own.

For the following it is helpful to identify any sequence $(a_k)_{k\in \N_0} \in \C^{\N_0}$
with  a formal power  series  $a  = \sum_{k=0}^\infty  a_k  X^k$. Here $X^k \in \C^{\N_0}$ is the sequence
defined by $X^k_k =1 $  and  $X^k_\ell = 0$ for $\ell \neq k$.   For two complex sequences
 $a, b \in \C^{\N_0}$, the  Cauchy product $a \ast b \in \C^{\N_0}$ is defined
by $(a\ast b)_ k \coloneqq  \sum_{j=0}^k a_j b_{k-j}$; see
\cite{henrici74applied}.
We say that $a \in \C^{\N_0}$ is invertible if there is  $b \in \C^{\N_0} $ with
$a \ast b =   (1, 0, \dots )$. We write $b \coloneqq a^{-1}$ and call it the reciprocal
formal power series of $a$, or simply the inverse of $a$.
Moreover one easily verifies  (see \cite{henrici74applied})
that the formal power series $a = \sum_{k=0}^\infty  a_k  X^k$ is invertible if and only if
$a_0 \neq 0$. In this case $b= a^{-1}$ is unique and defined by the recursion
$b_0  = 1 / a_0$ and  $
 b_k   = - \frac{1}{a_0} \sum_{j=0}^{k-1}b_j a_{k-j} $ for $k \geq 1$.
 One further  verifies  that  $\C^{\N_0}$ together with  point-wise addition
and scalar multiplication and the Cauchy product forms an associative algebra.

\subsection{Convolutions in Hilbert spaces}
\label{app:conv}

Throughout this subsection $\X$ denotes an arbitrary Hilbert space. For  $a  \in \C^{\N_0}$
and  $\signal \in  \X^{\N_0}$ define  the  convolution
$\signal  \ast  a  \in \X^{\N_0} $ by
\begin{equation*}
	 \forall k \in {\N_0} \colon \quad (\signal \ast a)_ k \coloneqq  \sum_{j=0}^k \signal_j a_{k-j} \,.
\end{equation*}
One verifies that   $(\signal \ast a) \ast b = \signal \ast (a \ast b)$
for  $a,b  \in \C^{\N_0}$ and  $\signal \in  \X^{\N_0}$.    Moreover, the set of  bounded sequences $\ell^{\infty} ({\N_0}, \X) \coloneqq \sset{ \signal \in \X^{\N_0} \mid \signal_k \text{ bounded}}$  forms  a Banach space together with the uniform norm
$\snorm{x}_\infty \coloneqq \sup \sset{\snorm{\signal_k} \mid k \in {\N_0}}$.
Finally,  $c_0 ({\N_0}, \X) \coloneqq \sset{ \signal \in \X^{\N_0} \mid \lim_{k \to \infty} \signal_k  = 0}$
denotes the space of sequences  in $\X$ converging to zero, {and $\ell^1({\N_0} , \C)\coloneqq \sset{ \signal \in \X^{\N_0} \mid \sum_{k=0}^\infty \abs{\signal_k} < \infty}$ the space of summable sequences.}

\begin{lemma} \label{lem:l1}
Let $b\in \ell^1({\N_0} , \C)$ and define $b^{(m)} \coloneqq (b_0, \dots, b_m, 0, \dots)$.
Then,
\begin{enumerate}
\item\label{lem:l1a} $\forall \signal \in c_0({\N_0} , \X) \colon \signal \ast b^{(m)} \in c_0({\N_0} , \X)$;
\item\label{lem:l1b} $\forall  \signal \in \ell^\infty({\N_0} , \X)    \colon \signal \ast b \in \ell^\infty({\N_0} , \X)
\wedge \lim_{m \to \infty }\snorm{\signal \ast b -  \signal\ast b^{(m)} }_\infty =0$;
\item\label{lem:l1c} $\forall  \signal \in c_0({\N_0} , \X)   \colon \signal \ast b \in c_0({\N_0} , \X)$.
\end{enumerate}
\end{lemma}

\begin{proof}
\ref{lem:l1a} For $k \geq m$ we have  $(\signal \ast b^{(m)})_k  =  \sum_{j=0}^k \signal_j b_{k-j} =
\sum_{j=k-m}^k \signal_j b_{k-j}$. Hence $\signal \ast b^{(m)}$ converges to zero  because $\signal_j b_{k-j}$ does so.

\ref{lem:l1b} For $k \leq m$ we have $(\signal \ast b^{(m)})_k  =  \sum_{j=0}^k \signal_j b_{k-j} =
(\signal \ast b)_k $. For $k > m$ we have
\begin{itemize}
\item $(\signal \ast b)_k - (\signal \ast b^{(m)})_k  =  \sum_{j=0}^k \signal_j b_{k-j} - \sum_{j=k-m}^k \signal_j b_{k-j}
= \sum_{j=0}^{k-m-1} \signal_j b_{k-j}$;
\item $\snorm{(\signal \ast b)_k - (\signal \ast b^{(m)})_k}  \leq  \norm{x}_\infty  \sum_{j=0}^{k-m-1} \abs{b_{k-j}}
\leq  \snorm{x}_\infty  \sum_{j=m+1}^\infty \abs{b_{j}}$;
\item $\sum_{j=m+1}^\infty \abs{b_{j}} \to 0$ (because $\sum_{k \in {\N_0}} \abs{b_k}  < \infty$).
\end{itemize}
We conclude that $\snorm{(\signal \ast b)  - (\signal \ast b^{(m)}) }_\infty\leq \snorm{x}_\infty
\sum_{j=m+1}^{\infty} \abs{b_{j}}  \to 0$.

\ref{lem:l1c} Follows from \ref{lem:l1a}, \ref{lem:l1b} and  the closedness of $c_0 ({\N_0}, \X) $
in $\ell^{\infty} ({\N_0}, \X)$.
\end{proof}

As an application of Lemma~\ref{lem:l1} we
can show the following result, which is  the main ingredient for the
proof of Lemma~\ref{lem:diff}.

\begin{proposition}[A deconvolution problem]\label{prop:deconv}
For any sequence $\dsignal = (\dsignal_k)_{k =1}^\infty$ in $\X^{\N_0}$ and  any
$n \in {\N_0}$, the following implication
holds true:
\begin{equation*}
\lim_{k \to \infty} \sum_{j = 1}^n j \dsignal_{k-n+j} = 0
\implies  \lim_{k \to \infty} \dsignal_k = 0 \,.
\end{equation*}
\end{proposition}

\begin{proof}
Set $a \coloneqq (n,n-1,\dots, 1,0, \dots)$ and suppose that $(\dsignal \ast a)_k  \to 0$  as $k \to \infty$.
We have to verify that  $\dsignal_k  \to 0$  as $k \to \infty$, which is divided in several steps.

 \begin{itemize}[wide]
 \item Step 1: All zeros of the polynomial $p \colon  \C \to \C \colon z \mapsto n +  (n-1) z + \cdots z^{n-1}$    are contained in $\set{z \in \C \mid \snorm{z} > 1}$.

Because $p(0) \neq 0$, in order to verify Step 1,  it is sufficient to show that all zeros of $p(1/z)$
are  contained in the  unit disc $B_1(0) = \set{z \in \C \mid \snorm{z} < 1}$.
Hence it is sufficient to show that   the polynomial $q(z) \coloneqq z^{n-1} p(1/z) \coloneqq n z^{n-1} +  (n-1) z^{n-2} + \cdots + 1$ has all zeros in  $B_1(0)$. Further note that $q(z) =  Q'(z)$, where $Q(z)\coloneqq  z^{n} +  z^{n-1} + \cdots + z$ has
the form $Q(z) = z \frac{z^n -1}{z-1}$.
Consequently, $\set{0} \cup \set{z \in \C \mid  z^{n} = 1 \wedge z \neq 1}$ is the set of zeros of $Q$.    The Gauss-Lukas theorem (see \cite[Theorem (6,1)]{marden66geometry}) states that all critical points
of a non-constant polynomial $f$ are contained in the convex hull $H$ of the set of zeros of $f$.
 If the zeros of $f$ are not collinear,  then no critical point lies on $\partial H$ unless it is a multiple zero of $f$.     Note that all zeros of $Q$ are simple, not collinear and contained
in $\overline{B_{1}(0)}$. According the Gauss-Lukas theorem all zeros of $q=Q'$ are contained in $B_{1}(0)$.    Consequently all zeros of $p$ are indeed contained in  $\sset{z \in \C \mid \snorm{z} > 1}$.

\item   Step 2: We have  $a^{-1} \in \ell^1(\N_0, \C)$.

All zeros of $p(z)$ are outside of $B_{1+\eps} (0)$ for some $\eps >0$ and
therefore   $1/p(z)$ is analytic in $B_{1+\eps} (0)$ and can be expanded in a power
series $ 1/p(z)  = \sum_{k \in \N_0} b_k z^k$. The radius of convergence is at least $1+\eps$
(as the radius of convergence  of  a function $f$ is the radius of the largest  disc where $f$ or an analytic
continuation of $f$ is analytic; see for example \cite[Theorem 3.3a]{henrici74applied}.)
We have
\begin{equation*}
1 = p(z)  \frac{1}{p(z)} = \sum_{j=0}^{n-1} a_j  z^j \sum_{k \in \N_0} b_k z^k
=   \sum_{k \in \N_0} (a \ast b)_k z^k \,.
\end{equation*}
Hence $a \ast b  = (1,0,0,\dots)$ and  $a^{-1} = b  \in  \ell^1(\N_0, \C)$.

 \item
Step 3: We are now ready to complete the proof. According to the assumption,
we have $\dsignal \ast a \in c_0(\N_0, \X)$. According to Step 2, we have  $a^{-1} \in \ell^1(\N_0, \C)$.
 Therefore Lemma~\ref{lem:l1}~\ref{lem:l1c}
implies that $\dsignal = (\dsignal \ast a) \ast a^{-1}  \in c_0(\N_0, \X)$.
\end{itemize}
\end{proof}

\subsection{Application to the AVEK iteration}

Now let  $\signal_k$ {be} defined by~\eqref{eq:avek-e1}, let  $\signal^* \in B_{\rho}(\signal_0)$ be an arbitrary solution to~\eqref{eq:ip} and assume that~\eqref{eq:tcc} and \eqref{eq:stepsz1} hold true.
We introduce the auxiliary sequences $\dsignal_k  \coloneqq \signal_{k+1} - \signal_k$,
$\zsignal_k  \coloneqq  \frac{1}{n} \sum_{j=1}^n j  \signal_{k-n+j}$  and
$\res_k   \coloneqq   \To_{\ii{k}}'(\signal_k)^*
\skl{ \To_{\ii{k}}(\signal_k) - \data_{\ii{k}}}$.
Here $\dsignal_k$ are the differences between two consecutive iterations that we show to converge to zero,
 $\zsignal_k$ will be required in the subsequent analysis, and $\res_k$ are the residuals.

\begin{lemma} \label{lem:diff1} \mbox{}
\begin{enumerate}
\item \label{lem:diff1a} $\lim_{k \to \infty}  \signal_{k+1} - \sfrac{1}{n} \sum_{l=k-n+1}^k \signal_\ell  = 0$;
\item \label{lem:diff1b} $\zsignal_{k+1} -  \zsignal_{k} = \signal_{k+1} - \sfrac{1}{n} \sum_{\ell=k-n+1}^k \signal_\ell$;
\item \label{lem:diff1c} $\lim_{k \to \infty} \zsignal_{k+1} -  \zsignal_{k}
= \sfrac{1}{n}  \lim_{k \to \infty}   \sum_{j=1}^n j  \dsignal_{k-n+j} =0$.

\end{enumerate}
\end{lemma}

\begin{proof}
\ref{lem:diff1a}:  By  the definition of $\signal_k$, $\res_k$
we have  $\signal_{k+1}  = \frac{1}{n} \sum_{\ell = k-n+1}^k  \signal_\ell  - s_\ell  \res_\ell$.
 Therefore
\begin{equation*}
 \Bigl\lVert \signal_{k+1} - \frac{1}{n} \sum_{l=k-n+1}^k \signal_\ell \Bigr\rVert
 = \Bigl\lVert \frac{1}{n} \sum_{\ell = k-n+1}^k  s_\ell \res_\ell \Bigr\rVert
 \leq \frac{1}{n}   \sum_{\ell = k-n+1}^k  s_\ell \snorm{\res_\ell} \,.
\end{equation*}
As we already know that $s_\ell\snorm{\res_\ell} \to 0$, the claim follows.

\ref{lem:diff1b}: We have
\begin{align*}
\zsignal_{k+1} -  \zsignal_{k}
&
= \frac{1}{n} \sum_{j=1}^n j  \signal_{k-n+j+1} -\frac{1}{n} \sum_{j=1}^n j  \signal_{k-n+j}
\\&=  \signal_{n+1} + \frac{1}{n} \sum_{j=1}^{n-1} j  \signal_{k-n+j+1} - \frac{1}{n} \sum_{j=2}^n j  \signal_{k-n+j} -
\frac{1}{n} \signal_{k-n+1}
\\&=  \signal_{n+1} + \frac{1}{n} \sum_{j=2}^{n} (j-1)  \signal_{k-n+j} - \frac{1}{n} \sum_{j=2}^n j  \signal_{k-n+j} -
\frac{1}{n} \signal_{k-n+1}
\\&=  \signal_{n+1} - \frac{1}{n} \sum_{j=2}^{n}  \signal_{k-n+j} - \frac{1}{n} \signal_{k-n+1}
=  \signal_{n+1} - \frac{1}{n} \sum_{j=1}^{n}  \signal_{k-n+j}  \,.
\end{align*}

\ref{lem:diff1c}: Follows from \ref{lem:diff1a}, \ref{lem:diff1b}.
\end{proof}

\subsubsection*{Proof of Lemma~\ref{lem:diff}}

Lemma~\ref{lem:diff} now is an immediate consequence of
Lemma~\ref{lem:diff1}  and Proposition~\ref{prop:deconv}.
In fact, from  Lemma~\ref{lem:diff1}~\ref{lem:diff1c} we know that
$\lim_{k \to \infty}   \sum_{j=1}^n j   \dsignal_{k-n+j} =0$ for $k \to \infty$.
Then the assertion follows from Proposition~\ref{prop:deconv}.

\end{document}